\documentclass[reqno]{amsart}
\usepackage[left=2.7cm,right=2.7cm,top=3.5cm,bottom=3cm]{geometry}
\usepackage[english]{babel}
\usepackage{pifont}
\usepackage[latin1]{inputenc}
\usepackage{amsmath}
\usepackage{amsfonts}
\usepackage{amsthm}
\usepackage{amscd}
\usepackage{amssymb}
\usepackage[all]{xy}
\usepackage{graphicx}
\usepackage[usenames,dvipsnames]{color}
\usepackage{mathrsfs}
\usepackage{comment}
\usepackage{caption}
\usepackage{braket}
\usepackage{mathrsfs}
\usepackage[colorlinks = true,linkcolor = black, urlcolor = blue, bookmarksopen = true ]{hyperref}

\theoremstyle{plain}
\newtheorem{thm}{Theorem}[section]
\newtheorem{cor}[thm]{Corollary}
\newtheorem{lem}[thm]{Lemma}
\newtheorem{prop}[thm]{Proposition}

\newtheorem{cnj}[thm]{Conjecture}

\newtheorem{innercustomgeneric}{\customgenericname}
\providecommand{\customgenericname}{}
\newcommand{\newcustomtheorem}[2]{\newenvironment{#1}[1]
{\renewcommand\customgenericname{#2}
\renewcommand\theinnercustomgeneric{##1}
\innercustomgeneric }
{\endinnercustomgeneric} }

\newcustomtheorem{customthm}{Theorem}

\theoremstyle{definition}
\newtheorem{dfn}[thm]{Definition}
\newtheorem{eg}[thm]{Example}
\newtheorem{eg+}[thm]{Examples}
\newtheorem{rmk}[thm]{Remark}
\newtheorem{rmks}[thm]{Remarks}

\newenvironment{prf}{\begin{proof}[Proof]}{\end{proof}}
\newenvironment{skprf}{\begin{proof}[Sketch of proof]}{\end{proof}}

\newcommand{\field}[1]{\mathbb{#1}}
\newcommand{\A}{\field{A}}
\newcommand{\C}{\field{C}}

\newcommand{\F}{\field{F}}

\newcommand{\N}{\field{N}}
\newcommand{\p}{\field{P}}
\newcommand{\Q}{\field{Q}}
\newcommand{\R}{\field{R}}

\newcommand{\Z}{\field{Z}}

\newcommand{\adl}{{\bf A}}

\newcommand{\za}{\widehat{\Z}}
\newcommand{\da}{\widehat D}

\newcommand{\sa}{\widehat S}
\newcommand{\xa}{\widehat X}
\newcommand{\ya}{\widehat Y}
\newcommand{\dap}{\da_\gotp}
\newcommand{\wds}{\widehat{D^*}}
\newcommand{\hpi}{\widehat\pi}
\newcommand{\cpr}{\widehat{\calp}}

\newcommand{\sqf}{\rm Sqf}

\DeclareMathOperator{\dlib}{d}
\newcommand{\dsbk}{\dlib_{\rm Bk}}
\newcommand{\dsbkp}{\dsbk^+}
\newcommand{\dsbkm}{\dsbk^-}

\DeclareMathOperator{\rank}{rank}

\DeclareMathOperator{\val}{val}
\DeclareMathOperator{\ev}{ev}

\DeclareMathOperator{\GL}{GL}

\DeclareMathOperator{\Hom}{Hom}
\DeclareMathOperator{\Homc}{Hom_{cont}}

\DeclareMathOperator{\b1}{{\mathbf 1}}
\DeclareMathOperator{\supp}{supp}

\DeclareMathOperator{\Frac}{Frac}

\DeclareMathOperator{\car}{char}

\DeclareMathOperator{\cond}{cond}

\newcommand{\liminv}{\displaystyle \lim_{\leftarrow}}

\newcommand{\calc}{\mathcal C}
\newcommand{\cald}{\mathcal D}

\newcommand{\calh}{\mathcal H}
\newcommand{\cali}{\mathcal I}
\newcommand{\calj}{\mathcal J}
\newcommand{\calk}{\mathcal K}
\newcommand{\call}{\mathcal L}
\newcommand{\calm}{\mathcal M}

\newcommand{\calo}{\mathcal O}
\newcommand{\calp}{\mathcal P}

\newcommand{\cals}{\mathcal S}

\newcommand{\calu}{\mathcal U}

\newcommand{\gota}{\mathfrak a}
\newcommand{\gotb}{\mathfrak b}

\newcommand{\gotm}{\mathfrak m}
\newcommand{\gotn}{\mathfrak n}
\newcommand{\gotp}{\mathfrak p}
\newcommand{\gotq}{\mathfrak q}

\newcommand{\mmu}{\boldsymbol\mu}
\newcommand{\bff}{{\rm\bf f}}

\renewcommand{\ge}{\geqslant}
\renewcommand{\le}{\leqslant}

\begin{document}

\title[Profinite Bateman--Horn]{Procounting measures and the Bateman--Horn conjecture}
 
\author[Demangos]{Luca Demangos}
\address{Department of Pure Mathematics, Xi'an Jiaotong Liverpool University, 111 Ren'ai road, 215123, Suzhou, China}
\email{Luca.Demangos@xjtlu.edu.cn}

\author[Longhi]{Ignazio Longhi}
\address{Dipartimento di Matematica dell'Universit\`a di Torino, via Carlo Alberto 10, 10123, Torino, Italy}
\email{ignazio.longhi@unito.it}

\author[Saettone]{Francesco Maria Saettone} 
\address{Department of Mathematics, Weizmann Institute of Science, 234 Herzl street, 76100, Rehovot, Israel}
\email{francesco.saettone@weizmann.ac.il}

\maketitle

\begin{abstract} Let $D$ be the ring of $S$-integers in a global field and $\da$ its profinite completion. We propose a profinite version of the  Bateman--Horn conjecture over $D$ and  provide a first comparison with the classical one and its generalizations. Our approach is based on the new notion of procounting measure: a distribution on $\da$ which should be seen as a profinite analogue of the counting function for a subset of $\R$. This allows us to deal with subsets of $\da$ having Haar measure $0$ (corresponding to density zero in $\R$) \end{abstract}

\tableofcontents

\section{Introduction} 

\medskip
\noindent\textbf{The classical conjecture and its kin.}\quad
Let $\calp$ denote the set of prime numbers and $\calp^k$ the product of $k$ copies of $\calp$ into $\Z^k$. The classical Bateman-Horn conjecture (\cite{1bh}, see also \cite{bh} for a recent introduction), states the following.

\begin{cnj}[Bateman-Horn] \label{cnj:BH} Let $f_1,\dots,f_k\in\Z[x]$ be distinct nonconstant irreducible polynomials with positive leading coefficients and $f$ their product. Let
$$\omega_{f}(p):=\big\vert\{ a\in \Z/p\Z : f(a)=0 \bmod p\}\big\vert$$
and assume $\omega_f(p)<p$ for every prime $p$. By denoting
\begin{equation} \label{e:cffBH} C(f):=\prod_{p\in\calp}\left(1-\frac{1}{p}\right)^{-k}\left(1-\frac{\omega_{f}(p)}{p}\right) \end{equation}
it follows
\begin{equation} \label{e:bhclssc}
\big|\bff^{-1}(\calp^k)\cap [0,r]\big|\;\approx\;\dfrac{C(f)}{\prod_{i=1}^k\deg(f_{i})}\int_2^r\dfrac{dt}{(\log(t))^k}
\end{equation}
where $\bff\colon\Z\rightarrow\Z^k$ is the map $a\mapsto(f_1(a),\dots,f_k(a))$ and $\approx$  denotes the same asymptotic behaviour as $r\to+\infty$. \end{cnj}

This subsumes and refines a number of previous claims (among which we mention qualitative statements like the Bunyakovsky conjecture and Schinzel's hypothesis H). The problem of proving Bateman-Horn is widely open: the only completely solved case is for $k=1$ and $\deg f=1$, which amounts to  the prime number theorem for arithmetic progressions.

Nonetheless, there has recently been a tremendous activity around this conjecture, either with variations in the function fields setting as \cite{kc}, \cite{bp}, \cite{pol}, \cite{bs}, \cite{bs2}, \cite{ae} culminating with the remarkable  papers of Sawin and Shusterman \cite{sstwin}, \cite{sslandau} (see also \cite{kow} for a survey), or with a striking combination of algebraic geometry and probabilistic ideas as \cite{bst} and \cite{ss}. A result on the Bateman--Horn conjecture on average over a single Kummer polynomial can also be found in \cite{zhou} and \cite{br}. Moreover, Bateman--Horn-type problems in several variables have also been studied by  the multivariable heuristics in \cite[Appendix~A]{dsofos} and the homogeneous dynamics approach of \cite{ktw}. Other relevant, important results can be found, for instance, in \cite{friediw}, \cite{iw} and \cite{may}, where more classical analytic techniques are exploited.

\medskip
\noindent\textbf{The profinite conjecture.}\quad
In this work we propose a new avatar of Conjecture \ref{cnj:BH}, formulated by means of measures attached to {\em profinite} sets. 

The heuristics can be traced to the qualitative conjectures, which, in a nutshell, claim that the only obstructions to prime values of irreducible polynomials come from congruences (and the sign condition). The profinite ring $\za$ takes into account all congruences in a single swoop. In order to get a quantitative statement like \eqref{e:bhclssc}, one may then be tempted, in the light of \cite{dl}, to use $\mu_{\za}$, the Haar measure on $\za$, which, however, does not ``see'' the primes (meaning that the closure of $\calp$ in $\za$ is a set of measure $0$ - note that this corresponds\footnote{ Actually, it is stronger: see e.g. \cite[Corollary 4.14]{dl}.} 
to the fact that $\calp$ has density $0$, i.e., the counting function $\pi(x)$ grows much more slowly than $x$). Thus we introduce the new idea of {\em procounting measure}: given $S\subseteq\za$, we attach to it a probability measure $\mu_S$ on $\za$ by lifting counting measures on $\Z/n\Z$ and taking a limit (when it exists - as we shall see, there are instances where $\mu_S$ is not defined).

In order to state a first version of our profinite Bateman-Horn, we introduce a little notation. The function $\bff$ of \eqref{e:bhclssc} is extended to $\za\rightarrow\za^k$. Let $\za^*$ be the group of units in $\za$ and $\hpi_n\colon\za\rightarrow\Z/n\Z$ the canonical projection. Consider the normalized counting measures
$$\mu_{\bff^{-1}(\calp^k),n}:=\frac{1}{|\hpi_n(\bff|_\N^{-1}(\calp^k))|}\sum_{x\in\hpi_n(\bff|_\N^{-1}(\calp^k))}\delta_x$$
on $\Z/n\Z$, where $\delta_x$ is the Dirac delta at the point $x$. 

\begin{cnj}[profinite Bateman-Horn for $\Z$] \label{cnj:pBH} Let $f_1,\dots,f_k\in\Z[x]$ be distinct irreducible polynomials with positive leading coefficients and $\bff\colon\za\rightarrow\za^k$ the map $a\mapsto(f_1(a),\dots,f_k(a))$. Assume  $\bff\big(\za\big)\cap(\za^*)^k\neq\emptyset$. Then one has
\begin{equation} \label{e:prfBHZ}  \lim_{n\to 0}\widetilde\mu_{\bff^{-1}(\calp^k),n}=\mu_{\bff^{-1}((\za^*)^k)}\ . \end{equation}
\end{cnj}

\noindent This requires some explanations. The hypothesis on $\bff(\za)$ is simply a synthetic reformulation of the Bunyakovsky condition $\omega_f(p)<p$ for all $p$ (Lemma \ref{l:bunyakcnd}). On the left-hand side of \eqref{e:prfBHZ}, the terms $\widetilde\mu_{\bff^{-1}(\calp^k),n}$ are lifts of the $\mu_{\bff^{-1}(\calp^k),n}$'s to measures on $\za$, while the limit is taken in the weak-* topology, letting $n$ vary among positive integers ordered by divisibility. By definition, this limit would be the procounting measure attached to $\bff^{-1}(\calp^k)$: note that its existence is part of our conjecture (but we can prove it in the case $k=1$, $\deg(f_1)=1$). On the other hand, we are going to show that the procounting measure on the right-hand side is indeed well-defined (as a consequence of Proposition \ref{p:invimgelr}).

\begin{rmk} \label{r:intro} Equidistribution statements are usually expressed by considering the limit of a sequence of counting measures attached to finite subsets of an ambient space $X$, a formalism at first sight very similar to the one in \eqref{e:prfBHZ}. We note a couple of subtle differences. Our ambient space is $\za$, while the finite subsets are in the quotients $\Z/n\Z$. As a consequence, our indexes vary among positive integers (or, better, non-zero ideals) ordered by divisibility, and not in $\N$ with its natural order. \end{rmk}

The profinite formulation does not dispel the enigma of Bateman--Horn, but changes its shape: the problem is no longer to count prime values in intervals, but to understand their shadows in all finite quotients. On the left-hand sides we have replaced the mysterious counting function of \eqref{e:bhclssc} with the  equally mysterious procounting measure of \eqref{e:prfBHZ}. As for the right-hand sides, they are closer than it might appear at first sight: indeed, the measure $\mu_{\bff^{-1}((\za^*)^k)}$ is a product of local factors corresponding  exactly to those in \eqref{e:cffBH}. We will provide some explicit examples in Section~\ref{s:da-eg}; now we just mention that, for instance, in the linear case $f(x)=a+bx$ with $a$ and $b$ coprime, the profinite local correction is $b/\varphi(b)=C(f)$. This is one of the main consistencies between the two conjectures: although,  to the best of our current understanding,  our version cannot see the archimedean growth term\footnote{ Note that \eqref{e:bhclssc} counts objects inside $\R$, while \eqref{e:prfBHZ} estimates the size of a subject of $\za$. The natural way to combine the two ambient spaces $\R$ and $\za$ is, perhaps, to use adele classes: we hope to explore this in a future paper.}, it still retains the finite local factors of the classical Bateman--Horn constant.

We also note that our conjecture is strong enough to imply Schinzel's hypothesis H (Proposition \ref{p:pbhsh}).\\

More important, the profinite formalism extends naturally (and uniformly) to a much more general setting, replacing $\Z$ with $D$, the ring of $S$-integers in any global field $F$, and $\calp$ with $\calp(D)$, the set of prime elements in $D$. For technical reasons, in this paper we shall need the minor limitation of either $\car(F)\neq 2$ or $D^*$ finite. Under this restriction, we formulate Conjecture \ref{cnj:pBHD}, a profinite avatar of Bateman--Horn in a generality never attempted, to our knowledge, in the literature. 

In  view of the impressive progress on function field versions of Bateman--Horn in recent years, the positive characteristic case of Conjecture \ref{cnj:pBHD}  is expected to be of particular interest.\\

A final remark: the idea of looking at $\za$ (or $\da$) in order to estimate the ``size'' of subsets of $\Z$ (or $D$) was the philosophy underlying \cite{dl}, of which this paper is a sequel, and a substantial expansion and refinement. More precisely, in \cite{dl} the basic question was, given $X\subseteq D$, to determine its closure $\xa$ in $\da$; the ``size'' of $X$ would then be computed as $\mu_{\da}(\xa)$, a value often coinciding with the density of $X$ in $D$.\footnote{For $D=\Z$, densities are usually defined by some method of counting objects along the real line - that is, they should be thought of as approximations of counting functions like the one on the left-hand side of \eqref{e:bhclssc}.} Here we develop the machinery to deal with the case $\mu_{\da}(\xa)=0$, still keeping the principle that the interesting information should be contained in $\xa$. As a matter of fact, it turns out that the procounting measure $\mu_X$, if it exists, is the same as $\mu_{\xa}$ (Lemma \ref{l:muS-chs}) and is concentrated in the accumulation points (see \S\ref{sss:isltpnt}). \medskip

\medskip
\noindent\textbf{Structure and main results.}\, We now give a more detailed description of the structure and results of our paper. In this introduction, for the sake of simplicity, we reduce the discussion mostly to the case $D=\Z$.

\subsubsection*{Section \ref{s:dmps}} We start by recalling, in broad generality, the theory of measures and distributions on a profinite set $X$. Such objects appear in Iwasawa theory (see \cite[\S7.1]{msd}) as the habitat where $p$-adic $L$-functions live: for the moment there is only a vague connection between the latter and our constructions (see  Remark \ref{r:pLf}), but the hope of further developing this link in a future work motivates us to take, as long as possible, coefficients in a general topological ring, even if we will mostly need the cases of $\R$ and $\C$. Actually, our primary goal in this paper is a quantitative theory beyond \cite{dl} and we achieve it by means of $\R$-valued measures.

Counting distributions on finite sets are indeed the fountainhead of measure theory; taking their limits (in a suitable sense) we introduce our crucial new tool, the {\em procounting distributions} attached to subsets $S$ of $X$ (Definitions \ref{d:muS1} and \ref{d:muS2}). We discuss some of their basic properties and subtleties (such as conditions for the existence, dependence on choices, relation with Fourier analysis when $X$ is a group) and hint at how they can be put in a hierarchy reminiscent of the ``little-o'' notation for functions. Our main result in this section concerns what happens in (infinite) products and is summarized in the next theorem (a slightly simplified combination of Theorem \ref{t:dstrprd} and Proposition \ref{p:prddstr2}).

\begin{customthm}{A} Let $X=\prod_n Y_n$ be a product of profinite spaces $Y_n$ and consider a subset $S=\prod_nT_n$, where $T_n\subset Y_n$ for all $n$. If the procounting distribution exists for every $T_n$, then it exists also for $S$ and one has
\begin{equation} \label{e:thmA} \mu_S=\bigotimes_n\mu_{T_n}\,, \end{equation}
meaning that for any compact open $U=\prod_nU_n$, with $U_n\subseteq Y_n$, one has $\mu_S(U)=\prod_n\mu_{T_n}(U_n)$. \end{customthm}

Note that this applies in situations of arithmetic interest, such as $\za=\prod_p\Z_p$. The decomposition \eqref{e:thmA} appears for example in the right-hand side of \eqref{e:prfBHZ} and is the reason why the latter recovers the local factors of \eqref{e:cffBH}. Actually, replacing a real product like \eqref{e:cffBH} with a formal one like \eqref{e:thmA} has the advantage that we can ignore convergence issues (the convergence of \eqref{e:cffBH} is a rather delicate matter: see \cite[Section 5]{1bh} for a detailed treatment).\\

Finally, we would like to express our hope that procounting measures will find many more applications. For example, it seems reasonable to expect that our method can be adapted to problems somehow similar to the one examined in this paper: the Artin conjecture on primitive roots and its cognate the Lang-Trotter conjecture immediately come to mind.

\begin{rmk} Profinite sets can often be realized as the boundary of some tree. The theory of measures on such has a long history (see for instance \cite{ccs} and its citation orbit), but, to our knowledge, the existent literature has almost empty intersection with our techniques and results. \end{rmk}

\subsubsection*{Section~\ref{s:da}} Here arithmetic starts playing a role: we use the structure of $\za$ as an infinite product to identify (following \cite{dl}) a special class of subsets, the {\em openly Eulerian} ones: they are those $X\subseteq\za^n$ ($n\ge1$) such that $\xa$ decomposes as a product of compact open subsets of $\Z_p^n$, for all $p$; a typical example is $\za^*=\prod_p\Z_p^*$. 

By the results of Section \ref{s:dmps}, if $X$ is openly Eulerian then the procounting measure $\mu_X$ does exist and it can be decomposed into a product of the form \eqref{e:thmA}; moreover, each of the factors is obtained by restricting the Haar measure of $\Z_p$ and multiplying by a scalar, the local coefficient at $p$ (Theorem \ref{t:opnelr}). As the computations in Example \ref{eg:sqrfr} will show, divergence in the infinite product of these local coefficients corresponds to the case $\mu_{\za}(\xa)=0$. Finally, openly Eulerian sets are well-behaved with respect to polynomial images and preimages. This fully explains our previous claims about the right-hand side of \eqref{e:prfBHZ}. 
 
The relation with primes arises via a relatively elementary observation, namely the fact that the closure of $\calp$ in $\za$ is
\begin{equation} \label{e:prmZ} \cpr=\calp \sqcup \za^* \end{equation} 
(this follows from Dirichlet's theorem on primes in arithmetic progressions and is actually equivalent to it, see e.g.~\cite[Theorem~3.27]{lms} for a proof). By \eqref{e:prmZ}, primes provide the ur-example of a set which is not openly Eulerian, but is close to being so: the non-isolated part of $\cpr$ is precisely the openly Eulerian set $\za^*$, and this topological fact strongly hints at the equality $\mu_\calp=\mu_{\za^*}$.

In order to justify this guess, we introduce the notion of {\em close pairs} (Definition \ref{d:clpr}; roughly, it consists of two subsets of a profinite ambient space having ``almost the same image" in the finite quotients). If $S$ and $T$ form a close pair, it follows $\mu_S=\mu_T$. The equality \eqref{e:prmZ} leads to the following statement (which recapitulates Theorem \ref{t:clprim} and Proposition \ref{p:measpol} in the case $D=\Z$). 

\begin{customthm}{B} \label{t:B} For $f_1,\dots,f_k\in\za[x]$, let $\bff\colon\za\rightarrow\za^k$ be the map $a\mapsto\big(f_1(a),\dots,f_k(a)\big)$. Assume that $\bff(\za)\cap(\za^*)^k$ is nonempty. Then both
\[ \big(\calp\ ,\ \za^*\big) \qquad \text{and} \qquad \big(\bff^{-1}(\cpr^k)\ ,\ \bff^{-1}((\za^{*})^k)\big) \]
form a close pair, so that
\[ \mu_\calp=\mu_{\za^*} \qquad \text{and} \qquad \mu_{\bff^{-1}(\cpr^k)}=\mu_{\bff^{-1}((\za^{*})^k)}\ . \]
\end{customthm} 

\noindent This is the basic model behind Conjecture~\ref{cnj:pBH}.\\

In the last part of Section \ref{s:da} we provide a more explicit description of Conjecture~\ref{cnj:pBH} in the cases corresponding to Dirichlet's theorem, the twin primes conjecture and Landau's conjecture, so to better illustrate its coherence with the classical statements. We conclude with a simple numerical experiment for our profinite twin primes conjecture, which gives us an empirical reason to be cautiously optimistic.

\begin{rmk} Eulerian sets were introduced in \cite{dl} as a profinite approach to local-to-global principles. Such a principle for primes is implicit in \cite[Theorem 1]{cm} and in following works like \cite{yam1}: we will not discuss this any further in the present paper, but it seems reasonable to relate our heuristics with those of \cite[formula (1.1)]{cm} and compare the factors $\mu_p(\mathbf s)$ of {\em loc.cit.} to our procounting measures on $\Z_p^*$. More qualitatively, we mention the ``prime Hasse principle" of \cite[Principle 1.1]{holdr}, which is satisfied by a given system of homogeneous Diophantine equations if the existence of a solution with all entries in $\calp$ is guaranteed by the existence of a solution with entries in $\Z_p^*$, for all $p$, and in the positive real numbers.\footnote{Essentially the same conditions had already appeared in \cite[Theorem 1.1]{yam2} as part of the hypotheses necessary for a lower bound on the growth of the number of prime solutions. It might be worth observing that one needs infinitely many prime solutions to force the existence of a solution in $\Z_p^*$ for all $p$.} One can argue that, like for our Conjecture \ref{cnj:pBH}, also the intuition behind such a principle stems from a model whose deep roots are summarized by \eqref{e:prmZ}. \end{rmk}

\subsubsection*{Section~\ref{s:cmpr}} It is natural to ask about the logical relation between Conjectures \ref{cnj:BH} and \ref{cnj:pBH} (as well as between their extensions to more general $D$).

In trying to answer this question, we take as starting point the connections between densities in $\Z$ and the Haar measure on $\za$, as studied in \cite{dl}. In the current setting, consider two subsets $X\subseteq Y$ of $\Z$ and their respective counting functions $c_X(r)$, $c_Y(r)$, where $r\in\R_{>0}$. One can look at the asymptotic behaviour ot the ratio $c_X(r)/c_Y(r)$ to define the relative asymptotic density $d_Y(X)$. If the procounting measure $\mu_Y$ exists, then our yoga suggests one should compare $d_Y(X)$ and $\mu_Y(\xa)$.

We carry out this program in Section \ref{ss:dens}, working, as usual, on a general $D$. The notion of density we use is an abstract one, defined via a series of properties: the most delicate of them, labelled (${\rm Dn}_Y$6), essentially asks for the elements of the set $Y$ to be equidistributed among residue classes modulo $\gotn$, for every non-zero ideal $\gotn$. When these properties are satisfied, we recover a number of results from \cite{dl}, useful for proving equalities of the form $d_Y(X)=\mu_Y(\xa)$.

Note that (${\rm Dn}_Y$6) holds for $D=\Z$ and $Y=\calp$: therefore the relative asymptotic density $d_\calp$ does exist. On the other hand, constructing non-trivial examples of a relative density $d_Y$ for a general $Y$ is a hard problem, well beyond the scope of this paper.\\

In order to extend Bateman-Horn to $D$ one has to choose a counting function on $D$: once this is done, we write $c_X$ for the function counting the elements in $X\subset D$. For example, for $D=\Z$, the left-hand side of \eqref{e:bhclssc} expresses the counting function $c_Y(r)$ for $Y=\bff^{-1}(\calp^k)\cap\N$.

Let $Y$ be a subset of $D$. We say that a counting function $c_Y$ induces a relative density if condition (${\rm Dn}_Y$6) is satisfied when one posits
$$d_{Y,c}(X):=\lim_{r\rightarrow\infty}\frac{c_{Y\cap X}(r)}{c_Y(r)}\,.$$

In Section \ref{ss:cmpr} we will recall the Bateman-Horn conjecture over $D=\F_q[t]$ (Conjecture \ref{cnj:BHfinite}). We shall prove the following result (a reformulation of Proposition \ref{p:4.11}; the set $\calp(D)$ shall be defined in \S\ref{sss:prmelm}).

\begin{customthm}{C} Let $D$ be either $\Z$ or $\F_q[t]$, with their respective counting function $c$. Let $f_1,\dots,f_k$ and $\bff$ be as in Conjecture \ref{cnj:BH} or \ref{cnj:BHmult}. Assume that $c_{\bff^{-1}(\calp(D)^k)}(r)$ induces a relative density for $\bff^{-1}(\calp(D)^k)$. Then the Bateman--Horn conjecture for $\bff$ implies its profinite version, Conjecture \ref{cnj:pBH} or \ref{cnj:pBHD}. \end{customthm}

In the single polynomial case, $\bff=f$, the conditions of Theorem C are satisfied if $f$ has degree $1$. Moreover, thanks to the work of Sawin and Shusterman, some cases of Conjecture \ref{cnj:BHmult} are known to be true: hence the profinite version follows if (${\rm Dn}_Y$6) is satisfied. Note that this requirement can be expressed as an equidistributed version of the Bateman-Horn conjecture, as it will appear from the asymptotic equality \eqref{e:thm4.11}. Also, the only obstacle to extending Theorem C to more general $D$ is the absence, to our knowledge, of an archimedean version of Bateman-Horn in such setting.\\

In the  conclusive paragraph \S\ref{sss:nonequiv}, we  argue that our profinite analogue \ref{cnj:pBH} seems weaker than its classical counterpart \ref{cnj:BH}: this alleged lack of equivalence may be a good motivation to further investigate our profinite case, as the classical one seems to be definitely out of reach.

\medskip
\noindent\textbf{Acknowledgments.}\quad The third author was supported by the ERC, SharpOS, 101087910, ISF grants 2067/23, 1963/20, and the BSF grant 2018250. We thank  Alejandro Vidal-L\'opez  for his assistance with Python and Zev Rosengarten for a correction in Lemma~\ref{l:qe}. We also thank Oren Ben-Bassat, Roman Panenko and Mark Shusterman for interesting conversations and Efthymios Sofos for his comments on a first draft of this paper. 

\medskip
\noindent\textbf{AI disclaimer.}\quad All of the new ideas of this work, and their executions, are due to the authors with no use of any LLM. 

\section{Profinite distributions} \label{s:dmps} 

\noindent{\bf Notations and conventions.} In the following,  by topological group we mean a group $G$ endowed with a topological structure of Hausdorff space such that the group operation is a continuous map from $G\times G$ to $G$. A topological ring $R$ is defined similarly, with the request that addition and multiplication are continuous. If $R$ is a topological ring, unless otherwise stated, we shall assume that all $R$-modules are topological groups and the map $R\times M\rightarrow M$ expressing the ring action of $R$ on $M$ is continuous. Also, unless otherwise stated, we shall assume that all $R$-module morphisms are continuous. 

For any topological spaces $A$, $B$, we write $\calc(A,B)$ to denote the set of continuous functions from $A$ to $B$. The characteristic function of a set $S$ will be denoted by $\b1_S$, simplifying it to $\b1_x$ when $S=\{x\}$ is a singleton. (By abuse of notation, we will use the same symbol $\b1_S$ independently of the unitary ring the characteristic function takes values in.)

For a ring $R$, we shall write $R^*$ for its group of units.

\subsection{Distributions on profinite sets} Throughout Section \ref{s:dmps}, we fix a profinite set
\begin{equation}\label{e:X} X:=\varprojlim_{\alpha\in \calj}X_\alpha \end{equation}
Here $\{X_\alpha\}_{\alpha\in \calj}$ is an inverse system of {\em finite} sets, labeled by $\alpha$ ranging in some arbitrary directed set of indexes $\calj$, with maps 
$\pi^{\beta}_{\alpha}\colon X_\beta\to X_\alpha$ for all $\beta\ge\alpha$ in $\calj$. Without loss of generality, we can (and will) assume that the maps $\pi_\alpha^\beta$ are all surjective. Let $\pi_\alpha\colon X\rightarrow X_\alpha$ denote the natural projection.

For simplicity, we also postulate that $\calj$ contains a cofinal subset isomorphic to $\N$.

\subsubsection{The profinite topology} In this paper finite sets are always assumed to have the discrete topology. Then the inverse limit topology on $X$ (that is, the coarsest topology such that all the maps $\pi_\alpha$ are continuous) makes it a profinite topological space.

\begin{lem} \label{l:clsr} If $S$ is a subset of $X$, its closure $\overline S$ is
$$\bigcap_{\alpha\in\calj_0}\pi_\alpha^{-1}(\pi_\alpha(S))\,,$$
where $\calj_0$ is any cofinal subset of $\calj$. As a consequence, one has $\pi_\alpha(S)=\pi_\alpha(\overline S)$ for every $\alpha\in\calj$. \end{lem}

\begin{prf} Assume $x$ is a point of $X$ outside the closure of $S$. Then there is a neighbourhood $U$ of $x$ such that $U\cap S=\emptyset$. One inclusion follows observing that, without loss of generality, we can take $U=\pi_\alpha^{-1}(\pi_\alpha(x))$ for some $\alpha\in\calj_0$. The opposite inclusion is obvious, as well as the final statement. \end{prf}

In the following, we shall say that $S\subseteq X$ is {\em $\alpha$-saturated} if $S=\pi_\alpha^{-1}(\pi_\alpha(S))$.

\subsubsection{Continuous and locally constant functions} Let $R$ be a (commutative and unitary) topological ring: then $\calc(X,R)$ becomes a topological $R$-module with the uniform 
convergence topology.

In order to simplify some arguments, we shall always assume that the topology on $R$ is induced by a non-trivial absolute value $|\cdot|_R$\,, so that $\calc(X,R)$ is endowed with the supremum norm $\|\cdot\|_\infty$\,. 

\begin{rmk} \label{r:lcfnstr} For further reference,  we note that $\calc(X_\alpha,R)$ is a free $R$-module with basis $\{\b1_x\}_{x\in X_\alpha}$, for every $\alpha\in\calj$, because $X_\alpha$ is finite and discrete. As a topological space, $\calc(X_\alpha,R)$ is homeomorphic to the product $R^{|X_\alpha|}$. \end{rmk}

For all $\beta\ge\alpha$ in $\calj$, we have a map 
$$(\pi_\alpha^\beta)^*\colon\calc(X_\alpha,R)\longrightarrow\calc(X_\beta,R)$$
by $f\mapsto f\circ\pi_\alpha^\beta$\,. Thus we get a direct system, with limit
\begin{equation}\label{e:lcdf} \call_c(X,R):=\varinjlim_{\alpha\in \calj}\calc(X_\alpha,R)\,. \end{equation}
There is a continuous injection $\call_c(X,R)\hookrightarrow\calc(X,R)$ induced by the maps
\[ \pi_\alpha^*\colon\calc(X_\alpha,R)\rightarrow\calc(X,R)\qquad  f\mapsto f\circ\pi_\alpha\;. \]

\begin{lem} \label{l:lclmcst} The image of $\call_c(X,R)$ in $\calc(X,R)$ consists exactly of the $R$-valued locally constant functions on $X$.  \end{lem}

\begin{prf} By definition, the fibres $\pi_\alpha^{-1}(\pi_\alpha(x))$, as $\alpha$ varies in $\calj$ and $x\in X$, form a base of the topology. Hence, if $f$ is locally constant, there 
is a cover of such fibres such that $f$ is constant on each of them. By compactness, one can extract a finite subcover, 
$$X=\bigcup_{i=1}^n\pi_{\alpha_i}^{-1}(\pi_{\alpha_i}(x_i)).$$
Take $\alpha\in\calj$ such that $\alpha\ge\alpha_i$ for $i=1,\dots,n$. Then $f$ factors through $X_\alpha$\,. \end{prf}

\begin{cor} \label{c:dl213} A subset of $X$ is compact open if and only if it is $\alpha$-saturated for some $\alpha\in\calj$. \end{cor}

\begin{prf} If $U\subseteq X$ is compact open, then its characteristic function $\b1_U$ is locally constant, hence one has $\b1_U=f\circ\pi_\alpha$ for some $\alpha$ and $f\in\calc(X_\alpha,R)$. This is possible only if $f=\b1_S$ for some $S\subseteq X_\alpha$, with $U=\pi_\alpha^{-1}(S)$. Therefore $U=\pi_\alpha^{-1}(\pi_\alpha(U))$. 

The converse implication is obvious from the definition of the profinite topology. \end{prf}

\begin{lem} \label{l:lcdnsc} The image of $\call_c(X,R)$ is dense in $\calc(X,R)$. \end{lem}

\begin{prf} Let $f\in\calc(X,R)$ and fix $U\subseteq R$ a neighbourhood of 0. By continuity, for any $x\in X$ we can find a neighbourhood $A_x$ such that $f(y)-f(x)\in U$ for all $y\in A_x$\,. We can also assume  $A_x=\pi_{\alpha_x}^{-1}(\pi_{\alpha_x}(x))$ for some $\alpha_x\in\calj$. By compactness, we can extract a finite subcover $\{A_{x_1},...,A_{x_n}\}$ from the open cover $\{A_x\}$. Since $\calj$ is directed, one can find an index $\alpha\ge\alpha_{x_i}$, $i=1,...,n$. For any $y\in X_\alpha$, choose $\tilde y\in\pi_\alpha^{-1}(y)$ and define a function $f_U$ by 
$$f_U(x)=\sum_{y\in X_\alpha}f(\tilde y)\b1_{\pi_\alpha^{-1}(y)}\,.$$
Then $f_U$ is locally constant and $f(x)-f_U(x)\in U$ for every $x\in X$. \end{prf}

By abuse of notation, in the following we will often identify $\call_c(X,R)$ and $\calc(X_\alpha,R)$ with their images in $\calc(X,R)$; note, however, that the topology on $\call_c(X,R)$ is not the one as a subspace of $\calc(X,R)$. More precisely, the topology on $\call_c(X,R)$ is induced by \eqref{e:lcdf}: that is, $\call_c(X,R)$ is given the finest topology such that all the embeddings $\calc(X_\alpha,R)\hookrightarrow\call_c(X,R)$ are continuous.

Finally, we observe that if $R$ is complete then a standard argument shows that so is $\calc(X,R)$. The same holds for $\call_c(X,R)$, with the direct limit topology, as we now prove.

\begin{prop} \label{p:cxrcmplt} If $R$ is complete then so is $\call_c(X,R)$. \end{prop}

\begin{prf} We are going to show that, by definition of direct limit topology, a sequence $(f_n)_n$ is Cauchy in $\call_c(X,R)$ only if there is some $\alpha$ such that $f_n\in\calc(X_\alpha,R)$ for every $n\gg0$. Hence the completeness of $\calc(X_\alpha,R)$ for every $\alpha$ implies that also $\call_c(X,R)$ is complete. 

We reason as in \cite[Theorem 6.5]{rudin2}. Recall that $E\subseteq\call_c(X,R)$ is bounded if for every $U$ neighbourhood of $0$ there exists $\varepsilon>0$ such that $a\in R$, $|a|_R<\varepsilon$ implies $aE\subseteq U$. We will prove that if there is no $\alpha$ such that $E\subseteq\calc(X_\alpha,R)$ then $E$ is unbounded (and hence it cannot consist of a Cauchy sequence).
 
For $\alpha\in\calj$, define a continuous function $m_\alpha\colon\calc(X,R)\rightarrow\R$ by letting $\{U_{\alpha,i}\}$ be the partition of $X$ into fibres of $\pi_\alpha$ and putting
$$m_\alpha(f):=\max_i\left\{\sup_{x,y\in U_{\alpha,i}}\big|f(x)-f(y)\big|_R\right\}\,.$$
Then $\beta\ge\alpha$ implies $m_\alpha(f)\ge m_\beta(f)$ and one has $f\in\calc(X_\alpha,R)$ if and only if $m_\alpha(f)=0$.

Let $\calj_0\simeq\N$ be a cofinal subset of $\calj$. If $E$ is not contained in any $\calc(X_\alpha,R)$, then for every $\alpha\in\calj_0$ there is $f_\alpha\in E$ such that $m_\alpha(f_\alpha)>0$. Choose non-zero elements $a_\alpha\in R$, indexed by $\calj_0$, so that $0$ is a limit point, and consider the set
$$W=\big\{f\in\call_c(X,R)\mid m_\alpha(f)<m_\alpha(a_\alpha f_\alpha)\;\forall\,\alpha\in\calj_0\big\}.$$
We claim that $W$ is open, i.e., $(\pi_\beta^*)^{-1}(W)$ is open for every $\beta\in\calj$. Indeed, the hypothesis on $\calj_0$ implies that for any fixed $\beta$ there are only finitely many $\alpha$'s in $\calj_0$ such that $m_\alpha$ does not vanish identically on $\calc(X_\beta,R)$.

On the other hand one has $aE\nsubseteq W$ for every $a\in R$, $a\neq0$, because $|a_\alpha|_R<|a|_R$ implies 
$$m_\alpha(a_\alpha f_\alpha)=|a_\alpha|_R\cdot m_\alpha(f_\alpha)<|a|_R\cdot m_\alpha(f_\alpha)=m_\alpha(af_\alpha)$$
and hence $af_\alpha\notin W$. Therefore $E$ cannot be bounded. \end{prf} 

\subsubsection{Measures and distributions} For a (pro)finite topological space, we define $R$-valued measures and distributions as the topological duals, respectively, of continuous and locally constant functions (so that they coincide when the space is finite). For $X$ as in \eqref{e:X} we obtain the $R$-modules 
$$\calm(X,R):=\Hom_R(\calc(X,R),R)$$
and
$$\cald(X,R):=\Hom_R(\call_c(X,R),R)\,.$$
(In both cases, we only consider continuous homomorphisms.) From \eqref{e:lcdf} one immediately gets
\begin{equation} \label{e:dxlmdxa} \cald(X,R)=\varprojlim_{\alpha\in \calj}\cald(X_\alpha,R) \end{equation}
where the inverse system is built by the obvious maps
$$(\pi_\alpha^\beta)_*\colon\cald(X_\beta,R)\longrightarrow\cald(X_\alpha,R)\,.$$

\begin{rmk} For a more concrete grasp of $\cald(X,R)$, it might be useful to note that there is a natural bijection between distributions and finitely additive $R$-valued functions on compact open subsets of $R$. In particular, in order to define $\delta\in\cald(X,R)$ it suffices to specify $\delta(\b1_U)$ for all compact open $U$ (and check that $U=U_1\sqcup U_2$ corresponds to $\delta(\b1_U)=\delta(\b1_{U_1})+\delta(\b1_{U_2})$). \end{rmk}

The injection of $\call_c(X,R)$ into $\calc(X,R)$ implies that every measure becomes a distribution when restricting it to locally constant functions.

\begin{lem} \label{l:msrinjd} The natural map $\calm(X,R)\rightarrow\cald(X,R)$ is injective. \end{lem}

\begin{prf} Assume $\mu_1$ and $\mu_2$ are different in $\calm(X,R)$. Then there is $f\in\calc(X,R)$ such that $\mu_1(f)\neq\mu_2(f)$. By Lemma \ref{l:lcdnsc} there is some sequence $(f_n)_{n\in\N}$ of locally constant functions which converge uniformly to $f$. Since both $\mu_i$'s are continuous, they satisfy
$$\mu_i(f)=\lim_{n\rightarrow\infty}\mu_i(f_n)\,.$$
Thus there must be some locally constant $f_k$ such that $\mu_1(f_k)\neq\mu_2(f_k)$. \end{prf}

\begin{lem} \label{l:msrdstrb} Let $\delta$ be an $R$-valued distribution on $X$. Assume that $R$ is complete and that, as $U$ varies among compact open subsets of $X$, the function $U\mapsto\delta(\b1_U)$ either is bounded (if $|\cdot|_R$ is non-archimedean) or takes values in $\R_{\ge0}$ (if $|\cdot|_R$ is archimedean). Then $\delta$ is a measure on $X$. \end{lem}

\begin{prf} By Lemma \ref{l:lcdnsc}, any $f\in\calc(X,R)$ is the uniform limit of a sequence $(f_n)$ of locally constant functions. The inequalities
$$|\delta(f_n)-\delta(f_m)|_R\le\begin{cases} \|f_n-f_m\|_\infty\cdot\sup_U|\delta(\b1_U)|_R & \text{ if }|\,|_R\text{ is non-archimedean} \vspace{2pt} \\ 
\|f_n-f_m\|_\infty\cdot\delta(\b1_X) & \text{ if } \delta(\b1_U)\in\R_{\ge0}\;\forall\,U \end{cases} $$
show that the sequence $\delta(f_n)$ is Cauchy. Its limit will be $\delta(f)$. \end{prf}

\begin{rmk} \label{r:albr1} Choosing a cofinal $\calj_0\simeq\N$ in $\calj$, the space $X$ can be realized as the boundary of a rooted tree and the results of \cite{ccs} apply to $\cald(X,\C)$. In particular, \cite[Theorem 3.2]{ccs} proves that a $\C$-valued distribution is a measure if and only if it is absolutely summable or, equivalently, its total variation is bounded (we refer to \cite{ccs} for the definition of these terms). \end{rmk}

Note that $\cald(X,R)$ is a module over the ring $\call_c(X,R)$, by
\begin{equation} \label{e:heckeprod} (f\cdot\delta)(g):=\delta(fg) \end{equation}
(which makes sense because $g\mapsto fg$ is a continuous endomorphism of $\call_c(X,R)$ for any locally constant $f$).\\

The simplest example of distributions are probably Dirac deltas, 
$$\delta_x(\b1_U):=\begin{cases}1 & \text{ if }x\in U, \\ 0 & \text{ if } x\notin U. \end{cases} $$

\begin{dfn} \label{d:spprt} The {\em support} of a distribution $\mu\in\cald(X,R)$ is the smallest closed subset $S\subseteq X$ with the property $\mu(\b1_U)=0$ for every compact open $U$ such that $S\cap U=\emptyset$. \end{dfn}

Equivalently, $x\in X$ is not in the support of $\mu$ if and only if $\mu$ is locally trivial at $x$ (that is, $x$ has an open neighbourhood $A$ such that $U\subseteq A$ implies $\mu(\b1_U)=0$). 

\begin{lem} \label{l:fntspprt} A distribution $\mu\in\cald(X,R)$ has finite support if and only if it is a linear combination of Dirac deltas. \end{lem}

\begin{prf} Assume that $\mu$ has finite support $S=\{x_1,\dots,x_n\}$. Fix $\alpha\in\calj$ so that the images $\pi_\alpha(x_i)$ are all distinct and put $c_i=\mu(\b1_{\pi_\alpha^{-1}(\pi_\alpha(x_i))})$. By definition of support one has $\mu(\b1_{\pi_\beta^{-1}(x)})=0$ for any $\beta\ge\alpha$ and $x\in X_\beta-\pi_\beta(S)$, hence 
$$\mu(\b1_{\pi_\beta^{-1}(\pi_\beta(x_i))})=\sum_{x\in(\pi_\alpha^\beta)^{-1}(\pi_\alpha(x_i))}\mu(\b1_{\pi_\beta^{-1}(x)})=\mu(\b1_{\pi_\alpha^{-1}(\pi_\alpha(x_i))})=c_i\,.$$
It follows $\mu=\sum_ic_i\delta_{x_i}$\,.

The opposite implication is obvious. \end{prf}

\subsubsection{The topology of $\cald(X,R)$} \label{sss:tplgd} Both $\calm(X,R)$ and $\cald(X,R)$ are endowed with the weak-* topology (that is, the coarsest topology which makes all evaluation maps $\ev_f\colon\mu\mapsto\mu(f)$ continuous, where $f$ varies respectively in either $\calc(X,R)$ or $\call_c(X,R)$\,). One easily checks that the $R$-module operations and the injection of Lemma \ref{l:msrinjd} are continuous.  Moreover, since $\calc(X,R)$ is a metric space, on $\calm(X,R)$ we also have the 
topology induced by the norm
\[ \|\mu\|:=\sup_{\substack{f\in\calc(X,R)\\f\neq0}}\frac{|\mu(f)|}{\|f\|_{\infty}}\ . \]

\begin{lem} \label{l:darflx} Let $\alpha\in\calj$. The $R$-module $\cald(X_\alpha,R)$ is free on the basis $\{\delta_x\}_{x\in X_\alpha}$, where $\delta_x$ is the Dirac functional at $x$. The weak-* topology on $\cald(X_\alpha,R)$ is the same as the product topology as a free module. \end{lem}

\begin{prf} Recall Remark \ref{r:lcfnstr}. The pairing $\displaystyle\bigoplus_{x\in X_\alpha}R\delta_x\times\calc(X_\alpha,R)\rightarrow R$ given by
\begin{equation} \label{e:pr1drc} \left(\sum_xa_x\delta_x\ ,\ \sum_xb_x\b1_x\right)\mapsto\sum_xa_xb_x \end{equation}
is continuous in both variables because $R$ is a topological ring. The first statement is an immediate consequence.

As for the second claim, it is clear from \eqref{e:pr1drc} that the product topology on $\cald(X_\alpha,R)$ makes $\ev_f$ continuous for any $f\in\calc(X_\alpha,R)$ and hence is finer than the weak-* topology. It is also coarser, because, by \eqref{e:pr1drc}, the coordinate maps defining it are $\{\ev_{\b1_x}\}_{x\in X_\alpha}$, which must be continuous with respect to the weak-* topology. \end{prf} 

Lemma \ref{l:darflx} yields an explicit description of the structure maps $(\pi_\alpha)_*\colon\cald(X,R)\rightarrow\cald(X_\alpha,R)$ from \eqref{e:dxlmdxa}, namely
\begin{equation} \label{e:pfesplc} (\pi_\alpha)_*(\mu)\colon f=\sum_{x\in X_\alpha}a_x\b1_x\mapsto\sum_{x\in X_\alpha}a_x\mu(\b1_{\pi_\alpha^{-1}(x)})=\sum_{x\in 
X_\alpha}\mu(\b1_{\pi_\alpha^{-1}(x)})\delta_x(f)\,. \end{equation}

\begin{prop} \label{p:dxtplg} On $\cald(X,R)$, the inverse limit topology from \eqref{e:dxlmdxa} is the same as the weak-* topology. \end{prop}

\begin{prf} Each of the two topologies is defined as the coarsest one which makes continuous the maps in a certain set, namely the evaluation maps $\{\ev_f\}_{f\in\call_c(X,R)}$ in the weak-* case and the projections $\{(\pi_\alpha)_*\}_{\alpha\in\calj}$ in the inverse limit case. So it is enough to show that if all the maps in one of the two sets are continuous 
then so are the ones in the other set and vice versa.

Take $f\in\call_c(X,R)$. Then, by \eqref{e:lcdf}, one has $f=\pi_\alpha^*(g)$ for some $g\in\calc(X_\alpha,R)$ and the equality
$$\mu(f)=\mu(\pi_\alpha^*(g))=(\pi_\alpha)_*(\mu)(g)$$
yields $\ev_f=\ev_g\circ(\pi_\alpha)_*$, with $\ev_g\colon\cald(X_\alpha,R)\rightarrow R$ the evaluation at $g$, which is continuous by Lemma \ref{l:darflx}. This proves that if all the projections are continuous then so are the evaluation maps.

On the other hand, \eqref{e:pfesplc} shows that the $\delta_x$-coordinate of $(\pi_\alpha)_*$ is simply the evaluation at $\b1_{\pi_\alpha^{-1}(x)}$\,. Therefore each $(\pi_\alpha)_*$ is continuous if all the evaluation maps are so. \end{prf}

\begin{cor} If $R$ is complete, then so is $\cald(X,R)$. \end{cor}

\begin{prf} An inverse limit of complete topological modules is complete. In this particular case, if $(\mu_n)_{n\in\N}$ is a Cauchy sequence in $\cald(X,R)$, then each $((\pi_\alpha)_*(\mu_n))_{n\in\N}$ is also Cauchy, with limit $\mu_\alpha$. The equalities $(\pi_\alpha^\beta)_*(\mu_\beta)=\mu_\alpha$ hold (by continuity and uniqueness of the limit) for every $\beta\ge\alpha$, so there is $\mu\in\cald(X,R)$ such that $(\pi_\alpha)_*(\mu)=\mu_\alpha$. Finally, $\mu=\lim_n\mu_n$ is true by construction.  \end{prf}

\begin{rmk} \label{r:cnvlt} When $X$ is a group, $\cald(X,R)$ is an algebra, with the convolution product
$$\cald(X,R)\times\cald(X,R)\longrightarrow\cald(X,R),\;\;\;\;(\mu,\nu)\mapsto\mu\nu\,,$$
defined by push-forward of the group law $X\times X\rightarrow X$, $(x,y)\mapsto xy$. One can check that on Dirac deltas this becomes $\delta_x\delta_y=\delta_{xy}$. Hence $\cald(X_\alpha,R)=\bigoplus_{x\in X_\alpha}R\delta_x$ is isomorphic to the group algebra $R[X_\alpha]$ and taking the limit one obtains
$$\cald(X,R)\simeq R[\![X]\!]:=\varprojlim_{\alpha\in \calj}R[X_\alpha]\,.$$
This additional structure plays an important role in Fourier analysis, with $R=\C$ (as we shall recall in \S\ref{sss:fourier}) and in Iwasawa theory, with $R=\Z_p$ and $X$ a $p$-adic Lie group. \end{rmk}

\subsubsection{Products} \label{sss:prdtt} Assume we are given a collection of profinite sets $\{Y_\kappa\}_{\kappa\in\calk}$, defined by $Y_\kappa=\varprojlim Y_{\kappa,\alpha_\kappa}$, where $\alpha_\kappa$ varies in a directed set $\calj_\kappa$ and each $Y_{\kappa,\alpha_\kappa}$ is finite. We want to describe distributions on $\prod_\kappa Y_\kappa$. 

\begin{lem} \label{l:prdst} Let $\{Y_\kappa\}_{\kappa\in\calk}$ be a collection of profinite sets, as above, and put $X:=\prod_{\kappa\in\calk}Y_\kappa$. Consider the directed set 
$\calj=\bigcup_K\prod_{\kappa\in K}\calj_\kappa$, where $K$ varies among all finite subsets of $\calk$ and the order is given by
$$\beta=(\beta_\kappa)_{\kappa\in K'}\ge\alpha=(\alpha_\kappa)_{\kappa\in K}\text{ if }K\subseteq K'\text{ and } \beta_\kappa\ge\alpha_\kappa\text{ for all }\kappa\in K.$$ 
Then $X=\displaystyle\varprojlim_{\alpha\in\calj}X_\alpha$\,, with $X_\alpha:=\prod_{\kappa\in K}Y_{\kappa,\alpha_\kappa}$ for $\alpha=(\alpha_\kappa)_{\kappa\in K}\in\calj$. \end{lem}

\begin{prf} Obvious by abstract nonsense.  \end{prf}

\begin{lem} \label{l:prddstrbf} Let $Y_1,\dots,Y_n$ be finite sets. There is an isomorphism of $R$-modules
$$\bigotimes_{i=1}^n\cald(Y_i,R)\rightarrow\cald(Y_1\times\dots\times Y_n,R)$$
given by $\delta_{y_1}\otimes\dots\otimes\delta_{y_n}\mapsto\delta_{(y_1,\dots,y_n)}$, where $y_i\in Y_i$. \end{lem}

\begin{prf} Straightforward from Lemma \ref{l:darflx}. \end{prf}

\begin{prop} \label{p:prdfdstr} Let $\{Y_\kappa\}_{\kappa\in K}$ be a finite collection of profinite sets and put $X:=\prod_{\kappa\in K}Y_\kappa$. Then there is a natural morphism of $R$-modules
\begin{equation} \label{e:tnsr} \bigotimes_{\kappa\in K}\cald(Y_\kappa,R)\longrightarrow\cald(X,R) \end{equation}
with dense image. \end{prop}

\begin{prf} We use the notation of Lemma \ref{l:prdst}. The isomorphism of Lemma \ref{l:prddstrbf} commutes with the structure maps induced by $\pi_{\kappa,\alpha_\kappa}^{\beta_\kappa}\colon Y_{\kappa,\beta_\kappa}\rightarrow Y_{\kappa,\alpha_\kappa}$ on the one side and $\pi_\alpha^\beta\colon X_\beta\rightarrow X_\alpha$ on the other. Taking the limit and composing with the natural morphism
$$\bigotimes_{\kappa\in K}\cald(Y_\kappa,R)=\bigotimes_{\kappa\in K}\varprojlim\cald(Y_{\kappa,\alpha_\kappa},R)\longrightarrow\varprojlim\bigotimes_{\kappa\in K}\cald(Y_{\kappa,\alpha_\kappa},R),$$ 
we obtain \eqref{e:tnsr}. The image is dense because the composition with  $(\pi_\alpha)_*\colon\cald(X,R)\rightarrow\cald(X_\alpha,R)$ is surjective for every $\alpha$, by Lemma \ref{l:prddstrbf}. \end{prf}

\begin{rmk} In order for \eqref{e:tnsr} to be an isomorphism, one has to replace $\bigotimes\cald(Y_\kappa,R)$ with $\widehat\bigotimes\cald(Y_\kappa,R)$, the completed tensor product. For an explicit example of why this is needed, just consider the case $R=\Z_p$ and $Y_1=Y_2=\Z_p$. It is well known that then one has $\cald(Y_i,R)\simeq\Z_p[\![t_i]\!]$ and $\cald(Y_1\times Y_2,R)\simeq\Z_p[\![t_1,t_2]\!]$. The natural map $\Z_p[\![t_1]\!]\otimes\Z_p[\![t_2]\!]\rightarrow\Z_p[\![t_1,t_2]\!]$ is not surjective. \end{rmk}

We shall denote the image via \eqref{e:tnsr} of $\bigotimes_\kappa\mu_\kappa$ by the same symbol and call it the {\em product of the distributions} $\{\mu_\kappa\}_{\kappa\in K}$\,. 

\begin{lem} Let $X$, $\{Y_\kappa\}_{\kappa\in K}$ be as in Proposition \ref{p:prdfdstr}. The product distribution $\bigotimes_\kappa\mu_\kappa$ is uniquely characterized by the property
\begin{equation} \label{e:tnsr1} \bigotimes_\kappa\mu_\kappa\left(\b1_{\prod U_\kappa}\right)=\prod_{\kappa\in K}\mu_\kappa(\b1_{U_\kappa})\;, \end{equation}
for  any choice of compact open subsets $U_\kappa\subseteq Y_\kappa$\,. \end{lem}

\begin{prf} By definition of the product topology, any compact open subset of $X$ can be written as a finite union of sets of the form $\prod_\kappa U_\kappa$, with $U_\kappa\subseteq Y_\kappa$ compact open. Thus a distribution on $X$ is determined by its values on the functions $\b1_{\prod U_\kappa}$. 

Notations are as in the proofs of Lemma \ref{l:prdst} and Proposition \ref{p:prdfdstr}: in particular, we have structure maps $\pi_\alpha\colon X\rightarrow X_\alpha$ and $\pi_{\kappa,\alpha_\kappa}\colon Y_\kappa\rightarrow Y_{\kappa,\alpha_k}$. In order to check \eqref{e:tnsr1}, there is no loss of generality in taking sets of the form $U_\kappa=\pi_{\kappa,\alpha_\kappa}^{-1}(x_\kappa)$, so to have $\prod U_\kappa=\pi_\alpha^{-1}(x)$, with $x=(x_\kappa)$. Thus one has
\begin{equation} \label{e:pgnl} \bigotimes_\kappa\mu_\kappa\big(\b1_{\prod U_\kappa}\big)=\bigotimes_\kappa\mu_\kappa\big(\b1_{\pi_\alpha^{-1}(x)}\big)\stackrel{\dag}{=}\big((\pi_\alpha)_*\bigotimes_\kappa\mu_\kappa\big)(\b1_x)\stackrel{\ddag}{=}\big(\bigotimes_\kappa(\pi_{\kappa,\alpha_\kappa})_*\mu_\kappa\big)(\b1_x) \end{equation}
where $\dag$ applies \eqref{e:pfesplc} and $\ddag$ follows from the definition of the map \eqref{e:tnsr}. On the other hand, putting $(\pi_{\kappa,\alpha_\kappa})_*(\mu_\kappa)=\sum c_{y_\kappa}\delta_{y_\kappa}$, the isomorphism of Lemma \ref{l:prddstrbf} allows to evaluate the last term of \eqref{e:pgnl} as 
$$\left(\bigotimes_{\kappa\in K}\left(\sum_{y_\kappa\in Y_{\kappa,\alpha_\kappa}}c_{y_\kappa}\delta_{y_\kappa}\right)\right)(\b1_x)=\prod_{\kappa\in K}\left(\sum_{y_\kappa\in Y_{\kappa,\alpha_\kappa}}c_{y_\kappa}\delta_{y_\kappa}(\b1_{x_\kappa})\right)=\prod_{\kappa\in K}c_{x_\kappa}$$
and we conclude, because \eqref{e:pfesplc} yields $c_{x_\kappa}=\mu_\kappa(\b1_{U_\kappa})$ for all $\kappa$. \end{prf}

Equality \eqref{e:tnsr1} shows the problem in extending Proposition \ref{p:prdfdstr} to the case when the set of indexes $\calk$ is infinite. Indeed, taking $\mu_\kappa\in\cald(Y_\kappa,R)$ such that $\mu_\kappa(\b1_{Y_\kappa})=c$ for every $\kappa$, where $c\in R$ is such that the sequence $c^n$ has no limit, then we see no way of giving to $(\bigotimes_\kappa\mu_\kappa)(\b1_X)$ a meaning compatible with \eqref{e:tnsr1}.\\

We say that an infinite product $\prod_{\kappa\in\calk}c_\kappa$ {\em converges in $R$} if the equality $\lim_K\prod_{\kappa\in K}c_\kappa=r$ (where the limit is taken for $K$ varying among finite subsets of $\calk$) is satisfied for some $r\in R$.

\begin{thm} \label{t:dstrprd} Let $X=\prod_{\kappa\in\calk}Y_\kappa$, where each $Y_\kappa$ is profinite. For each $\kappa\in\calk$ let $\mu_\kappa\in\cald(Y_\kappa,R)$ be such that $\prod_{\kappa\in\calk}\mu_\kappa(\b1_{Y_\kappa})$ converges. Then the product distribution $\bigotimes_\kappa\mu_\kappa$ exists.  \end{thm}

\begin{prf} Let $U$ be a compact open subset of $X$.  We first consider the situation
\begin{equation} \label{e:Uprd} U=\prod_{\kappa\in K}U_\kappa\times\prod_{\kappa\in\calk-K}Y_\kappa\;, \end{equation}
where $K\subseteq\calk$ is finite and each $U_\kappa\subseteq Y_\kappa$ compact open. In this case \eqref{e:tnsr1} forces us to define 
\begin{equation} \label{e:tnsr2} \bigotimes_\kappa\mu_\kappa(\b1_U):=\prod_{\kappa\in K}\mu_\kappa(\b1_{U_\kappa})\cdot\prod_{\kappa\in\calk-K}\mu_\kappa(\b1_{Y_\kappa})\,. \end{equation}
The hypothesis implies that the product on the right-hand side of \eqref{e:tnsr2} converges. Moreover, one immediately sees that the limit is independent of the choice of $K$ in \eqref{e:Uprd}.

For the general case, Corollary \ref{c:dl213} and Lemma \ref{l:prdst} together imply that any compact open $U$ is $\alpha$-saturated for some $\alpha=(\alpha_\kappa)_{\kappa\in K}$. Hence one can write $U=\bigsqcup_{i=1}^nU_i$ with each $U_i=\prod_\kappa U_{i,\kappa}$ as in \eqref{e:Uprd}. Then we put
\begin{equation} \label{e:tnsr3} \bigotimes_\kappa\mu_\kappa(\b1_U):=\sum_{i=1}^n\bigotimes_\kappa\mu_\kappa(\b1_{U_i})=\sum_{i=1}^n\prod_{\kappa\in K}\mu_\kappa(\b1_{U_{i,\kappa}})\cdot\prod_{\kappa\in\calk-K}\mu_\kappa(\b1_{Y_\kappa})\,. \end{equation}
(Note that we can use the same auxiliary set $K$ for all $U_i$, since the value in \eqref{e:tnsr2} does not depend on this choice.)

We have to check that \eqref{e:tnsr3} is independent of the decomposition $U=\bigsqcup_iU_i$. So let $U$ and $V$ be compact open subsets of $X$, with empty intersection, and such that $U,V$ and $U\cup V$ are all of the form \eqref{e:Uprd} with respect to the same $K$. Then one has 
\begin{align*} \bigotimes_\kappa\mu_\kappa(\b1_{U\cup V}) & = \prod_{\kappa\in K}\mu_\kappa(\b1_{(U\cup V)_\kappa})\cdot\prod_{\kappa\notin K}\mu_\kappa(\b1_{Y_\kappa})=\\
& \stackrel{\ast}{=}\bigotimes_{\kappa\in K}\mu_\kappa(\b1_{\prod(U\cup V)_\kappa})\cdot\prod_{\kappa\notin K}\mu_\kappa(\b1_{Y_\kappa})=\\
& \stackrel{**}{=}\left(\bigotimes_{\kappa\in K}\mu_\kappa(\b1_{\prod U_\kappa})+\bigotimes_{\kappa\in K}\mu_\kappa(\b1_{\prod V\kappa})\right)\cdot\prod_{\kappa\notin K}\mu_\kappa(\b1_{Y_\kappa})=\\
& \stackrel{*}{=}\left(\prod_{\kappa\in K}\mu_\kappa(\b1_{\prod U_\kappa})+\prod_{\kappa\in K}\mu_\kappa(\b1_{\prod V\kappa})\right)\cdot\prod_{\kappa\notin K}\mu_\kappa(\b1_{Y_\kappa})= \bigotimes_\kappa\mu_\kappa(\b1_U)+\bigotimes_\kappa\mu_\kappa(\b1_V)  \end{align*}
where $\ast$ is \eqref{e:tnsr1} and $\ast\ast$ uses the fact that a finite product of distributions is a distribution (hence, finitely additive). \end{prf}

\begin{rmk} By the above definition of convergence, we allowed $\lim_K\prod_{\kappa\in K}\mu_\kappa\big(\b1_{Y_\kappa}\big)=0$. When this happens, $\bigotimes_\kappa\mu_\kappa$ need not be trivial: for example, fix a finite set $K_0$ and choose the distributions $\mu_\kappa$ so that $\mu_\kappa(\b1_{Y_\kappa})$ is $0$ if $\kappa\in K_0$ and $1$ otherwise, but for every $\kappa$ there is $U_\kappa\subset Y_\kappa$ such that $\mu_\kappa\big(\b1_{U_\kappa}\big)\neq0$. Distributions with total mass $0$ play an important role in certain parts of non-archimedean analysis (see e.g. \cite{t})  and  appear also in harmonic analysis on profinite spaces (as in \cite{ccs}), so it seems more convenient to keep this possibility open also for our definition of product distribution.  \end{rmk}

In general, one cannot reconstruct the distributions $\mu_\kappa$ from their product. However, we have the following result.

\begin{cor} Let $X=\prod_{\kappa\in\calk}Y_\kappa$, as in Theorem \ref{t:dstrprd}, with projections $\pi_\kappa\colon X\rightarrow Y_\kappa$. For each $\kappa$ choose $\mu_\kappa\in\cald(Y_\kappa,R)$ such that $\mu_\kappa(\b1_{Y_\kappa})=1$. Then $(\pi_\iota)_*\big(\bigotimes_\kappa\mu_\kappa\big)=\mu_\iota$ holds for every $\iota\in\calk$. \end{cor}

\begin{prf} Let $U\subseteq Y_\iota$ be compact open. The equalities
$$\big((\pi_\iota)_*\big(\bigotimes_\kappa\mu_\kappa\big)\big)(\b1_U)=\bigotimes_\kappa\mu_\kappa\big(\b1_{\pi_\iota^{-1}(U)}\big)= \mu_\iota(\b1_U)\cdot\prod_{\kappa\neq\iota}\mu_\kappa(\b1_{Y_\kappa})=\mu_\iota(\b1_U)$$
(which hold by definition of push-forward and by \eqref{e:tnsr2}, respectively) yield our thesis. \end{prf}

\subsection{Procounting distributions} \label{sss:Hrlkd} From now on we assume that  $R$ is a $\Q$-algebra, so to have
$$\frac{1}{|X_\alpha|}\in R$$
for all $\alpha\in\calj$.

\subsubsection{A categorical limit} We define:
\begin{equation} \label{e:mualfa} \mu_\alpha:=\frac{1}{|X_\alpha|}\sum_{x\in X_\alpha}\delta_x\in\cald(X_\alpha,R). \end{equation}
This is a distribution on $X_\alpha$\,.

\begin{dfn} \label{d:muS1} The {\em procounting distribution} on $X$ is
$$\mu_X:=\varprojlim_{\alpha\in \calj}\mu_\alpha\in \cald(X,R)$$
(when this limit exists). \end{dfn}

\begin{prop} \label{p:hrl} The inverse limit $\mu_X$ exists as a distribution in $\cald(X,R)$ if and only if
\begin{equation} \label{e:condhaar} |(\pi^\beta_\alpha)^{-1}(x)|=\frac{|X_\beta|}{|X_\alpha|}\;\;\text{ for all }x\in X_\alpha\,,\beta\ge\alpha\text{ in }\calj.\end{equation}
\end{prop}

\begin{prf} By definition of inverse limit, $\mu_X$ exists if and only if the equality 
\begin{equation} \label{e:condhaar0} \mu_\alpha=(\pi_\alpha^\beta)_*(\mu_\beta) \end{equation} 
is satisfied for every $\beta\ge\alpha$ in $\calj$. By \eqref{e:mualfa}, condition \eqref{e:condhaar0} can be rewritten as
\begin{align*} \frac{1}{|X_\alpha|}\sum_{x\in X_\alpha}\delta_x & = \mu_\alpha=(\pi_\alpha^\beta)_*(\mu_\beta)=\frac{1}{|X_\beta|}\sum_{y\in X_\beta}(\pi_\alpha^\beta)_*(\delta_y)= \frac{1}{|X_\beta|}\sum_{y\in X_\beta}\delta_{\pi^\beta_\alpha(y)}=\\
& =\frac{1}{|X_\beta|}\sum_{x\in X_\alpha}\sum_{y\in (\pi^\beta_\alpha)^{-1}(x)}\delta_x=\frac{1}{|X_\beta|}\sum_{x\in X_\alpha}|(\pi^\beta_\alpha)^{-1}(x)|\delta_x \end{align*}
The equivalence between \eqref{e:condhaar} and \eqref{e:condhaar0} now follows by Lemma \ref{l:darflx}. \end{prf}

\begin{rmks} \label{r:msrx} \begin{itemize} \item[]\end{itemize}
\noindent{\bf \ref{r:msrx}.1.} Condition \eqref{e:condhaar} is obviously satisfied when $X$ is a profinite group. In this case the inverse limit $\mu_X$ is precisely the Haar measure on $X$.\\ 
\noindent{\bf \ref{r:msrx}.2.} Assume $\mu_X$ exists and let $U\subseteq X$ be compact open. Then $U=\pi_\alpha^{-1}(\pi_\alpha(U))$ for some $\alpha$, by Corollary \ref{c:dl213}, and thus 
$\b1_U=\b1_{\pi_\alpha(U)}\circ\,\pi_\alpha$. Definition \ref{d:muS1} yields
$$\mu_\alpha=(\pi_\alpha)_*(\mu_X)=\mu_X\circ(\pi_\alpha)^*\,,$$
which, together with \eqref{e:mualfa}, implies
\begin{equation} \label{e:muxca} \mu_X(\b1_U)=\mu_X((\pi_\alpha)^*\b1_{\pi_\alpha(U)})=\mu_\alpha(\b1_{\pi_\alpha(U)})=
\frac{1}{|X_\alpha|}\sum_{x\in X_\alpha}\delta_x(\b1_{\pi_\alpha(U)})=\frac{|\pi_\alpha(U)|}{|X_\alpha|}\;.\end{equation}
\noindent{\bf \ref{r:msrx}.3.} In general, we cannot expect $\mu_X$ to be a measure: indeed, it is well known that $\mu_{\Z_p}$ is a distribution and not a measure when $R=\Q_p$ (see below for a proof). However $\mu_X$ is always a measure in the cases of most interest for this paper, with $R$ either $\R$ or $\C$, as the following proposition proves.  \end{rmks}

\begin{prop} \label{p:dstrms} Assume $\mu_X$ exists. It is a measure only if the set $\big\{\frac{1}{|X_\alpha|}\big\}_{\alpha\in\calj}$ is bounded in $R$. This condition is also sufficient if $R$ is complete. \end{prop}

\begin{prf} The last statement is a straightforward consequence of formula \eqref{e:muxca} and Lemma \ref{l:msrdstrb} (using the fact that $\N$ is bounded in $R$ if $|\cdot|_R$ is non-archimedean, while in the archimedean setting the inequalities $0\le|\pi_\alpha(U)|\le|X_\alpha|$ do the job). As for necessity, assume that $\{|X_\alpha|^{-1}\}_{\alpha\in\calj}$ is unbounded. Then there exists a sequence $(\alpha_n)_{n\in\N}\subseteq\calj$ such that $|X_{\alpha_n}|=c_n$ with 
\begin{equation} \label{e:cnvrg0} |c_n|_R<\frac{1}{n}\;. \end{equation}
Fix $z\in X$ and consider 
$$f_n:=c_n\b1_{\pi_{\alpha_n}^{-1}(\pi_{\alpha_n}(z))}=g_n\circ\pi_{\alpha_n}$$ 
with $g_n=c_n\b1_{\pi_{\alpha_n}(z)}$. The sequence $(f_n)_{n\in\N}$ converges to $0$ in $\calc(X,R)$, by \eqref{e:cnvrg0}. However, reasoning as in \eqref{e:muxca} shows
$$\mu_X(f_n)=\mu_X((\pi_{\alpha_n})^*g_n)=\mu_{\alpha_n}(g_n)=\frac{1}{|X_{\alpha_n}|}\sum_{x\in X_{\alpha_n}}\delta_x\big(c_n\b1_{\pi_{\alpha_n}(z)}\big)=\frac{1}{c_n}\cdot c_n=1$$
for every $n$. Thus $\mu_X$ cannot extend to a continuous functional on $\calc(X,R)$.  \end{prf} 

\subsubsection{The Hecke submodule} Assume that $\mu_X$ exists (for example because $X$ is a group). Then we can use \eqref{e:heckeprod} to define a map
$$\calh\colon\call_c(X,R)\longrightarrow\cald(X,R)$$
by $f\mapsto f\cdot\mu_X$\,. We define the {\em Hecke submodule} to be the image of $\calh$. The equality
\begin{equation} \label{e:Hck} (\pi_\alpha)_*\big(\calh(\b1_{\pi_\alpha^{-1}(x)})\big)=(\pi_\alpha)_*(\b1_{\pi_\alpha^{-1}(x)}\cdot\mu_X)\stackrel{\ast}{=} \b1_x\cdot(\pi_\alpha)_*(\mu_X)=\b1_x\cdot\mu_\alpha=\frac{1}{|X_\alpha|}\,\delta_x \end{equation}
(where $\ast$ comes from the fact that $\b1_{\pi_\alpha^{-1}(x)}$ factors through $X_\alpha$) implies that the Hecke submodule surjects onto $\cald(X_\alpha,R)$ for every $\alpha$ and hence is dense in $\cald(X,R)$.

\begin{rmk} As already observed, when $X$ is a group, the $R$-module $\cald(X,R)$ becomes an algebra with the convolution product. Readers might be familiar with the case $X=\GL_n(\adl_K^{\infty})$, where $K$ is a global field and $\adl_K^{\infty}$ the non-archimedean part of its adele ring: in this case, the Hecke submodule is the  Hecke algebra, as appearing in the theory of automorphic forms. \end{rmk}

The distribution functor is covariant. In fact, if $\phi\colon X\rightarrow Y$ is a continuous map of (pro)finite sets, it induces $\phi_*\colon\cald(X,R)\rightarrow\cald(Y,R)$ by $\phi_*(\mu)(f) =\mu(f\circ\phi)$ (because $f\circ\phi$ is locally constant if so is $f$). However, restriction to the Hecke submodule yields a contravariant functor: in particular, this allows us to lift distributions from $X_\alpha$ to $X$. The idea is summarized by the commutative diagram 
$$\begin{CD} \call_c(X,R)@>{\calh}>> \cald(X,R) \\ @A{(\pi_\alpha)^*}AA @VV{(\pi_\alpha)_*}V \\ \calc(X_\alpha,R) @>>\calh_\alpha> \cald(X_\alpha,R) \end{CD}$$
(where $\calh_\alpha$ is defined in the obvious way). The equality $(\pi_\alpha)_*\circ\calh\circ(\pi_\alpha)^*=\calh_\alpha$ is implicit in \eqref{e:Hck}, since $\b1_{\pi_\alpha^{-1}(x)}=(\pi_\alpha)^*(\b1_x)$\,.

\subsubsection{A topological limit} \label{sss:tplglm} As above, assume that $\mu_X$ exists. For a nonempty $S\subseteq X$, let
$$\mu_{S,\alpha}:=\frac{1}{|\pi_\alpha(S)|}\sum_{x\in\pi_\alpha(S)}\delta_x\in\cald(X_\alpha,R)\,.$$
We also put $\mu_{\emptyset,\alpha}:=0$. If $S$ is closed we have $S=\liminv\pi_\alpha(S)$. However, there is no reason why the counting distributions $\mu_{S,\alpha}$ should satisfy condition \eqref{e:condhaar} and form an inverse system. Define 
\begin{equation} \label{e:dmusa} \widetilde\mu_{S,\alpha}:=\frac{|X_\alpha|}{|\pi_\alpha(S)|}\sum_{x\in\pi_\alpha(S)}\b1_{\pi_\alpha^{-1}(x)}\cdot\mu_X
=\frac{|X_\alpha|}{|\pi_\alpha(S)|}\b1_{\pi_\alpha^{-1}(\pi_\alpha(S))}\cdot\mu_X\in\cald(X,R)\,, \end{equation}
so that 
\begin{align*} (\pi_\alpha)_*(\widetilde\mu_{S,\alpha}) & = \frac{|X_\alpha|}{|\pi_\alpha(S)|}(\pi_\alpha)_*(\b1_{\pi_\alpha^{-1}(\pi_\alpha(S))}\cdot\mu_X)=\frac{|X_\alpha|}{|\pi_\alpha(S)|} \;\b1_{\pi_\alpha(S)}\cdot\mu_\alpha= \\
& = \frac{|X_\alpha|}{|\pi_\alpha(S)|}\;\frac{1}{|X_\alpha|}\sum_{x\in\pi_\alpha(S)}\delta_x=\mu_{S,\alpha}\,.\end{align*}
In particular, if $U$ is $\alpha$-saturated one has
\begin{equation} \label{e:prmusa} \widetilde\mu_{S,\alpha}(\b1_U)=\mu_{S,\alpha}(\b1_{\pi_\alpha(U)})=\frac{|\pi_\alpha(U)\cap\pi_\alpha(S)|}{|\pi_\alpha(S)|} \end{equation}
by the same reasoning as used in \eqref{e:muxca}.

\begin{dfn} \label{d:muS2} Let $S$ be a subset of $X$. The {\em procounting distribution} $\mu_S$ attached to $S$ is the limit (if it exists)
\begin{equation} \label{e:dmus2} \mu_S=\lim_{\alpha\in\calj}\widetilde\mu_{S,\alpha} \end{equation}
of the net $(\widetilde\mu_{S,\alpha})_\alpha$ in $\cald(X,R)$. \end{dfn}

The limit in \eqref{e:dmus2} is taken with respect to the topology of $\cald(X,R)$, as discussed in Subsection \ref{sss:tplgd}. Thus $\mu_S=\lim_\alpha\widetilde\mu_{S,\alpha}$ means that for every neighbourhood $\calu$ of $\mu_S$ one can find an index $\alpha_0$ such that $\widetilde\mu_{S,\alpha}\in\calu$ if $\alpha\ge\alpha_0$. By Proposition \ref{p:dxtplg}, it is enough to check the convergence of the values $\widetilde\mu_{S,\alpha}(\b1_U)$ for every compact open $U\subseteq X$.

\begin{rmk} Definitions \ref{d:muS1} and \ref{d:muS2} are compatible: if $\mu_X$ exists in the sense of Definition \ref{d:muS1} then it is also the limit of the net $\widetilde\mu_{X,\alpha}$\,. Indeed, \eqref{e:dmusa} yields $\widetilde\mu_{X,\alpha}=\mu_X$ for every $\alpha$. \end{rmk}

\begin{eg} If $S=\{x\}$ then $\mu_S$ exists and is the Dirac delta at $x$, because if $U$ is $\alpha$-saturated then \eqref{e:prmusa} implies $\tilde\mu_{\{x\},\beta}(\b1_U)=\delta_x(\b1_U)$ for every $\beta\ge\alpha$. \end{eg}

\begin{lem} \label{l:muS-chs} Let $\overline S$ be the closure of $S\subseteq X$. Then $\mu_S$ exists if and only if so does $\mu_{\overline S}$\,. Moreover, these two distributions are equal. \end{lem}

\begin{prf} One has $\pi_\alpha(S)=\pi_\alpha(\overline S)$ by Lemma \ref{l:clsr} and then \eqref{e:dmusa} yields $\widetilde\mu_{S,\alpha}=\widetilde\mu_{\overline S,\alpha}$\,. \end{prf}

Because of Lemma \ref{l:muS-chs}, in the following we shall mostly consider closed $S$. If $S$ is also open, the situation is particularly nice, as the next result shows.

\begin{lem} \label{l:mucmpap} Let $U\subseteq X$ be compact open. Then $\widetilde\mu_{U,\alpha}$ converges to
\begin{equation} \label{e:mucmpap} \mu_U=\frac{\b1_U}{\mu_X(\b1_U)}\cdot\mu_X\;. \end{equation}
\end{lem}

\begin{prf} By Corollary \ref{c:dl213}, we can assume $U=\pi_\gamma^{-1}(\pi_\gamma(U))$ for some $\gamma\in\calj$. Then $\pi_\alpha(U)=(\pi_\gamma^\alpha)^{-1}(\pi_\gamma(U))$ holds for $\alpha\ge\gamma$. Since $\mu_X$ exists, all fibres of the transition maps $\pi_\alpha^\beta$ have the same cardinality, by Proposition \ref{p:hrl}, and \eqref{e:condhaar} yields
$$\frac{|X_\alpha|}{|X_\gamma|}=\frac{|\pi_\alpha(U)|}{|\pi_\gamma(U)|}\;.$$
showing that the ratio $|\pi_\alpha(U)|/|X_\alpha|$ is constant (and equal to $\mu_X(\b1_U)$, as explained in Remark \ref{r:msrx}.2) for $\alpha\ge\gamma$. Therefore \eqref{e:dmusa} becomes $\widetilde\mu_{U,\alpha}=c_U\b1_U\cdot\mu_X$, where $c_U=1/\mu_X(\b1_U)$, for any such $\alpha$. \end{prf}

We conclude with an example where $\mu_S$ does not exist.

\begin{eg} \label{eg:nonex} Let $p$ be an odd prime and take $X=\Z_p$ (with defining maps $\pi_n\colon\Z_p\rightarrow\Z/p^n\Z$). Define $S=S_0\cup S_1$ by
$$S_0:=\left\{\sum_{k=0}^\infty a_kp^k\mid a_k=0\text{ if }k\text{ is even },a_k\in\{1,\dots,p-1\}\text{ if }k\text{ is odd}\right\}$$
and
$$S_1:=\left\{\sum_{k=0}^\infty a_kp^k\mid a_k=0\text{ if }k\text{ is odd },a_k\in\{1,\dots,p-1\}\text{ if }k\text{ is even}\right\}\,.$$
(Note that $S_0$ and $S_1$ are both closed sets.) Thus one has $\pi_1(S_0)=\{\pi_1(0)\}$, $\pi_1(S_1)=(\Z/p\Z)^*$ and, for $x\in\pi_n(S_0)$,
\begin{equation}\label{e:egfiber} |(\pi_n^{n+1})^{-1}(x)|=\begin{cases} p-1 & \text{ if }n\text{ is odd}\\ 1 & \text{ if }n\text{ is even}\end{cases} \end{equation}
while the reverse holds for $x\in\pi_n(S_1)$. A simple induction then shows
$$|\pi_n(S_0)|=(p-1)^{\lfloor n/2\rfloor}\;\;\text{ and }\;\;|\pi_n(S_1)|=(p-1)^{\lfloor (n+1)/2\rfloor}\,.$$
Since $\pi_n(S)\cap\pi_n(p\Z_p)=\pi_n(S_0)$, formula \eqref{e:prmusa} yields
\begin{equation} \label{e:egnonex} \widetilde\mu_{S,n}(\b1_{p\Z_p})=\frac{|\pi_n(S_0)|}{|\pi_n(S)|}=
\begin{cases}\displaystyle\frac{1}{p} & \text{ if }n\text{ is odd},\vspace{2pt} \vspace{3pt}\\ \displaystyle\frac{1}{2} & \text{ if }n\text{ is even}.\end{cases} \end{equation}
Therefore the distributions $\widetilde\mu_{S,n}$ do not converge (independently of the choice of the coefficient ring $R$). \end{eg}

\begin{rmk} \label{r:union} By formula \eqref{e:egfiber} and its analogue for $S_1$, Proposition \ref{p:hrl}  implies that the distributions $\mu_{S_0}$ and $\mu_{S_1}$ exist. Therefore, Example \ref{eg:nonex} shows that, for two disjoint sets $A,B\subset X$, the existence of both $\mu_A$ and $\mu_B$ does not yield the existence of the distribution $\mu_{A\cup B}$.

The example also illustrates how sensitive pro-counting distributions are to the choice of the maps defining the inverse system: indeed, $\mu_S$ exists if one takes as inverse system either $\big(\Z/p^{2k}\Z,\pi_{2k}^{2k+2}\big)_{k\in\N}$ or $\big(\Z/p^{2k+1}\Z,\pi_{2k+1}^{2k+3}\big)_{k\in\N}$, because the quantity $|(\pi_n^{n+2})^{-1}(x)|$ is constant on $\pi_n(S)$, for every $n$. However, these choices yield two different distributions, as shown by \eqref{e:egnonex}.  \end{rmk}

\subsubsection{The ambient space} \label{sss:ambspc} In \S\ref{sss:tplglm}, the distribution $\mu_S$ was defined assuming the existence of $\mu_X$. Now we will show that this definition is in fact independent of the ambient space.\\

If $S\subseteq X$ is closed then one has $\displaystyle S=\varprojlim_{\alpha\in\calj}\pi_\alpha(S)$. The inclusions $\iota_\alpha\colon\pi_\alpha(S)\hookrightarrow X_\alpha$ induce
$$\cald(\pi_\alpha(S),R)\hookrightarrow\cald(X_\alpha,R)$$
for all $\alpha\in\calj$. Thus, taking the limit, $\iota_*\colon\cald(S,R)\hookrightarrow\cald(X,R)$.

\begin{lem} \label{l:clsemb} In the setting above, the map $\iota_*$ is a closed embedding, with image
$$\iota_*\big(\cald(S,R)\big)=\big\{\mu\in\cald(X,R)\mid\mu(\b1_U)=0\text{ if }U\cap S=\emptyset\big\}$$
where $U$ varies among all the compact open subsets of $X$. \end{lem}

\begin{prf} Lemma \ref{l:darflx} immediately yields that the image of $(\iota_\alpha)_*$ is the space
$$\big\{\mu\in\cald(X_\alpha,R)\mid\mu(\b1_x)=0\text{ if }x\notin\pi_\alpha(S)\big\}$$
with basis $\{\delta_x\mid x\in\pi_\alpha(S)\}$. Hence the maps $(\iota_\alpha)_*$ are all closed embeddings. This implies the same for $\iota_*$\,.

Let $\mu\in\iota_*\big(\cald(S,R)\big)$ and fix a compact open $U\subseteq X$. Since $S$ is closed, Lemma \ref{l:clsr} shows that $U\cap S=\emptyset$ holds only if there is some $\alpha\in\calj$ such that $\pi_\alpha(S)\cap\pi_\alpha(U)=\emptyset$. Replacing, if needed, $\alpha$ with a bigger index, we can assume that $U$ is $\alpha$-saturated: but then
$$\mu(\b1_U)=\mu(\b1_{\pi_\alpha^{-1}(\pi_\alpha(U))})=(\pi_\alpha)_*(\mu)(\b1_{\pi_\alpha(U)})=0,$$
because $(\pi_\alpha)_*(\mu)$ is in  $\iota_*\big(\cald(\pi_\alpha(S),R)\big)$.

Vice versa, if $\mu(\b1_U)=0$ for every compact open $U$ disjoint from $S$, it follows 
$$(\pi_\alpha)_*(\mu)(\b1_x)=\mu\big(\b1_{\pi_\alpha^{-1}(x)}\big)=0$$ 
for all $x\notin\pi_\alpha(S)$. This shows that $(\pi_\alpha)_*(\mu)$ is in the image of $(\iota_\alpha)_*$ for all $\alpha$ and hence $\mu\in\cald(S,R)$. \end{prf}

In the following, we shall identify $\cald(S,R)$ with its image via $\iota_*$\,. Recalling Definition \ref{d:spprt}, Lemma \ref{l:clsemb} can be reformulated as the statement that $\cald(S,R)$ consists of those distributions having support in $S$.

\begin{cor} \label{c:musintr} Let $S\subseteq X$. The distribution $\mu_S$ exists if and only if
\begin{equation} \label{e:dmus3} \mu_S(\b1_V)=\lim_\alpha\frac{|\pi_\alpha(V)|}{|\pi_\alpha(S)|} \end{equation}
holds for every $V\subseteq \overline S$ which is compact open in the induced topology. Moreover, $\mu_S$ is in $\cald(\overline S,R)$. \end{cor}

\begin{prf} By Lemma \ref{l:muS-chs}, we may replace $S$ by $\overline S$, since $\pi_\alpha(S)=\pi_\alpha(\overline S)$ for every $\alpha$, and $\mu_S=\mu_{\overline S}$ whenever either side exists. Thus we may assume $S$ closed.

If $V\subseteq S$ is a  compact open, then there is $U$ compact open in $X$ such that $V=U\cap S$. By Definition~\ref{d:muS2}, $\mu_S$ exists if and only if the net $\widetilde\mu_{S,\alpha}(\b1_U)$ converges for every such $U$, and then the limit is $\mu_S(\b1_U)$. This condition is equivalent to \eqref{e:dmus3} by \eqref{e:prmusa}, since
$U=\pi_\beta^{-1}(\pi_\beta(U))$ for some $\beta$ (Corollary \ref{c:dl213}) and hence $\pi_\alpha(V)=\pi_\alpha(U)\cap\pi_\alpha(S)$ if $\alpha\ge\beta$. This also shows $\mu_S(\b1_U)=0$ if $S\cap U=\emptyset$, and thus $\mu_S\in\cald(S,R)$ by Lemma~\ref{l:clsemb}. \end{prf}

\begin{rmk} Corollary \ref{c:musintr} makes it clear that $\mu_S$ is independent of the ambient space $X$. The intrinsic characterization of $\mu_S$ provided by \eqref{e:dmus3} could be used as definition of the procounting distribution attached to $S$, in alternative to \eqref{e:dmus2}. We have chosen to start with the former because in the situations of interest to us there is always a natural ambient space $X$ such that $\mu_X$ exists in the sense of Definition \ref{d:muS1}.  \end{rmk}

The assumption about the existence of $\mu_X$ can be made under mild conditions: we show that, subject to our hypotheses, any profinite space can be embedded into a bigger one for which the procounting distribution exists.

\begin{prop} \label{p:XY} Assume $\calj$ contains a countable cofinal chain $\calj_0=\{\alpha_n\}_{n\in\N}\subseteq\calj$. Then there exists a profinite space $Y$ such that it contains $X$ as a closed subset and $\mu_Y$ exists (in the sense of Definition \ref{d:muS1}). \end{prop}

\begin{skprf} We use the shortenings $X_n$ for $X_{\alpha_n}$ and so on.

The idea is to construct $Y$ by embedding each $X_n$ into a set $Y_n$ so that for every $n$ the fibres of $\pi_n^{n+1}\colon Y_{n+1}\rightarrow Y_n$ have all the same cardinality. This can be achieved starting with $Y_0=X_0$ and then defining recursively $Y_n$ by adding enough points to $X_n$. The existence of $\mu_Y$ then follows from Proposition \ref{p:hrl}. \end{skprf}

As an (in)conclusive remark, we expect that Proposition \ref{p:XY} may be used to expand the scope of our theory beyond arithmetic (analysis on graphs seems to be a plausible candidate). 

\subsubsection{The procounting distribution of a product} \label{sss:prcntprd} Let $X=\prod_{\kappa\in\calk}Y_\kappa$ be as in Subsection \ref{sss:prdtt} and consider a subset of the form $S=\prod_\kappa T_\kappa$, where $T_\kappa\subseteq Y_\kappa$\,. By Lemma \ref{l:prdst}, $X$ is the limit of $X_\alpha$\,, with indexes $\alpha=(\alpha_\kappa)_{\kappa\in K}\in\bigcup_K\prod_{\kappa\in K}\calj_\kappa$. One has
$$\pi_\alpha(S)=\prod_{\kappa\in K}\pi_{\alpha_\kappa}(T_\kappa)$$
and hence the isomorphism of Lemma \ref{l:prddstrbf} yields
\begin{equation} \label{e:prdtns} \bigotimes_{\kappa\in K}\mu_{T_\kappa,\alpha_\kappa}=\bigotimes_{\kappa\in K}\left(\frac{1}{|\pi_{\alpha_\kappa}(T_\kappa)|}\sum_{y\in\pi_{\alpha_\kappa}(T_\kappa)}\delta_y\right)\;\mapsto\;\frac{1}{|\pi_\alpha(S)|}\sum_{x\in\pi_\alpha(S)}\delta_x=\mu_{S,\alpha}\,. \end{equation}

\begin{prop} \label{p:prddstr2} Assume $X=\prod_{\kappa\in\calk}Y_\kappa$, as in Theorem \ref{t:dstrprd}, and let $S=\prod_\kappa T_\kappa$ with $T_\kappa\subseteq Y_k$\,. If the procounting distribution $\mu_{T_\kappa}$ exists for every index $\kappa$, then $\mu_S$ also exists and is the product of the $\mu_{T_\kappa}$\!'s. \end{prop}

\begin{prf} Since, by definition, $\mu_{T_\kappa}(\b1_{T_\kappa})=1$ holds for every $\kappa$, Theorem \ref{t:dstrprd} ensures the existence of the distribution $\bigotimes_\kappa\mu_{T_\kappa}\in\cald(X,R)$ and we just have to check that this product satisfies Definition \ref{d:muS2}, i.e., that the equality
\begin{equation} \label{e:lmtmprd} \bigotimes_\kappa\mu_{T_\kappa}(\b1_U)=\lim_{\alpha}\widetilde\mu_{S,\alpha}(\b1_U) \end{equation}
holds for every compact open $U\subseteq X$. We can write $U$ as in \eqref{e:Uprd}, so that \eqref{e:tnsr2} yields
\begin{equation} \label{e:prdtns1} \bigotimes_\kappa\mu_{T_\kappa}(\b1_U)=\prod_{\kappa\in K}\mu_{T_\kappa}(\b1_{U_\kappa})\cdot\prod_{\kappa\in\calk-K}\mu_{T_\kappa}(\b1_{Y_\kappa})=\prod_{\kappa\in K}\mu_{T_\kappa}(\b1_{U_\kappa})=\prod_{\kappa\in K}\lim_{\alpha_\kappa\in\calj_\kappa}\widetilde\mu_{T_\kappa,\alpha_\kappa}(\b1_{U_\kappa}) \end{equation}
where the last equality is the definition of $\mu_{T_\kappa}$. On the other hand, in computing the right-hand side of \eqref{e:lmtmprd}, we can restrict to those indexes  $\alpha=(\alpha_\kappa)_{\kappa\in K'}$ such that $K'$ contains $K$. For such an $\alpha$ one has
\begin{equation} \label{e:prdtns2} \widetilde\mu_{S,\alpha}(\b1_U)=\mu_{S,\alpha}(\b1_{\pi_\alpha(U)})\stackrel{*}{=}\prod_{\kappa\in K}\frac{|\pi_{\alpha_k}(U)\cap\pi_{\alpha_\kappa}(T_\kappa)|}{|\pi_{\alpha_\kappa}(T_\kappa)|}\cdot\prod_{\kappa\in K'-K}1=\prod_{\kappa\in K}\mu_{T_\kappa,\alpha_\kappa}(\b1_{\pi_{\alpha_k}(U)})\,.
\end{equation}
where equality $*$ uses \eqref{e:prdtns} and $\pi_{\alpha_\kappa}(U)=Y_{\kappa,\alpha_\kappa}$ for $\kappa\notin K$. As a finite product of convergent factors, the most right-hand term of \eqref{e:prdtns2} has a limit as $\alpha$ grows. This proves that $\mu_S$ exists; as for \eqref{e:lmtmprd}, it follows from comparing \eqref{e:prdtns1} with \eqref{e:prdtns2}. \end{prf}

\subsection{Procounting measures} From now on, we take $\R$, with the usual absolute value, as our ring of coefficients $R$, and assume that $X$ satisfies condition \eqref{e:condhaar}. Then $\mu_X$ exists and, by Proposition \ref{p:dstrms}, can be extended to a positive functional on $\calc(X,\R)$. Hence, by the Riesz representation theorem \cite[2.14]{rudin1}, it defines a regular Borel measure (i.e., a Radon measure), which we shall still denote as $\mu_X$. Accordingly to this double interpretation, we shall write $\mu_X(f)$ for $f$ a continuous function and $\mu_X(A)$ for $A$ an element in the Borel $\sigma$-algebra. In particular, any closed set $C\subseteq X$ is measurable and one has $\mu_X(C)=\inf_U\mu_X(\b1_U)$, where $U$ varies among compact open subsets containing $C$. By Lemma \ref{l:clsr} this becomes
\begin{equation} \label{e:musc1} \mu_X(C)=\inf_{\alpha\in\calj}\mu_X\big(\b1_{\pi_\alpha^{-1}(\pi_\alpha(C))}\big)=\lim_{\alpha\in\calj}\mu_X\big(\b1_{\pi_\alpha^{-1}(\pi_\alpha(C))}\big)=\lim_{\alpha\in\calj}\frac{|\pi_\alpha(C)|}{|X_\alpha|} \end{equation}
(with the last equality coming from \eqref{e:muxca}, while the coincidence between infimum and limit is because the sets $\pi_\alpha^{-1}(\pi_\alpha(C))$ decrease as $\alpha$ increases). Note also that \eqref{e:musc1} can fail if $C$ is not closed (for example, take $X=\Z_p$, $C=\N$ and recall that the Haar measure of a countable subset of $\Z_p$ is $0$).
 
\begin{rmk} More generally, given $S\subseteq X$, the existence of $\mu_S$ as in Definition \ref{d:muS2} implies it is a Radon measure on $X$ (with support inside $\overline S$) , by the same reasoning as above: $0\le\mu_S(\b1_U)\le1$ holds for every compact open $U$ by \eqref{e:prmusa} and thus Lemma \ref{l:msrdstrb} yields $\mu_S\in\calm(X,\R)$. Moreover, for $C\subseteq S$ closed, the analogue of \eqref{e:musc1} leads to
\begin{equation} \label{e:musc} \mu_S(C)=\lim_{\alpha\in\calj}\mu_S\big(\b1_{\pi_\alpha^{-1}(\pi_\alpha(C))}\big)=\lim_{\alpha\in\calj}\lim_{\beta\ge\alpha}\dfrac{|\pi_\beta(C)\cap\pi_\beta(S)|}{|\pi_\beta(S)|}\,, \end{equation}
with the last equality coming from \eqref{e:dmus2} and \eqref{e:prmusa}. In particular, \eqref{e:musc} shows that if $S$ is infinite then every finite set has measure $0$.  \end{rmk}

Having $\R$ as ring of scalars, we obtain the following extension of Lemma \ref{l:mucmpap}.

\begin{prop} \label{p:msrS>0} The procounting measure $\mu_S$ exists for every $S\subseteq X$ satisfying $\mu_X\big(\overline S\big)>0$. For such $S$ one has
\begin{equation} \label{e:msrs>0} \mu_S=\frac{\b1_{\overline S}}{\mu_X\big(\overline S\big)}\cdot\mu_X \,.\end{equation}
\end{prop}

\begin{prf} By Corollary \ref{c:musintr}, for the existence of $\mu_S$ it is enough to show that the ratios
\begin{equation} \label{e:trpprd} \frac{|\pi_\alpha(U)\cap\pi_\alpha(S)|}{|\pi_\alpha(S)|}=\frac{|\pi_\alpha(U)\cap\pi_\alpha(S)|}{|\pi_\alpha(U)|}\cdot
\frac{|\pi_\alpha(U)|}{|X_\alpha|}\cdot\frac{|X_\alpha|}{|\pi_\alpha(S)|} \end{equation}
converge for every compact open $U\subseteq X$. By \eqref{e:musc1} one sees that the three factors on the right-hand side of \eqref{e:trpprd} have limits, respectively, $\mu_U\big(\overline S\big)$, $\mu_X(U)$ and $1/\mu_X\big(\overline S\big)$ (the existence of $\mu_U$ and the fact that \eqref{e:musc1} applies to it both follow from Lemma \ref{l:mucmpap}).

As for \eqref{e:msrs>0}, by Lemma \ref{l:msrinjd} (and Subsection \ref{sss:ambspc}) it suffices to check it as an equality of distributions: that is, we need to prove
\begin{equation} \label{e:vlrmsrsu} \mu_S(\b1_U)\stackrel?=\left(\frac{\b1_{\overline S}}{\mu_X\big(\overline S\big)}\cdot\mu_X\right)(\b1_U):=\frac{\mu_X\big(U\cap\overline S\big)}{\mu_X\big(\overline S\big)} \end{equation} 
for every compact open $U\subseteq X$. Now note that \eqref{e:mucmpap} yields
$$\mu_U\big(\overline S\big)=\left(\frac{\b1_U}{\mu_X(\b1_U)}\,\mu_X\right)\!\big(\overline S\big)=\frac{\mu_X\big(U\cap\overline S\big)}{\mu_X(U)}\;.$$
Also, the right-hand side of \eqref{e:trpprd} converges to $\mu_S(\b1_U)$, by \eqref{e:dmus3}. Therefore the limit of \eqref{e:trpprd} becomes
$$\mu_S(\b1_U)=\mu_U\big(\overline S\big)\cdot\mu_X(\b1_U)\cdot\frac{1}{\mu_X\big(\overline S\big)}=\frac{\mu_X\big(U\cap\overline S\big)}{\mu_X\big(\overline S\big)}$$
as claimed. \end{prf}

\begin{rmk} In particular, $\mu_S$ exists if $\overline S$ has nonempty interior, because \eqref{e:condhaar} implies $\mu_X(A)>0$ for every nonempty open $A$. This will be crucial in the proof of Proposition \ref{p:dernot0}. On the other hand, a procounting measure usually ``fails to see'' isolated points (as we are going to discuss in Section \ref{sss:isltpnt}), so $\mu_S(A)$ need not be positive if $A$ is open as a subset of $\overline S$. \end{rmk} 

\begin{prop} \label{p:brd0} Let $\partial_XS$ denote the boundary of $S\subseteq X$. If $\mu_X(\partial S)=0<\mu_X\big(\overline S\big)$ then one has
\begin{equation} \label{e:muslmtu} \mu_S=\lim_{U\subseteq S}\mu_U \end{equation}
where the limit is taken over all compact open subsets of $S$ (ordered by inclusion). \end{prop}

\begin{prf} Proposition \ref{p:msrS>0} ensures the existence of $\mu_S$. The conditions imply that $S$ is $\mu_X$-measurable and its interior is nonempty. Assume $U\subseteq S$. Lemma \ref{l:mucmpap} yields, for any compact open $V$,
$$\mu_U(\b1_V)=\frac{\mu_X(V\cap U)}{\mu_X(U)}\;.$$
As $U$ increases, so do $\mu_X(V\cap U)$ and $\mu_X(U)$, so they converge to some limits and, by \eqref{e:msrs>0}  and \eqref{e:vlrmsrsu}, we just need to check that these limits are respectively $\mu_X(V\cap\overline S)$ and $\mu_X(\overline S)$.
 
For any $\alpha\in\calj$, let
$$A_\alpha:=\{x\in X_\alpha\mid\pi_\alpha^{-1}(x)\subseteq S\}$$
and $B_\alpha:=\pi_\alpha(S)-A_\alpha$\,. Then $A:=\cup_\alpha\pi_\alpha^{-1}(A_\alpha)$ is the interior of $S$, while $\cap_\alpha\pi_\alpha^{-1}(B_\alpha)=\partial S$. Put $U_\alpha:=\pi_\alpha^{-1}(A_\alpha)$ and $W_\alpha:=\pi_\alpha^{-1}(B_\alpha)$. The $U_\alpha$'s form a cofinal subsystem of the compact open subsets of $S$, hence one can restrict to them for computing the limits. The hypotheses imply
$$0=\mu_X\big(\partial S \big)=\lim_{\alpha\in\calj}\mu_X(W_\alpha)=\lim_{\alpha\in\calj}\frac{|B_\alpha|}{|X_\alpha|}$$
(because $\beta\ge\alpha$ implies $B_\beta\subseteq B_\alpha$ and \eqref{e:muxca} applies to each $W_\alpha$), hence, by \eqref{e:musc1},
$$\mu_X\big(\overline S\big)=\lim_{\alpha\in\calj}\frac{|A_\alpha|+|B_\alpha|}{|X_\alpha|}=\lim_{\alpha\in\calj}\frac{|A_\alpha|}{|X_\alpha|}=\lim_{\alpha\in\calj}\mu_X(U_\alpha)\,.$$
Similarly, one has
$$\lim_{\alpha\in\calj}\mu_X(V\cap U_\alpha)=\lim_{\alpha\in\calj}\frac{|\pi_\alpha(V)\cap A_\alpha|}{|X_\alpha|}=\lim_{\alpha\in\calj}\frac{|\pi_\alpha(V)\cap A_\alpha|+|\pi_\alpha(V)\cap B_\alpha|}{|X_\alpha|}=\mu_X(V\cap\overline S)$$
(where the first equality uses $\mu_X(V\cap U_\alpha)=|\pi_\alpha(V)\cap A_\alpha|/|X_\alpha|$, which holds for $\alpha$ big enough, since then $V$ is $\alpha$-saturated). \end{prf}

\subsubsection{Isolated points and support} \label{sss:isltpnt} For a Radon measure $\mu$ on $X$, the classical notion of support (the set of those $x\in X$ such that $\mu(A)>0$ for every $A$ open neighbourhood of $x$) coincides with the one introduced in Definition \ref{d:spprt}.

\begin{lem}\label{l:isolated} Let $S\subseteq X$ be such that $\mu_S$ exists. If $S$ is infinite, then the isolated points of $S$ do not lie in the support of $\mu_S$. \end{lem}

\begin{prf} If $x$ is an isolated point of $S$, by definition there exists a compact open $U_0$ not containing any other point of $S$. Without loss of generality, we assume $U_0=\pi_\alpha^{-1}(y)$, for $y\in X_\alpha$. If $\beta >\alpha$, then \eqref{e:prmusa} yields
\[ \widetilde{\mu}_{S,\beta}(\b1_U)\le\dfrac{1}{|\pi_\beta(S)|}  \]
for any compact open $U\subseteq U_0$. As $S$ is infinite, we have $\lim_\beta|\pi_\beta(S)|=\infty$. \end{prf}

\begin{eg+} \label{eg:ptaccm} Let us consider some procounting distribution attached to countably infinite sets with few limit points.\\
{\bf \ref{eg:ptaccm}.1.} First we consider the case when $S$ has one limit point. We can write $\overline S$ as a union $\{\ell\}\cup\{a_n\}_{n\in\N}$, where the sequence $\{a_n\}$ has no repeated values and converges to $\ell$. We claim that in this case one has $\mu_S=\delta_\ell$. Indeed, let $U$ be a compact open subset of $X$. If $\ell\notin U$, then $|S\cap U|$ is finite and so, for $\alpha$ big enough,
\begin{equation} \label{e:eg0accm} \widetilde\mu_{S,\alpha}(\b1_U)=\dfrac{|\pi_\alpha(S)\cap\pi_\alpha(U)|}{|\pi_\alpha(S)|}\le\dfrac{|S\cap U|}{|\pi_\alpha(S)|} \end{equation}
converges to $0$. On the other hand, if $\ell\in U$, then there is $n_U$ such that $n>n_U$ implies $a_n\in U$, so that
\[ 1\ge\widetilde\mu_{S,\alpha}(\b1_U)=\dfrac{|\pi_\alpha(S)\cap\pi_\alpha(U)|}{|\pi_\alpha(S)|}\ge\dfrac{|\pi_\alpha(S)|-n_U}{|\pi_\alpha(S)|} \]
yields $\mu_S(\b1_U)=1$.

For a concrete instance of such behaviour, we can take $X=\varprojlim_{k\in\N}\Z/k!\Z\simeq\za$ and $S=\{n!:n\in\N\}$. Then $\overline S=S\cup\{0\}$, the projection $\pi_{k!}(S)$ consists of the $k$ classes $n!$ (mod $k!$), for $k>0$, $0\le n\le k$, and the reasoning above proves
$$\delta_0=\mu_S=\lim_{k\rightarrow\infty}\widetilde\mu_{S,k!}=\lim_{k\rightarrow\infty}\frac{1}{k}\sum_{n=0}^k\delta_{n!}\,.$$
{\bf \ref{eg:ptaccm}.2.} When $S$ has finitely many limit points $\ell_1,\dots,\ell_r$, with $r>1$, the only possible general statement is that $\mu_S$, if it exists, must be a linear combination $\sum_ic_i\delta_{\ell_i}$, with $\sum_ic_i=1$ and $0\le c_i\le1$ for all $i$ (as follows from Lemmata \ref{l:isolated} and \ref{l:fntspprt}). However, we cannot even guarantee the existence of $\mu_S$ in this generality, much less determine the coefficients $c_i$, as shown by the following reasoning.

Take $X=\varprojlim_{k\in\N}\Z/k!\Z$ as before and consider
$$S=\{n!:n\in A\}\cup\{n!+1: n\in B\},$$
where $A$ and $B$ define a partition of $\N$ in two infinite subsets, so that $T$ has limit points $\{0,1\}$ and closure $\overline S=S\cup\{0,1\}$. Moreover, $|\pi_{k!}(S)|$ is roughly $k$. Let $U\subseteq\za$ be compact open: then the same reasoning as in \eqref{e:eg0accm} yields $\lim_k\tilde\mu_{S,k!}(\b1_U)=0$ if $U\cap\{0,1\}$ is empty. On the other hand, the formula
\[ \tilde\mu_{S,k!}(\b1_U)=\frac{|\pi_{k!}(U)|}{|\pi_{k!}(S)|}\approx\begin{cases} \dfrac{|\{n\in A:n\le k\}|}{k} & \text{ if }0\in U\not\ni 1 \\ \\ 
\dfrac{|\{n\in B:n\le k\}|}{k} & \text{ if }1\in U\not\ni 0\end{cases} \]
shows that if $U$ contains exactly one of the limit points then $\tilde\mu_{S,k!}(\b1_U)$ converges to the asymptotic density of $A$ or $B$. But there is no need for this density to exist and, if it does, it could take any value in $[0,1]$. 

This construction also yields instances of different procounting measures sharing the same support. \end{eg+}

\subsubsection{Close pairs of subsets} \label{sss:clpr}

\begin{dfn} \label{d:clpr} We say that two subsets $S,T$ of $X$ {\em form a close pair} if 
\begin{equation} \label{e:dfclpr} \lim_{\alpha\in\calj}\frac{|\pi_\alpha(S)\triangle\pi_\alpha(T)|}{|\pi_\alpha(S\cup T)|}=0 \end{equation}
where $\triangle$ denotes the symmetric difference.\end{dfn}
 
\begin{rmk} Instead of checking \eqref{e:dfclpr}, it might be more convenient to consider the equivalent condition
\begin{equation} \label{e:clpr1} \lim_{\alpha\in\calj}\frac{|\pi_\alpha(S)\Delta\pi_\alpha(T)|}{|\pi_\alpha(S)\cap\pi_\alpha(T)|}=0. \end{equation}
(with the convention $n/0=\infty$ for every $n\in\N$). \end{rmk}

\begin{eg} \label{eg:clprbnl} Lemma \ref{l:clsr} implies that every $S\subseteq X$ forms a close pair with its own closure. Also, if $\mu_X(\partial S)=0<\mu_X(S)$, then $S$ and its interior form a close pair (as one sees comparing \eqref{e:dfclpr} with the computations in the proof of Proposition \ref{p:brd0}). \end{eg}

\begin{lem} \label{l:clprtrns} Forming a close pair is an equivalence relation among non-empty subsets of $X$. \end{lem}

\begin{skprf} Reflexivity and symmetry are obvious. It remains to show that if $S,T$ and $T,U$ form close pairs then so do $S,U$: that is, one has to check that $|\pi_\alpha(S)\cap\pi_\alpha(T)\cap\pi_\alpha(U)|$ grows much faster than the cardinality of its complement in $\pi_\alpha(S)\cup\pi_\alpha(T)\cup\pi_\alpha(U)$. This can be achieved by writing the latter set as a disjoint union of seven parts, according to the various inclusion relations, and then applying the hypotheses and \eqref{e:clpr1}. For example, one has 
$$\lim_{\alpha\in\calj}\frac{|\pi_\alpha(S)-(\pi_\alpha(T)\cup\pi_\alpha(U))|}{|(\pi_\alpha(S)\cap\pi_\alpha(T))-\pi_\alpha(U)|+|\pi_\alpha(S)\cap\pi_\alpha(T)\cap\pi_\alpha(U)|}=0$$
because $S$ and $T$ form a close pair. Verifying all details is simple but cumbersome and we leave the task to the reader. \end{skprf}

\begin{cor} \label{c:clprcls}  If $S$ and $T$ form a close pair, so do their closures. \end{cor}

\begin{prf} By Lemma \ref{l:clprtrns} and Example \ref{eg:clprbnl}. \end{prf} 

\begin{lem} \label{l:inclprssm} Let $A\subset B\subset C$ be a chain of subsets of $X$. If $A,C$ form a close pair, then so do $A,B$ and $B,C$.  \end{lem}

\begin{prf} Straightforward from \eqref{e:dfclpr}.\end{prf}

Our interest in close pairs comes from the following generalization of Lemma \ref{l:muS-chs}. 

\begin{prop} \label{p:muST} Let $S,T\subseteq X$ form a close pair. Then either they have the same procounting measure or both their procounting measures do not exist. \end{prop}

\begin{prf} For $\alpha\in\calj$, set $B_\alpha=\pi_\alpha(S)\cap\pi_\alpha(T)$, $A_\alpha=\pi_\alpha(S)- B_\alpha$ and $C_\alpha=\pi_\alpha(T)- B_\alpha$\,. Thus for a compact open $U\subseteq X$ we have
$$\mu_{S,\alpha}(\b1_{\pi_\alpha(U)})=\frac{|\pi_\alpha(U)\cap A_\alpha|}{|A_\alpha|+|B_\alpha|}+\frac{|\pi_\alpha(U)\cap B_\alpha|}{|A_\alpha|+|B_\alpha|}$$
and
$$\mu_{T,\alpha}(\b1_{\pi_\alpha(U)})=\frac{|\pi_\alpha(U)\cap C_\alpha|}{|C_\alpha|+|B_\alpha|}+\frac{|\pi_\alpha(U)\cap B_\alpha|}{|C_\alpha|+|B_\alpha|}$$
Moreover, the closeness hypothesis can be rewritten as
$$\lim_{\alpha\in\calj}\frac{|A_\alpha|+|C_\alpha|}{|A_\alpha|+|B_\alpha|+|C_\alpha|}=0\,,$$
which implies
$$\lim_{\alpha\in\calj}\frac{|A_\alpha|}{|B_\alpha|}=0=\lim_{\alpha\in\calj}\frac{|C_\alpha|}{|B_\alpha|}$$
and therefore
$$\lim_{\alpha\in\calj}\mu_{S,\alpha}(\b1_{\pi_\alpha(U)})=\lim_{\alpha\in\calj}\frac{|\pi_\alpha(U)\cap B_\alpha|}{|B_\alpha|}=\lim_{\alpha\in\calj}\mu_{T,\alpha}(\b1_{\pi_\alpha(U)})$$
if the limits exist. Since $U$ is $\beta$-saturated for some $\beta\in\calj$, these limits are the same as $\mu_S(\b1_U)$, $\mu_T(\b1_U)$ respectively, by \eqref{e:prmusa} and \eqref{e:dmus2}.\end{prf}

\subsubsection{A hierarchy of sets and measures} \label{sss:hier} A classical approach to estimating ``how big'' is an infinite discrete set $S$ (such as a subset of $\N$) is by means of a counting function $c_S$: in this way one can, for example, compare the size of $S$ and $T$ by looking at the order of growth of $c_S$ and $c_T$. We expect that the theory here developed can enjoy a similar flexibility and propose a definition in this direction.\\

Let $S,T$ be subsets of our profinite space $X$. We say that $S$ is {\em little-o} of $T$, and write $S=o(T)$, if
\begin{equation} \label{e:dfnSoT}  \lim_{\alpha\in\calj}\frac{|\pi_\alpha(S)|}{|\pi_\alpha(T)|}=0. \end{equation}
Note that if $S=o(T)$ then $T$ and $S\cup T$ form a close pair.

\begin{lem} \label{l:SoTmsr0} Let $S,T$ be subsets of $X$ and assume that $\mu_T$ exists. Then $S=o(T)$ implies $\mu_T(\overline S)=0$. The converse holds if $S\subseteq T$. \end{lem}

\begin{prf} The first statement is an easy consequences of \eqref{e:musc} and \eqref{e:dfnSoT}. As for the second claim, if $\mu_T(\overline S)=0$ and $S\subseteq T$, then \eqref{e:musc} yields that for every $\varepsilon>0$ exists $\gamma\in\calj$ such that
$$\varepsilon>\lim_{\beta\ge\gamma}\frac{|\pi_\beta(S)|}{|\pi_\beta(T)|}\;.$$
In turn, this implies
$$\varepsilon\ge\limsup_{\alpha\in\calj}\frac{|\pi_\alpha(S)|}{|\pi_\alpha(T)|}$$
and we conclude by the arbitrariness of $\varepsilon$. \end{prf}

\begin{eg} \label{eg:opccgrp} Assume $H\subseteq G$ are closed subsets of $X$, such that the structure maps $\pi_\alpha$ make $G$ a profinite group and $H$ a subgroup. (We will meet such a situation in Section \ref{s:da}, with $X=\da$, $G=\da^*$ and $H=\wds$.) Then one has 
$$H=o(G)\;\;\Longleftrightarrow\;\;[G:H]=\infty.$$
This follows from Lemma \ref{l:SoTmsr0}, since $\mu_G(H)=[G:H]^{-1}$. \end{eg}

We say that two sets $S$ and $T$ are {\em pro-commensurable} if they are both contained in a closed $Y\subseteq X$ such that the pro-counting measure $\mu_Y$ exists and  $\mu_Y\big(\overline S\big)\cdot\mu_Y\big(\overline T\big)>0$. On the other hand, under the same hypotheses, we have  $S=o(T)$ if $\mu_Y\big(\overline T\big)>0$ and $\mu_Y\big(\overline S\big)=0$.

\subsubsection{Fourier coefficients} \label{sss:fourier} Assume $X=G=\varprojlim G_\alpha$ is a profinite abelian group and write
\[ G^\vee:=\Homc(G,\C^*). \]
Every $\chi\in G^\vee$ factors through a finite quotient, hence through $G_\alpha$ for some $\alpha$. It follows that, with a little abuse of notation, we can write $G^\vee=\bigcup_\alpha G_\alpha^\vee$, with $G_\alpha^\vee\subseteq G_\beta^\vee$ if $\beta\ge\alpha$.

Let $\mu_G$ be the Haar measure of $G$. For $\chi\in G^\vee$, we set
\[ e_\chi:=\chi^{-1}\cdot\mu_G\in\cald(G,\C).\]
Similarly, for any finite abelian group $H$ and $\xi\in H^\vee$, put
$$e_{\xi,H}:=\xi^{-1}\cdot\mu_H=\frac{1}{|H|}\sum_{h\in H}\xi(h)^{-1}\delta_h\;\;\in\cald(H,\C).$$
It is well known that the elements $e_{\xi,H}$ form a basis of orthogonal idempotents of $\cald(H,\C)$ (with respect to the ring structure induced by the convolution product, as mentioned in Remark \ref{r:cnvlt}). In particular, the ring unity $\delta_e$ (with $e$ the identity element of $H$) decomposes as $\delta_e=\sum_{\xi\in H^\vee}e_{\xi,H}$.

\begin{lem}\label{l:fntfrr} Let $H$ be a finite abelian group. Then, for $\nu\in\cald(H,\C)$, one has the Fourier decomposition
\[ \nu=\sum_{\xi\in H^\vee}\nu(\xi)e_{\xi,H}. \]
\end{lem}

\begin{prf} For $x\in H$ and $\xi\in H^\vee$, a straightforward computation yields $\delta_xe_{\xi,H}=\xi(x)e_{\xi,H}$. Hence for $\nu=\sum_{x\in H}c_x\delta_x$ one gets
$$\nu=\nu\,\delta_e=\sum_{x\in H}c_x\delta_x\,\sum_{\xi\in H^\vee}e_{\xi,H}=\sum_{x,\xi}c_x\delta_xe_{\xi,H}=\sum_{\xi\in H^\vee}\left(\sum_{x\in H}c_x\xi(x)\right)e_{\xi,H}=\sum_{\xi\in H^\vee}\nu(\xi)e_{\xi,H}.$$
\end{prf}

In the case $H=G_\alpha$, we shall write $e_{\xi,\alpha}$ for $e_{\xi,G_\alpha}$.

\begin{lem} \label{l:frrqtnt} For $\chi\in G_\alpha^\vee$ one has $(\pi_\alpha)_*(e_\chi)=e_{\chi,\alpha}$. Moreover, for $\beta\ge\alpha$ and $\xi\in G_\beta^\vee$, one has 
$$(\pi_\alpha^\beta)_*(e_{\xi,\beta})=\begin{cases} e_{\xi,\alpha} & \text{\em if }\xi\in G_\alpha^\vee, \\ 0 & \text{\em otherwise}. \end{cases}$$
\end{lem}

\begin{prf} Straightforward from the definitions. \end{prf}

By Lemma \ref{l:frrqtnt} it follows $(\pi_\alpha)_*(e_\chi)=(\pi_\alpha^\beta)_*(e_{\xi,\beta})=0$ for $\chi\in G_\beta^\vee-G_\alpha^\vee$. This implies that the measures $e_\chi$ converge to $0$ with respect to the inverse limit topology of $\cald(G,\C)$ (in the sense that every neighbourhood of $0$ contains all the $e_\chi$'s but finitely many). One can also show that they form a topological basis of idempotents and deduce from this the Fourier decomposition in $\cald(G,\C)$. We take a slightly different approach.

\begin{prop} \label{p:frrprfnt} For every $\mu\in\cald(G,\C)$, one has
\begin{equation} \label{e:frrdcmpstm} \mu=\sum_{\chi\in G^\vee}\mu(\chi)e_\chi\,. \end{equation}
\end{prop}

\begin{prf} It suffices to show that the two sides of \eqref{e:frrdcmpstm} project to the same element in each $\cald(G_\alpha,\C)$. Lemma~\ref{l:fntfrr}, applied to $H=G_\alpha$ and $\nu=(\pi_\alpha)_*(\mu)$, yields
\[ (\pi_\alpha)_*(\mu)=\sum_{\xi\in G_\alpha^\vee}(\pi_\alpha)_*(\mu)(\xi)e_{\xi,\alpha}=\sum_{\xi\in G_\alpha^\vee}\mu(\xi)e_{\xi,\alpha} \] 
(by the abuse of notation $\xi=\xi\circ\pi_\alpha$ for $\xi\in G_\alpha^\vee$). On the other hand, by Lemma \ref{l:frrqtnt} one has 
$$(\pi_\alpha)_*\left(\sum_{\chi\in G^\vee}\mu(\chi)e_\chi\right)=\sum_{\chi\in G^\vee}\mu(\chi)(\pi_\alpha)_*(e_\chi)=\sum_{\xi\in G_\alpha^\vee}\mu(\xi)e_{\xi,\alpha}\,.$$
\end{prf}

\begin{rmk} The convergence of the infinite sum on the right-hand side of \eqref{e:frrdcmpstm} is justified remembering $\lim_{\chi\in G^\vee}e_\chi=0$. Equivalently, but in a different language, one can refer to the weak-* topology: if $f\in\call_c(G,\C)$, then $f$ factors through some $G_\gamma$, hence for all $\alpha\ge\gamma$ one has $f=f_\alpha\circ\pi_\alpha$ and
$$\mu(f)=(\pi_\alpha)_*(\mu)(f_\alpha)=\sum_{\xi\in G_\alpha^\vee}\mu(\xi)e_\xi(f_\alpha)\,,$$
with the last equality coming from Lemma \ref{l:fntfrr}. Thus  weak-* convergence yields $\mu=\lim_\alpha \sum_{\xi\in G_\alpha^\vee}\mu(\xi)e_\xi$. \end{rmk}
 
In particular, Proposition~\ref{p:frrprfnt} applies to the procounting measures attached to subsets of $G$: if $\mu_S$ exists, then it satisfies $\mu_S=\sum_{\chi\in G^\vee}\mu_S(\chi)e_\chi$.

\begin{eg} \label{eg:coset} Let $H$ be a subgroup of $G$ and $S=aH$ one of its cosets. Then condition \eqref{e:condhaar} holds, so $\mu_S$ exists and we can compute its Fourier coefficients $\mu_S(\chi)$ for all $\chi\in G^\vee$. Without loss of generality, we can assume $H$ is closed (since $\overline S=a\overline H$), so to have $H=\varprojlim_\alpha H_\alpha$, with $H_\alpha=\pi_\alpha(H)$. From
$$\mu_{S,\alpha}=\frac{1}{|H_\alpha|}\sum_{h\in H_\alpha}\delta_{\pi_\alpha(a)h}=\frac{1}{|H_\alpha|}\sum_{h\in H_\alpha}\delta_{\pi_\alpha(a)}\delta_h=\delta_{\pi_\alpha(a)}\mu_{H,\alpha}$$
it follows $\mu_S=\delta_a\mu_H$ (with $\mu_H$ the Haar measure $\mu_H$) and, for $\chi\in G^\vee$,
$$\mu_S(\chi)=\chi(a)\mu_H(\chi)=\begin{cases} \chi(a) & \text{ if }H\subseteq\ker(\chi) \\ 0 & \text{ otherwise} \end{cases}$$
(because the average of a nontrivial character on a compact group is zero). Thus we obtain the Fourier decomposition
$$\mu_{aH}=\sum_{\substack{\chi\in G^\vee\\ \chi(H)=1}}\chi(a)e_\chi\,,$$
where the sum is finite if and only if $H$ is open. \end{eg}

\section{Procounting measures on $\da$} \label{s:da}

Let $F$ be a global field. We fix a nonempty finite set $\cals$ of places of $F$ (containing the archimedean ones, if there are any) and let $D$ be the ring of $\cals$-integers in $F$: that is,
$$D:=\big\{x\in F\mid \val_v(x)\ge0\text{ for all }v\notin\cals\big\}$$ 
(where $\val_v$ is the valuation attached to $v$). We define
\begin{equation}\label{e:dfnda} \da:=\varprojlim D/\gota \end{equation}
with $\gota$ varying among all non-zero ideals of $D$. Each reduction-modulo-$\gota$ map $\pi_\gota\colon D\twoheadrightarrow D/\gota$ extends by continuity to a ring homomorphism 
$$\hpi_\gota\colon\da\twoheadrightarrow D/\gota\,.$$
By abuse of notation, we shall use the symbols $\hpi_\bullet$ also when the domain is $\da^n$, with $n>1$.

By construction there is a canonical injection of $D$ into $\da$ and in the following we will always think of $D$ as a (dense) subring of $\da$. Since the set of ideals of $D$ is ordered by inclusion, with final object $0$, we shall write $\displaystyle\lim_{\gotn\rightarrow0}$ in reference to the inverse system \eqref{e:dfnda}.

For every $\gotp$ non-zero prime ideal of $D$, we also have the $\gotp$-adic completion
$$\dap:=\varprojlim D/\gotp^n.$$
These objects are related by a canonical isomorphism of topological rings
\begin{equation}\label{e:crt} \da\simeq\prod_\gotp\dap\,, \end{equation}
the product being over all  non-zero prime ideals of $D$ (a proof can be found in \cite[Theorem 2.1]{dl}), with canonical projections
$$\hpi_{\gotp^\infty}\colon\da\twoheadrightarrow\dap\,.$$
The decomposition \eqref{e:crt} is compatible with the conventions of Lemma \ref{l:prdst}: indeed, taking as $\calk$ the set of maximal ideals of $D$ and $\calj_\gotp:=\{\gotp^n\}_{n\in\N}$ for each $\gotp\in\calk$, there is an obvious correspondence between $\calj$ (as in Lemma \ref{l:prdst}) and the set $\cali(D)$ of non-zero ideals of $D$. We also note that $\calj=\cali(D)$ is countable and contains totally ordered cofinal subsets (e.g., index the non-zero prime ideals by $i\in\N$ and define $\calj_0=\{\gotm_n\}_{n\in\N}$, with $\gotm_n=\prod_{i\le n}\gotp_i^n$). 

For the sake of simplicity, in the following we shall usually think of \eqref{e:crt} as an equality. Each ring $\dap$ is endowed with a discrete valuation, which yields a valuation
$$v_\gotp\colon\da\rightarrow\N\cup\{\infty\}$$
extending the $\gotp$-adic valuation on $D$.

For $\gotp$ as above, we fix the notations $\F_\gotp:=D/\gotp$ and $q_\gotp:=|\F_\gotp|$.\\

\noindent{\bf Caveat}: from now on, we shall write $\xa$ (and not $\overline X$) to denote the closure of $X\subseteq\da^n$.\\

In the case of an ideal $\gota$ of $D$, it is easy to check that one has $\widehat\gota=\gota\da$; moreover the equality
$$D/\gota=\da/\widehat\gota$$
holds for every non-zero $\gota$.

\subsection{Eulerian sets} Following \cite[Definition 6.1]{dl}, we say that $X\subset\da^n$ is {\em Eulerian} if
$$\xa=\prod_\gotp X(\gotp)\,,$$ 
where $X(\gotp)$ is the closure of $\hpi_{\gotp^\infty}(X)$ in $\da_\gotp^n$\,. If moreover each $X(\gotp)$ is open, we say that $X$ is {\em openly Eulerian}.
For example, coprime pairs in $\Z^2$ are an openly Eulerian subset of $\za^2$ (this is a special case of \cite[Corollary 6.11]{dl}). Also, we note that the definition above is slightly different from its analogue in \cite{dl}, which  assumed $X\subseteq D^n$.

\begin{rmk} Since each $X(\gotp)$ is compact, so is their product. Because $\da$ is Hausdorff, it follows that the inclusion $\xa\subseteq\prod_\gotp X(\gotp)$ always holds. The hard part in showing that a set is Eulerian is to prove the opposite inclusion.\end{rmk}

\begin{eg} The most basic examples of openly Eulerian subsets are the cosets of non-zero ideals: if $A=a+\gota$ then $\widehat A=a+\widehat\gota=\prod_\gotp(a+\gota\dap)$. The existence of the procounting measure $\mu_A$ and its description follow from Example \ref{eg:coset}; in particular, we have the Fourier decomposition $\mu_A=\sum\chi(a)e_\chi\,,$ where the sum is taken over the (additive) characters of $D$ which factor through $D/\gota$ (that is, the characters of conductor $\gota$). \end{eg}

\begin{thm} \label{t:opnelr} If $X\subseteq\da^n$ is openly Eulerian, then its procounting measure exists, given by the formula
\begin{equation} \label{e:msrapelr} \mu_X=\bigotimes_\gotp\frac{1}{\mu_{\dap^n}\big(X(\gotp)\big)}\b1_{X(\gotp)}\mu_{\dap^n}\,. \end{equation}
\end{thm}

\noindent We shall call the numbers $1/\mu_{\dap^n}\big(X(\gotp)\big)$ appearing in \eqref{e:msrapelr} the {\em local coefficients} of $\mu_X$ and often use the shortening $c_{X,\gotp}$ to denote them.

\begin{prf} Each $X(\gotp)$ is compact open, so Lemma \ref{l:mucmpap} shows that $\mu_{X(\gotp)}\in\cald(\dap^n,\R)$ exists and has the form appearing as a local factor on the right-hand side of \eqref{e:msrapelr}. Now apply Proposition \ref{p:prddstr2} to obtain the procounting measure of $\xa$. By Lemma \ref{l:muS-chs}, this is the same as $\mu_X$\,. \end{prf}

\begin{rmk} \label{r:pLf} The results used in the proof of Theorem \ref{t:opnelr} are true for distributions with coefficients in any field of characteristic $0$: therefore, if $X$ is openly Eulerian, one has $\mu_X\in\cald(X,\Q_p)$ for every prime $p$. This gives hopes for a possible connection of our theory with $p$-adic zeta functions. Indeed, taking (say) $X=\za^*$, the procounting measure is just the Haar measure on $\za^*$, which is deeply connected with the Riemann zeta function; and the latter notoriously admits $p$-adic interpolation. \end{rmk} 

In the following, we shall say that $Y\subseteq\da^n$ is {\em almost openly Eulerian} if it forms a close pair with an openly Eulerian set $X$. If this happens, the closure $\xa$ is uniquely determined by $Y$ (because \eqref{e:msrapelr} implies that openly Eulerian sets with different closure originate different measures).

\begin{cor}\label{c:3.4} If $X\subseteq\da^n$ is almost openly Eulerian, then its procounting measure exists. \end{cor}

\begin{prf} Obvious from Proposition \ref{p:muST}. \end{prf}

\begin{eg} \label{eg:sqrfr} Let $\sqf\subset\Z$ denote the set of squarefree integers. Then $\sqf$ is openly Eulerian, with closure $\prod_p(\Z_p-p^2\Z_p)$ (see \cite[Corollary 6.15 and formula (75)]{dl}), so that, by Theorem \ref{t:opnelr}, its associated procounting measure is
\begin{equation} \label{e:msrsqf} \mu_{\sqf}=\bigotimes_p\mu_{\sqf(p)}=\bigotimes_p\left(1-\dfrac{1}{p^2}\right)^{-1}\b1_{\Z_p-p^2\Z_p}\cdot\mu_{\Z_p}\,. \end{equation}
As well known, the product of the local coefficients of $\mu_{\sqf}$ converges to $\zeta(2)$. Hence one can rewrite \eqref{e:msrsqf} as $\mu_{\sqf}=\zeta(2)\b1_{\sqf}\mu_{\za}$\,. Actually, for any compact open $U\subseteq\za$, formulae \eqref{e:msrsqf} and \eqref{e:tnsr2} yield
\begin{equation} \label{e:msrsqf1} \mu_{\sqf}(\b1_U)=\prod_p\left(1-\dfrac{1}{p^2}\right)^{-1}\cdot\prod_p\mu_{\Z_p}\big(U\cap(\Z_p-p^2\Z_p)\big)=\zeta(2)\cdot\mu_{\za}\big(U\cap\widehat\sqf\big) \end{equation}
where the first passage is justified by absolute convergence of both products (with almost all factors being $1-p^{-2}$ in the second one). For more insight, we also recall the equality $\zeta(2)=1/\mu_{\za}\big(\widehat\sqf\big)$, which reduces \eqref{e:msrsqf1} to a special case of \eqref{e:msrs>0}.

We also compute the Fourier decomposition of $\mu_{\sqf}$. By Proposition \ref{p:frrprfnt}, the component at $\chi\in\za^\vee$ has coefficient $\mu_{\sqf}(\chi)$. Under the decomposition $\za=\prod_p\Z_p$, write $\chi=\prod_p\chi_p$, with almost all $\chi_p$ trivial. Let $m$ be the conductor of $\chi$ (i.e., the least positive integer $m$ such that $\chi$ factors through $\Z/m\Z$). From \eqref{e:msrsqf} and \eqref{e:tnsr2}, we get
\begin{equation} \label{e:frrcffsqfr} \mu_{\sqf}(\chi)=\prod_p\left(1-\frac{1}{p^2}\right)^{-1}\int_{\Z_p-p^2\Z_p}\chi_p\,d\mu_{\Z_p}. \end{equation}
Elementary computations show that the factor at $p$ of \eqref{e:frrcffsqfr} is
\[\begin{cases} 1 & \text{ if }v_p(m)=0,\\[4pt] \dfrac{1}{1-p^2} & \text{ if }v_p(m)\in\{1,2\},\\[8pt] 0 & \text{ if }v_p(m)\ge 3. \end{cases} \] 
Consequently, one has
\[ \mu_{\sqf}(\chi)=\begin{cases} \displaystyle\prod_{p\mid m}\frac{1}{1-p^2} & \text{ if }v_p(m)\le 2\text{ for every }p,\\[12pt] 0 & \text{ if }v_p(m)\ge 3\text{ for some }p,\end{cases} \]
with the convention that the empty product is $1$. By Proposition~\ref{p:frrprfnt}, applied with $G=\za$ and $\mu=\mu_{\sqf}$, we therefore obtain
\[ \mu_{\sqf}=\sum_{\chi\in\za^\vee}\mu_{\sqf}(\chi)e_\chi=\sum_{\substack{\chi\in\za^\vee\\ v_p(\cond(\chi))\le 2\ \forall p}}\left(\prod_{p\mid \cond(\chi)}
\frac{1}{1-p^2}\right)e_\chi\ .\]
\end{eg}

\begin{rmk} We were able to simplify \eqref{e:msrsqf} by moving a factor $\zeta(2)$ outside of the tensor product. It is important to observe that this is an exception: in general, in the setting of \eqref{e:msrapelr}, the product of the local coefficients $c_{X,\gotp}$ has no limit in $\R_{>0}$ and therefore $\mu_X$ cannot be written as the restriction to $\xa$ of a scalar multiple of $\mu_{\da^n}$. Note also that the product diverges if and only $\mu_{\da^n}(\xa)=0$, because $1\le c_{X,\gotp}\le\infty$ holds by construction. For instance, this happens with $X=\da^*=\prod_\gotp\dap^*$, which yields
$$\mu_{\da^*}=\bigotimes_\gotp\left(1-\frac{1}{q_\gotp}\right)^{-1}\b1_{\dap^*}\mu_{\dap}.$$
\end{rmk}

\subsubsection{Prime elements} \label{sss:prmelm} Recall that the prime elements of $D$ are those irreducible $x\in D$ such that the ideal $xD$ is prime. We shall denote the set of such elements by $\calp(D)$. For notational convenience, we fix a generator $t_\gotp\in D$ for each principal prime ideal $\gotp$  and let $T:=\{t_\gotp:\gotp\text{ principal}\}$.

The closure of $\calp(D)$  in $\da$ was computed in \cite[Theorem 3.2]{dl}, which  states the equality \footnote{Caveat: the notation in \cite{dl} is not always coherent with the one we use here. In particular, the meaning of $\calp(D)$ is different.}
\begin{equation} \label{e:clirr} \widehat{\calp(D)}=\da^*\sqcup\wds\calp(D)=\da^*\sqcup\bigsqcup_{t_\gotp\in T}t_\gotp\wds\,. \end{equation}
Here $\wds$ is the closure of $D^*$ in $\da$. Note that $\wds$ is a subgroup of $\da^*$,  of infinite index (as proved in \cite[Proposition 3.7]{dl}): actually, comparison between the sizes of these two multiplicative groups plays a crucial role in the following result.

\begin{thm} \label{t:clprim} The sets $\da^*$ and $\calp(D)$ form a close pair. \end{thm}

\begin{proof} Let $\gotn\subset D$ be a non-zero ideal and set
$T_\gotn:=\{t_\gotp:\gotp\mid\gotn\}$. Since $D$ is a Dedekind
domain, one has
$\hpi_\gotn(\da^*)=(D/\gotn)^*$.
On the other hand, \eqref{e:clirr} yields
\[
\hpi_\gotn(\calp(D))=(D/\gotn)^*\sqcup P_\gotn,
\]
where
\[
P_\gotn:=
\bigsqcup_{t_\gotp\in T_\gotn}\hpi_\gotn(t_\gotp\wds)
=
\bigsqcup_{t_\gotp\in T_\gotn}\hpi_\gotn(t_\gotp D^*);
\]
the second equality follows from Lemma~\ref{l:clsr}. Thus, by
Definition~\ref{d:clpr}, it is enough to prove
\begin{equation}\label{e:untprim}
\lim_{\gotn\to0}\frac{|P_\gotn|}{|(D/\gotn)^*|}=0.
\end{equation}

Suppose first that either $F$ is a number field or its constant field
contains $q>2$ elements. By the inequalities
\[
|P_\gotn|\le|\hpi_\gotn(D^*)|\,|T_\gotn|
\qquad
\text{and}
\qquad
|T_\gotn|\le\omega(\gotn)
\]
it is enough to show
\begin{equation}\label{e:untomg}
\lim_{\gotn\to0}
\frac{|\hpi_\gotn(D^*)|\,\omega(\gotn)}
{|(D/\gotn)^*|}=0.
\end{equation}
There exists a prime number $p_F$ dividing $q_\gotp-1$ for all but
finitely many prime ideals $\gotp\subset D$. Indeed, in the number
field case one may take $p_F=2$, while in the function field case any
prime divisor of $q-1$ will do. By the Chinese remainder theorem there
is a surjective homomorphism
\[
\psi_\gotn\colon
(D/\gotn)^*
\twoheadrightarrow
\prod_{\gotp\mid\gotn}
\F_\gotp^*/(\F_\gotp^*)^{p_F},
\]
whose image has cardinality at least
\[
\big|\psi_\gotn((D/\gotn)^*)\big|
\ge p_F^{\omega(\gotn)-c_D},
\]
where $c_D$ is the number of prime ideals $\gotp$ such that
$p_F\nmid q_\gotp-1$. On the other hand,
\[
\big|\psi_\gotn(\hpi_\gotn(D^*))\big|
\le |D^*/(D^*)^{p_F}|
\le p_F^{|\cals|},
\]
because $D^*$ is finitely generated of rank $|\cals|-1$ and has cyclic
torsion subgroup. Since for every subgroup $H$ of the source of
$\psi_\gotn$ one has
\[
\frac{|H|}{|(D/\gotn)^*|}
\le
\frac{|\psi_\gotn(H)|}
{|\psi_\gotn((D/\gotn)^*)|},
\]
we obtain
\begin{equation}\label{e:untprm2}
\frac{|\hpi_\gotn(D^*)|\,\omega(\gotn)}
{|(D/\gotn)^*|}
\le
\frac{p_F^{|\cals|}\,\omega(\gotn)}
{p_F^{\omega(\gotn)-c_D}}.
\end{equation}
As $\omega(\gotn)\to+\infty$ when $\gotn\to0$, this proves
\eqref{e:untprim} in this case.

It remains to consider the case in which $F$ has exact constant field
$\F_2$. Let us set
\[
r:=\rank D^*=|\cals|-1,\qquad
s_\gota:=|\hpi_\gota(D^*)|,
\qquad
\Phi(\gota):=|(D/\gota)^*|.
\]
Since $\F_2^*=\{1\}$, one has $D^*\simeq\Z^r$.
For every principal prime ideal $\gotp\mid\gotn$ and every
$u,v\in D^*$,
\[
t_\gotp u\equiv t_\gotp v\bmod{\gotn}
\qquad\text{if and only if}\qquad
u\equiv v\bmod{\gotn\gotp^{-1}}.
\]
Moreover, the subsets
$\hpi_\gotn(t_\gotp D^*)$ corresponding to distinct prime ideals are
disjoint. Hence
\begin{equation}\label{e:unt-F2-exact}
|P_\gotn|
=
\sum_{\substack{\gotp\mid\gotn\\ \gotp\ {\rm principal}}}
s_{\gotn\gotp^{-1}}.
\end{equation}

Let $\gotm$
be an ideal dividing $\gotn$. For every $\gotp\mid\gotn$, reduction
induces a surjection
\[
(D/(\gotn\gotp^{-1}))^*
\twoheadrightarrow
(D/(\gotm\gotp^{-1}))^*.
\]
Since $|(D/\gotq^a)^*|=(q_\gotq-1)q_\gotq^{a-1}$, it is easy to see that its kernel has cardinality at most
$\Phi(\gotn)/\Phi(\gotm)$. Therefore
\[
s_{\gotn\gotp^{-1}}
\le
\frac{\Phi(\gotn)}{\Phi(\gotm)}
s_{\gotm\gotp^{-1}},
\]
and \eqref{e:unt-F2-exact} gives
\begin{equation}\label{e:unt-F2-rad}
\frac{|P_\gotn|}{\Phi(\gotn)}
\le
\frac{|P_\gotm|}{\Phi(\gotm)}.
\end{equation}
It is thus enough to consider squarefree ideals.
Let $\gotm$ be squarefree and, for $d\ge1$, let us set
\[
c_d:=|\{\gotp\mid\gotm:\deg(\gotp)=d\}|
\]
The Chinese remainder theorem gives
$\Phi(\gotm)=\prod_{d\ge1}(2^d-1)^{c_d}$.
Let $\gotp\mid\gotm$ have degree $e$. For every $d$, the projection of
$\hpi_{\gotm\gotp^{-1}}(D^*)$ to the product of the degree $d$
factors is generated by at most $r$ elements. Since this product
consists of $c_d-\delta_{d,e}$ cyclic groups of order $2^d-1$, its image
has cardinality at most
$(2^d-1)^{\min\{r,c_d-\delta_{d,e}\}}$,
where $\delta_{d,e}$ is the Kronecker delta. Thus
\[
s_{\gotm\gotp^{-1}}
\le
\prod_{d\ge1}
(2^d-1)^{\min\{r,c_d-\delta_{d,e}\}}.
\]
Let us set
$A_\gotm:=
\prod_{c_d>r}(2^d-1)^{-(c_d-r)}$.
After dividing by $\Phi(\gotm)$, summing
\eqref{e:unt-F2-exact}, and enlarging the sum from principal prime
divisors to all prime divisors, we obtain
\begin{equation}\label{e:unt-F2-degree}
\frac{|P_\gotm|}{\Phi(\gotm)}
\le
A_\gotm
\left(
\sum_{c_d\le r}\frac{c_d}{2^d-1}
+\sum_{c_d>r}c_d
\right).
\end{equation}

Let $N_1$ be the number\footnote{The fact that there are at most finitely many ideals of a given norm is a well known fact. A possible reference is \cite{Samuel}.} of degree one prime ideals of $D$, and consider $C_0:=\sum_{d\ge2}\frac{1}{2^d-1}<+\infty$.
By the prime number theorem for global function fields (see, for instance, \cite[Chapter~5]{rosen}), the number of
degree $d$ prime ideals of $D$ is asymptotic to $2^d/d$. Thus, for
every sufficiently large $d_0$, there exist $r+1$ distinct prime
ideals $\gotq_1,\ldots,\gotq_{r+1}$
of degree $d_0$. Consider
$\gotb:=\gotq_1\cdots\gotq_{r+1}$.
If $\gotb\mid\gotm$, then $c_{d_0}\ge r+1$, and hence
$A_\gotm\le (2^{d_0}-1)^{-1}$. Moreover, using
\[
\frac{c}{M^{c-r}}\le\frac{r+1}{M}
\qquad\text{for}\quad c\ge r+1\;\; \text{and}\;\; M\ge2,
\]
which holds since the function $c\mapsto c/M^{c-r}$ is decreasing for $c\ge r+1$, we obtain
\[
A_\gotm\sum_{c_d\le r}\frac{c_d}{2^d-1}
\le
\frac{N_1+rC_0}{2^{d_0}-1}
\]
and
\[
A_\gotm\sum_{c_d>r}c_d
\le
\frac{N_1+(r+1)(1+C_0)}{2^{d_0}-1}.
\]
It follows from \eqref{e:unt-F2-degree} that
\[
\frac{|P_\gotm|}{\Phi(\gotm)}
\le
\frac{C_D}{2^{d_0}-1},
\qquad\text{where}
\quad
C_D:=2N_1+(r+1)+(2r+1)C_0
\]
uniformly for all squarefree $\gotm$ divisible by $\gotb$.

Given $\varepsilon>0$, choose $d_0$ sufficiently large that
$C_D/(2^{d_0}-1)<\varepsilon$, and choose $\gotb$ as above. For every
$\gotn$ divisible by $\gotb$, its radical $\text{rad}(\gotn)$ is also divisible by
$\gotb$, and \eqref{e:unt-F2-rad} gives
\[
\frac{|P_\gotn|}{|(D/\gotn)^*|}
\le
\frac{|P_\gotm|}{|(D/\gotm)^*|}
\le
\frac{C_D}{2^{d_0}-1}
<\varepsilon.
\]
This proves \eqref{e:untprim} also when the exact constant field of
$F$ is $\F_2$.

\end{proof}

\begin{rmks} \label{r:indcunt} \begin{itemize} \item[]  \end{itemize}
\noindent{\bf \ref{r:indcunt}.1.} The proof of Theorem~\ref{t:clprim} is ``horizontal'': the key point is the increasing number of distinct prime ideals dividing $\gotn$ as $\gotn$ tends to zero. One can be tempted to use instead a ``vertical'' argument, considering first the powers $\gotp^k$ of a fixed prime ideal, $k$ going to infinity: more precisely, assuming $D^*=o(\dap^*)$ for all $\gotp$ (which is always true in positive characteristic), the ensuing equalities $\mu_{\hpi_{\gotp^\infty}(\calp(D))}=\mu_{\dap^*}$ yield $\mu_{\da^*}=\mu_C$, with $C=\prod_\gotp\hpi_{\gotp^\infty}(\calp(D))$. The inclusions
$$\da^*\subset\widehat{\calp(D)}\subset C$$
then would complete the proof, if one could apply Lemma \ref{l:inclprssm}. Unfortunately, $\da^*$ and $C$ do not form a close pair. Moreover, the conditions $A\subset B\subset C$ and $\mu_A=\mu_C$ do not imply $\mu_A=\mu_B$ (one can easily find a counterexample by a reasoning similar to the one in Example \ref{eg:ptaccm}.2), so this argument is not sufficient to prove $\mu_{\da^*}=\mu_{\calp(D)}$.\\
\noindent{\bf \ref{r:indcunt}.2.} The proof of Theorem~\ref{t:clprim} becomes much simpler when $D^*$ is finite. We shall illustrate this in the case of $\Z$ (and $\calp$) in Example \ref{eg:PZ}.
\end{rmks} 

\begin{cor} \label{c:msrprmd} The set $\calp(D)$ is almost openly Eulerian. The procounting measure $\mu_{\calp(D)}$ is the Haar measure of the compact group $\da^*$. \end{cor}

\begin{prf} It suffices to observe that $\da^*$ is openly Eulerian, since \eqref{e:crt} implies $\da^*=\prod_\gotp\dap^*$ and the group of units is compact open in each ring $\dap$. \end{prf}

\begin{eg}[{\bf primes in $\Z$}] \label{eg:PZ} In the case $D=\Z$, one has $\calp(\Z)=\calp\sqcup\{-p\mid p\in\calp\}$ and also $\cpr=\calp\sqcup\za^*$, as we already observed in \eqref{e:prmZ}. Thus Theorem \ref{t:clprim} implies that $\calp$ and $\za^*$ form a close pair: in particular, $\mu_\calp$ exists and is equal to $\mu_{\za^*}$. For a direct proof of this last statement, note that two Radon measures on a compact space are equal if they coincide on closed sets: hence $\mu_{\za^*}=\mu_\calp$ amounts to the claim that for any $T\subseteq\calp$ one has
\begin{equation} \label{e:msrprmZ} \mu_{\za^*}\big(\widehat T\cap \za^*\big)=\mu_\calp\big(\widehat T\big)\stackrel{\text{by }\eqref{e:musc}}{=}\lim_{n\to 0}\frac{|\pi_n(T)|}{|\pi_n(\calp)|}\;. \end{equation} 
For the sake of clarity, we give a direct proof of \eqref{e:msrprmZ}. Writing $\pi_n(\calp)=(\Z/n\Z)^*\sqcup\supp(n)$ (where $\supp(n)$ denotes the set of primes dividing $n$), we obtain 
$$|\pi_n(\calp)|=\varphi(n)+\omega(n)\,,$$
with $\varphi$ the Euler totient function and $\omega$ counting the distinct prime divisors of $n$. Also,
$$\pi_n(T)=\pi_n(\widehat T\cap \za^*)\sqcup X_n$$
for some $X_n\subseteq\supp(n)$. Elementary considerations show
$$\lim_{n\rightarrow0}\frac{\omega(n)}{\varphi(n)}=0$$
and hence
$$\lim_{n\rightarrow0}\frac{|\pi_n(T)|}{|\pi_n(\calp)|}=\lim_{n\rightarrow0}\frac{|\pi_n(\widehat T\cap\za^*)|}{\varphi(n)}=\lim_{n\rightarrow0}\frac{|\pi_n(\widehat T\cap\za^*)|}{|\pi_n(\za^*)|}=\mu_{\za^*}\big(\widehat T\cap \za^*\big).$$ 
The equality $\mu_\calp=\mu_{\za^*}$ is in accordance with Lemma \ref{l:isolated}: the measure $\mu_\calp$ cannot detect the isolated points of $\cpr$, which are exactly the prime numbers.

We also remark that, when $T$ is an arithmetic progression, Dirichlet's theorem on primes implies that \eqref{e:msrprmZ} yields exactly the $\pi$-measure defined in \cite[Section III]{gol}. \end{eg}

\begin{eg}[Fourier coefficients of $\mu_\calp$] In the case $D=\Z$, by Corollary~\ref{c:msrprmd} we have $\mu_{\calp}=\mu_{\za^*}$. Hence, by Proposition~\ref{p:frrprfnt} applied to $G=\za$,
\[ \mu_{\calp}=\sum_{\chi\in\za^\vee}\mu_{\calp}(\chi)e_\chi=\sum_{\chi\in\za^\vee}\mu_{\za^*}(\chi)e_\chi .\]
Let $n>0$, and let $\chi\in\Hom(\Z/n\Z,\C^*)$, viewed also as a character of $\za$. Since $(\hpi_n)_*(\mu_{\za^*})$ is the uniform probability measure on $(\Z/n\Z)^*$, we get
\[ \mu_{\calp}(\chi)=\mu_{\za^*}(\chi)=\frac{1}{\varphi(n)}\sum_{u\in(\Z/n\Z)^*}\chi(u). \]
In terms of Ramanujan's sums, if $\chi_a(x):=\exp(2\pi i ax/n)$ and $a\in\Z/n\Z$, then
\[ \mu_{\calp}(\chi_a)=\frac{1}{\varphi(n)}\sum_{\substack{u\bmod n\\ (u,n)=1}}\exp(2\pi i au/n)\ .\]
In particular, if $m=\cond(\chi)$, then, by  \cite[Theorem~8.6]{apostol}, $\mu_{\calp}(\chi)=
\dfrac{\mu_{\mathrm{Mob}}(m)}{\varphi(m)}$, where $\mu_{\mathrm{Mob}}$ denotes the M\"obius function. Thus
\[ \mu_{\calp}=\sum_{\chi\in\za^\vee}\frac{\mu_{\mathrm{Mob}}(\cond(\chi))}{\varphi(\cond(\chi))}\,e_\chi. \]
\end{eg}

\subsubsection{Polynomial images and preimages} Given $f\in D[x_1,\dots,x_n]$, it induces maps (which, by abuse of notation, we denote by the same letter) $f\colon\da^n\rightarrow\da$  and $f\colon\dap^n\rightarrow\dap$. It is straightforward to check that such maps are well-behaved with respect to the decomposition \eqref{e:crt}: given any sets $A_\gotp\subseteq\dap^n$, $B_\gotp\subseteq\dap$, one has
\begin{equation} \label{e:plynmprd} f\left(\prod_\gotp A_\gotp\right)=\prod_\gotp f(A_\gotp)\;\;\text{ and }\;\;f^{-1}\left(\prod_\gotp B_\gotp\right)=\prod_\gotp f^{-1}(B_\gotp)\,. \end{equation}
For more detail on \eqref{e:plynmprd} and how it relates with Eulerian sets, we refer to \cite[Sections 6.1.1 and 6.3]{dl}. Here we just want to use it to produce procounting measures attached to images and preimages. We start with a couple of easy examples, to get a feeling about what to expect for $\mu_{f(D)}$.

\begin{eg+} \label{eg:imgx2} \begin{itemize}  \item[] \end{itemize}
\noindent {\bf \ref{eg:imgx2}.1.} The image of a degree $1$ polynomial $ax+b$ is an ideal coset $aD+b$ and the closure is a product of terms $a\dap+b$ (which become $\dap$ if $a\notin\gotp$).\\
\noindent {\bf \ref{eg:imgx2}.2.} The next simplest example is $f(x)=x^2$, in which case the local factors have the form
$$f(\dap)=\{0\}\sqcup\bigsqcup_{n\in\N}\gotp^{2n}f(\dap^*)$$
and $f(\dap^*)$ is easy to compute explicitly: e.g., for $D=\Z$ one has
$$f(\Z_p^*)=\begin{cases}1+8\Z_2 & \text{ if }p=2 \\ \mmu_{(p-1)/2}\times(1+p\Z_p) & \text{ if }p\neq2 \end{cases}$$
(where $\mmu_n$ denote the group of roots of unity of order $n$ in $\Z_p$), as discussed in \cite[II, \S3.3]{ser}. \end{eg+}

\begin{rmk} \label{r:dg>1msr0} From the description in Example \ref{eg:imgx2}.2, for $f(x)=x^2$ one computes $\mu_{\Z_p}\big(f(\Z_p)\big)<1/2$ for every $p$, which implies $\mu_{\za}\big(f(\za)\big)=0$. The vanishing of $\mu_{\da}\big(f\da)\big)$ for non-linear $f\in D[x]$ is a general phenomenon, already observed in \cite[Theorem 9]{fptw} and, for $D=\Z$, in \cite[Theorem 3.9]{lt2}. \end{rmk}

We propose another example, which reconnects to Section \ref{sss:hier}.

\begin{eg}[\bf Landau's primes] \label{eg:LandauLittleO} Denote by $\call:=\{p\in\calp:\ p=m^2+1\text{ for some }m\in\N\}$ the set of Landau's primes. We shall show that
\[ \call=o(\calp). \]
Let $S:=\{m^2+1:\ m\in\N\}$. We first show that $\mu_{\za^*}\big(\widehat S\cap\za^*\big)=0$. Let $f(x)=x^2+1$. Since $\N$ is dense in $\za$ and $f$ is continuous, one has $\widehat S=f(\za)=\prod_p f(\Z_p)$. Therefore $\widehat S\cap\za^*=\prod_p\big(f(\Z_p)\cap\Z_p^*\big)$.

For an odd prime $p$, the reduction modulo $p$ of $f(\Z_p)\cap\Z_p^*$ is contained in $V_p:=\{a^2+1\in\F_p^*:\ a\in\F_p\}$. Since the set of squares in $\F_p$ has cardinality $(p+1)/2$, we have $|V_p|\le \frac{p+1}{2}$. For every finite set $\Sigma$ of primes, we have $\widehat S\cap\za^*\subset \prod_{p\in\Sigma}\big(f(\Z_p)\cap\Z_p^*\big)\times\prod_{p\notin\Sigma}\Z_p^*$.
As $\mu_{\za^*}=\bigotimes_p\mu_{\Z_p^*}$, it follows that
\[ \mu_{\za^*}\big(\widehat S\cap\za^*\big)\le\prod_{p\in\Sigma}\mu_{\Z_p^*}\big(f(\Z_p)\cap\Z_p^*\big). \]
For odd $p$, the reduction of $f(\Z_p)\cap\Z_p^*$ modulo $p$ is contained in $V_p$, and each residue class in $\F_p^*$ has measure $1/(p-1)$. Hence
\[ \mu_{\Z_p^*}\big(f(\Z_p)\cap\Z_p^*\big)\le\frac{|V_p|}{p-1}. \]
Hence, for every finite set $\Sigma$ of primes $p\ge5$, 
\[ \mu_{\za^*}\big(\widehat S\cap\za^*\big)\le\prod_{p\in \Sigma}\frac{|V_p|}{p-1}\le\prod_{p\in \Sigma}\frac{p+1}{2(p-1)}\le\left(\frac34\right)^{|\Sigma|}. \]
Letting $|\Sigma|\to\infty$, we obtain $\mu_{\za^*}\big(\widehat S\cap\za^*\big)=0$.
As $\call\subseteq S$, it follows $\widehat{\call}\cap\za^*\subseteq\widehat S\cap\za^*$. Therefore $\mu_{\za^*}\big(\widehat{\call}\cap\za^*\big)=0$. Using \eqref{e:msrprmZ} with $T=\call$, we get
\[ \lim_{n\to0}\frac{|\pi_n(\call)|}{|\pi_n(\calp)|}=\mu_{\za^*}\big(\widehat{\call}\cap\za^*\big)=0 \]
which is our definition of $\call=o(\calp)$. \end{eg} 

The example \ref{eg:imgx2}.2 of $x^2$ suggests that, if a polynomial image fails to be openly Eulerian  (as in this case), nonetheless the failure must be ``small'',  in the sense that if $X$ is openly Eulerian then $f(X)$ can be approximated by openly Eulerian sets $f(X_i)$ so that $\mu_{f(X_i)}$ converge to $\mu_{f(X)}$ (in the sense of Proposition \ref{p:brd0}). 

\begin{lem}\label{l:qe} Let $\gotp\subset D$ be a non-zero prime, and consider the fraction field $K_{\gotp}:=\Frac(\da_{\gotp})$.
Let $f\in D[x_1,\dots,x_n]$ be non-constant. Set 
\[ Z_{\gotp}:=\Big\{\alpha\in \da_{\gotp}^n:\ \frac{\partial f}{\partial x_1}(\alpha)=\cdots=\frac{\partial f}{\partial x_n}(\alpha)=0\Big\}. \]
Define $A_{\gotp}\subset f(\da_{\gotp}^n)$ to be the set of values $a$ for which there exists $\alpha\in \da_{\gotp}^n- Z_{\gotp}$ with $f(\alpha)=a$, and set\footnote{Note that in general $B_{\gotp}\subseteq f(Z_{\gotp})$, and the inclusion may be strict.}
\[ B_{\gotp}:=f(\da_{\gotp}^n)- A_{\gotp}. \]
Then $f(\da_{\gotp}^n)=A_{\gotp}\sqcup B_{\gotp}$ such that: \begin{enumerate}
\item $A_{\gotp}$ is an open subset of $\da_{\gotp}$.
\item Assume that the generic fiber of $f\colon\A^n_{K_{\gotp}}\to \A^1_{K_{\gotp}}$ is smooth. Then $A_{\gotp}$ is nonempty and $B_{\gotp}$ is finite.
\item If $\partial f/\partial x_i\equiv 0$ for all $i$, then $A_{\gotp}=\emptyset$ and $B_{\gotp}=f(\da_{\gotp}^n)$. \end{enumerate} \end{lem}

\begin{proof} By definition one always has the disjoint union $f(\da_{\gotp}^n)=A_{\gotp}\sqcup B_{\gotp}$.

\smallskip\noindent
(1) \emph{Openness of $A_{\gotp}$.} Fix $a_0\in A_{\gotp}$ and choose $\alpha=(\alpha_1,\dots,\alpha_n)\in\da_{\gotp}^n- Z_{\gotp}$ with $f(\alpha)=a_0$.
Then $\partial f/\partial x_i(\alpha)\neq 0$ for some $i$. Freezing all variables except $x_i$, consider
\[ g_a(T):=f(\alpha_1,\dots,\alpha_{i-1},T,\alpha_{i+1},\dots,\alpha_n)-a\in \da_{\gotp}[t], \]
where $a\in \da_{\gotp}$ is a parameter. We have $g_{a_0}(\alpha_i)=0$ and $g'_{a_0}(\alpha_i)=\partial f/\partial x_i(\alpha)\neq 0$. By Hensel's lemma (as stated in \cite[II, Section~2, Theorem~1]{ser}), for all $a$ sufficiently $\gotp$-adically close to $a_0$ there exists $\beta_i\in \da_{\gotp}$ close to $\alpha_i$ such that $g_a(\beta_i)=0$.
Replacing only the $i$-th coordinate of $\alpha$ by $\beta_i$ yields $\beta\in \da_{\gotp}^n$ with $f(\beta)=a$, hence $a\in A_{\gotp}$. Thus a neighbourhood of $a_0$ is contained in $A_{\gotp}$, so $A_{\gotp}$ is open.

\smallskip\noindent
(2) \emph{Finiteness of $B_{\gotp}$.} We first show that $\da_{\gotp}^n\subset \A^n(K_{\gotp})$ is Zariski dense. Indeed, let $F\in K_{\gotp}[x_1,\dots,x_n]$ vanish on $\da_{\gotp}^n$. Fix $(a_2,\dots,a_n)\in \da_{\gotp}^{\,n-1}$ and consider $F_{a_2,\dots,a_n}(T):=F(T,a_2,\dots,a_n)\in K_{\gotp}[t]$. Since $\da_{\gotp}$ is infinite and a non-zero polynomial over the field $K_{\gotp}$ has only finitely many zeros, it follows that $F_{a_2,\dots,a_n}\equiv 0$. Writing $F(x_1,\dots,x_n)=\sum_{m=0}^d c_m(x_2,\dots,x_n)\,x_1^m$ with $c_m\in K_{\gotp}[x_2,\dots,x_n]$, we deduce $c_m(a_2,\dots,a_n)=0$ for all $m$ and all $(a_2,\dots,a_n)\in \da_{\gotp}^{\,n-1}$. By induction on $n$ this forces each $c_m$ to be the zero polynomial, hence $F=0$. Equivalently, the Zariski closure of $\da_{\gotp}^n$ in $\A^n_{K_{\gotp}}$ is all of $\A^n_{K_{\gotp}}$. Let
\[ C(f):=V\!\left(\frac{\partial f}{\partial x_1},\dots,\frac{\partial f}{\partial x_n}\right)\subset \A^n_{K_{\gotp}}\,. \]
This is the non-smooth locus of the morphism $f\colon\A^n_{K_{\gotp}}\rightarrow \A^1_{K_{\gotp}}$. By Chevalley theorem (see \cite[Section~29.22]{stacks-project}), the image $f(C(f))\subset \A^1_{K_{\gotp}}$ is constructible. By the hypothesis that the generic fiber is smooth, the generic point of $\A^1_{K_{\gotp}}$ does not lie in $f(C(f))$, hence the Zariski closure
\[ Z:=\overline{f(C(f))}\subset \A^1_{K_{\gotp}} \]
is a proper closed subset, and therefore finite. Set $W:=\A^1_{K_{\gotp}}-Z$. By construction, $f$ is smooth over $W$, i.e., $f^{-1}(W)\cap C(f)=\emptyset$. We claim that $f(\da_{\gotp}^n)$ is infinite. Otherwise $f(\da_{\gotp}^n)\subset\{a_1,\dots,a_m\}$ would be finite, hence $\da_{\gotp}^n \subset \bigcup_{j=1}^m f^{-1}(a_j)=V(f-a_1)\cup\cdots\cup V(f-a_m)$, a finite union of proper Zariski closed subsets of $\A^n_{K_{\gotp}}$, contradicting Zariski density.

Since $\A^1_{K_{\gotp}}- W=Z$ is finite while $f(\da_{\gotp}^n)$ is infinite, we have $W\cap f(\da_{\gotp}^n)\neq\emptyset$. Pick $a\in W\cap f(\da_{\gotp}^n)$ and choose $\alpha\in \da_{\gotp}^n$ with $f(\alpha)=a$. Since $a\in W$, we have $\alpha\notin C(f)$, hence $\alpha\notin Z_{\gotp}$. Therefore $a\in A_{\gotp}$, so
$W\cap f(\da_{\gotp}^n)\subset A_{\gotp}.$ This implies that
\[ B_{\gotp}=f(\da_{\gotp}^n)- A_{\gotp}\subset f(\da_{\gotp}^n)- W \subset Z, \]
and since $Z$ is finite, $B_{\gotp}$ is finite. In particular $A_{\gotp}\neq\emptyset$.

\smallskip\noindent
(3) \emph{The case $\nabla f\equiv 0$.} If $\partial f/\partial x_i\equiv 0$ for all $i$, then $Z_{\gotp}=\da_{\gotp}^n$, hence $A_{\gotp}=\emptyset$ and $B_{\gotp}=f(\da_{\gotp}^n)$. In characteristic $0$ this forces $f$ to be constant, while in characteristic $p>0$ it is the purely inseparable case. \end{proof}

\begin{prop}\label{p:dernot0} Let $f\in D[x]$ be a polynomial and let $X\subseteq\da$ be a nonempty openly Eulerian set. Assume that either $\deg f=0$ or $f'\neq0$. Then the
procounting measure $\mu_{f(X)}$ exists, of the form
\[ \mu_{f(X)}=\bigotimes_\gotp\mu_{f(X(\gotp))}. \]
Moreover, if the formal derivative $f'$ is non-zero, each local factor $\mu_{f(X(\gotp))}$ is a limit as in Proposition \ref{p:brd0}. \end{prop}

\begin{prf} By Lemma~\ref{l:muS-chs}, it is enough to work with the closure of $f(X)$. Since $X$ is openly Eulerian, we have $\xa=\prod_\gotp X(\gotp)$. By compactness, continuity, and \eqref{e:plynmprd},
\[ \widehat{f(X)}=f(\xa)=\prod_\gotp f(X(\gotp)). \]
Thus, by Proposition~\ref{p:prddstr2}, it is enough to prove that every local factor $f(X(\gotp))\subseteq\dap$ has a procounting measure.\\
If $\deg f=0$, say $f=c$, then $f(X)=\{c\}$, so $\mu_{f(X)}=\delta_c=\bigotimes_\gotp \delta_{c_\gotp}$, where $c_\gotp$ denotes the image of $c$ in $\dap$.

Assume now that $f'\neq0$. Fix a non-zero prime ideal $\gotp$. Since $f'$ is a non-zero polynomial, the set
\[ Z_\gotp:=\{\alpha\in X(\gotp):f'(\alpha)=0\} \]
is finite. As in Lemma~\ref{l:qe}, Hensel's lemma shows that $A_\gotp:=f\bigl(X(\gotp)- Z_\gotp\bigr)$ is open in $\dap$. Moreover $A_\gotp\neq\emptyset$, because $X(\gotp)$ is a nonempty open subset of $\dap$, hence infinite, whereas $Z_\gotp$ is finite. Finally,
\[ f(X(\gotp))- A_\gotp\subseteq f(Z_\gotp), \]
so $f(X(\gotp))$ is the union of a nonempty open set and a finite set. 

Since $X(\gotp)$ is compact and $f$ is continuous, $f(X(\gotp))$ is compact, hence closed. Also $A_\gotp$ is contained in its interior, so $\partial f(X(\gotp))\subseteq f(Z_\gotp)$.
Thus $\mu_{\dap}\bigl(\partial f(X(\gotp))\bigr)=0$. By Proposition~\ref{p:brd0}, the procounting measure $\mu_{f(X(\gotp))}$ exists and is obtained as in that Proposition.

Therefore all local factors exist, and Proposition~\ref{p:prddstr2} allows us to conclude. \end{prf}

Now we look at inverse images. Our treatment will be limited to what is needed for Conjecture \ref{cnj:pBHD}.

\begin{prop} \label{p:invimgelr} Given $f_1,\dots,f_k\in D[x]$, let $\bff\colon\da\rightarrow\da^k$ be the map $a\mapsto\big(f_1(a),\dots,f_k(a)\big)$. If $X\subseteq\da^k$ is openly Eulerian, then the distribution $\mu_{\bff^{-1}(\xa)}$ exists. \end{prop}

\begin{prf} By hypothesis we have $\xa=\prod_\gotp X(\gotp)$, which, by \eqref{e:plynmprd}, yields $\bff^{-1}(\xa)=\prod_\gotp\bff^{-1}\big(X(\gotp)\big)$, by \eqref{e:plynmprd}. Each factor of $\xa$ is compact open, so the same applies to $\bff^{-1}(\xa)$ and we conclude by Lemma \ref{l:mucmpap} and Proposition \ref{p:prddstr2}. \end{prf}

In general, we don't know if a (nonempty) polynomial inverse image of an openly Eulerian set is also openly Eulerian (actually, this was asked in \cite[Section 3.3]{dl} as question\footnote{In \cite{dl}, Eulerian sets are subsets of $D^n$. Here we allow subsets of $\da^n$, so the counterexample in \cite[Remark 6.7]{dl} concerns
a different definition.} {\bf (Q1)} and would have interesting consequences, see {\em loc.~cit.}) or at least close\footnote{In the sense of Definition \ref{d:clpr}.} to being such. On the other hand, it is easy to see that close pairs need not be preserved. The next example illustrates how wrong things can go.

\begin{eg} \label{eg:prdinfdlt} Consider $f(x)=2x^2$, in characteristic different from $2$.  The inverse image of $f\colon D\rightarrow D$ is just $Y=\{1,-1\}$, but  $f\colon\da\rightarrow\da$ yields $X=f^{-1}\big(\{2\}\big)=\prod_\gotp X(\gotp)=\prod_\gotp\{1,-1\}_\gotp$\,. Looking at the associated procounting measures, one finds $\mu_Y=\frac{1}{2}(\delta_1+\delta_{-1})$, while $\mu_X$ is better described by \eqref{e:tnsr2}, which leads, for $U=\prod_\gotp U_\gotp$ compact open in $\da$, to
\[ \mu_X(\b1_U)=\frac{|X_U\cap U_0|}{2^{|S(U)|}}\;, \]
where $S(U)=\big\{\gotp\mid U_\gotp\neq\dap\big\}$, $X_U=\prod_{\gotp\in S(U)}\{1,-1\}_\gotp$ and $U_0=\prod_{\gotp\in S(U)}U_\gotp$\,.

Now take $D=\Z$ and observe that $f^{-1}(\calp)=f^{-1}\big(\{2\}\big)=X$, while $f^{-1}\big(\za^*\big)=\emptyset$  (so that $\mu_{f^{-1}(\za^*)}=0$), in spite of the closeness of $\calp$ and $\za^*$. \end{eg}

For the application to the profinite Bateman--Horn conjecture, we shall not use a general preservation result for close pairs under inverse images. Instead, we prove directly the finite level estimate required by Definition~\ref{d:clpr}. The new input is a  non-concentration estimate for polynomial values in square classes.

For the rest of this section, we shall set $q_\gotp:=|D/\gotp|$.

\begin{lem}\label{l:square-class-nonconcentration} Assume that either $\operatorname{char}(F)\neq 2$, or that $\operatorname{char}(F)=2$ and the constant field of $F$ is $\F_q$, with $q>2$. Put $p_F:=2$ in the former case, while in the latter fix a prime divisor $p_F$ of $q-1$. Let $g\in D[x]$ be nonconstant, irreducible and separable. Then, for every $\rho\in(1/p_F,1)$, there exists a finite set of prime ideals $\Sigma_{g,\rho}$ such that, for every $\gotp\notin\Sigma_{g,\rho}$, every $a\ge1$, and every $c\in (D/\gotp^a)^*/((D/\gotp^a)^*)^{p_F}$, one has \[ \left| \left\{x\in D/\gotp^a:g(x)\in c\right\} \right| \le \rho\,|D/\gotp^a|. \] \end{lem} 

\begin{prf} After removing finitely many prime ideals, we may assume that $p_F\mid q_\gotp-1$, that the residue characteristic of $\gotp$ is coprime to $p_F$, and that the reduction $\bar g$ of $g$ modulo $\gotp$ is nonconstant and squarefree. In particular, $\bar g$ is not a $p_F$-th power in $(D/\gotp)[x]$. Let $\chi_\gotp$ be a character of order $p_F$ of $(D/\gotp)^*$, extended by $\chi_\gotp(0)=0$. For every $j=1,\ldots,p_F-1$, the character $\chi_\gotp^j$ also has order $p_F$. Hence, by the Weil bound for multiplicative character sums \cite[Theorem~11.23]{ik}, \[ \left| \sum_{x\in D/\gotp} \chi_\gotp^j\bigl(\bar g(x)\bigr) \right| \le (\deg g-1)q_\gotp^{1/2}. \] For $c\in(D/\gotp)^*/((D/\gotp)^*)^{p_F}$, character orthogonality therefore gives \[ \left| \left\{x\in D/\gotp:\bar g(x)\in c\right\} \right| \le \frac{1}{p_F}q_\gotp+ \frac{p_F-1}{p_F}(\deg g-1)q_\gotp^{1/2}. 
\] 
Enlarging  $\Sigma_{g,\rho}$ if necessary, this is at most $\rho q_\gotp$ for every $\gotp\notin\Sigma_{g,\rho}$. Now let $a\ge1$. Since the residue characteristic of $\gotp$ is coprime to $p_F$, the $p_F$-power map is an automorphism of $1+\gotp/\gotp^a$. Hence the reduction map induces an isomorphism 
\[ 
(D/\gotp^a)^*/((D/\gotp^a)^*)^{p_F} \simeq (D/\gotp)^*/((D/\gotp)^*)^{p_F}. \] 
Thus the $p_F$-power class of a unit modulo $\gotp^a$ is determined by its reduction modulo $\gotp$. Each class modulo $\gotp$ has exactly $q_\gotp^{a-1}$ lifts modulo $\gotp^a$, and so $\left| \left\{x\in D/\gotp^a:g(x)\in c\right\} \right| \le \rho q_\gotp^a = \rho\,|D/\gotp^a|$. \end{prf}

\begin{rmk} If $D^*$ is finite, the assumption that $\operatorname{char}(F)\neq2$ can be dropped. Indeed, in that case the exceptional pieces $t_\gotp D^*$ involve only
finitely many values modulo each $\gotn$. One therefore replaces the square-class estimate by the elementary estimate
\[ \left|\{x\in D/\gotq^a:\ g(x)\equiv c \bmod \gotq^a\}\right|\le\rho\, |D/\gotq^a| \]
after enlarging a finite set $\Sigma$ of prime ideals, for every $\gotq\notin\Sigma$, uniformly in $a$ and in the finitely many relevant values $c$. This follows from the fact that, modulo $\gotq$, a non-zero polynomial equation $g(x)=c$ has at most $\deg g$ solutions, after excluding finitely many primes. \end{rmk}

\begin{lem}\label{l:unit-values-product} Let $h\in D[x]$ be a non-zero product of finitely many separable polynomials. Then, for every $\lambda\in(0,1)$, there exists a finite set of prime ideals $\Sigma_{h,\lambda}$ such that, for every $\gotp\notin\Sigma_{h,\lambda}$ and every $a\ge1$, one has
\[ \left|\left\{x\in D/\gotp^a:h(x)\in (D/\gotp^a)^*\right\}\right|\ge\lambda\, |D/\gotp^a|. \]
\end{lem}

\begin{prf} After removing finitely many prime ideals, we assume that the reduction $\bar h$ of $h$ modulo $\gotp$ is non-zero. Let $r_\gotp$ be the number of roots of $\bar h$ in $D/\gotp$. Then $r_\gotp\le \deg h$. Since $h(x)$ is a unit modulo $\gotp^a$ if and only if its reduction modulo $\gotp$ is non-zero, we have
\[ \left|\left\{x\in D/\gotp^a:h(x)\in (D/\gotp^a)^*\right\}\right|=q_\gotp^{a-1}(q_\gotp-r_\gotp). \]
Hence $q_\gotp^{a-1}(q_\gotp-r_\gotp)\ge q_\gotp^{a-1}(q_\gotp-\deg h)$. After enlarging $\Sigma_{h,\lambda}$, we assume $q_\gotp-\deg h\ge\lambda q_\gotp$ for every $\gotp\notin\Sigma_{h,\lambda}$. Therefore $\left|\left\{x\in D/\gotp^a:h(x)\in(D/\gotp^a)^*\right\}\right|\ge\lambda q_\gotp^a=\lambda |D/\gotp^a|$. \end{prf}

Let $f_1,\ldots,f_k\in D[x]$ be irreducible separable polynomials, and let
\[ \bff\colon\da\rightarrow\da^k\qquad a\mapsto (f_1(a),\ldots,f_k(a)). \]

\begin{prop}\label{p:measpol} Assume that $D^*$ is finite if $F$ has constant field $\F_2$. Assume also that $\bff^{-1}\big((\da^*)^k\big)\neq\emptyset$. Then \[ \mu_{\bff^{-1}((\widehat{\calp(D)})^k)} = \mu_{\bff^{-1}((\da^*)^k)}. \]
\end{prop}

\begin{prf}
Set 
\[ X:=\bff^{-1}((\da^*)^k)\;\text{ and }\;Z:=\bff^{-1}((\widehat{\calp(D)})^k). \]
Since $\da^*\subseteq\widehat{\calp(D)}$, one has $X\subseteq Z$.

Let $\gotn$ be a non-zero ideal and set $G_\gotn:=(D/\gotn)^*$. We also set
\[ P_\gotn:=\hpi_\gotn(\widehat{\calp(D)})-G_\gotn, \]
so that
\[ \hpi_\gotn(\widehat{\calp(D)})=G_\gotn\sqcup P_\gotn. \]

We first identify the finite image of $X$. Let us set
\[ X_\gotn:= \{a\in D/\gotn:f_i(a)\in G_\gotn\text{ for every }i\}. \]
Then $\hpi_\gotn(X)=X_\gotn$. The inclusion $\hpi_\gotn(X)\subseteq X_\gotn$ is immediate. Conversely, take $a\in X_\gotn$, and choose $b\in X$. By the Chinese remainder theorem,
one can choose $\widetilde a\in\da$ such that $\widetilde a\equiv a\pmod{\gotp^{v_\gotp(\gotn)}}$ for every $\gotp\mid\gotn$, and $\widetilde a=b$ in $\dap$ for every
$\gotp\nmid\gotn$. Then every $f_i(\widetilde a)$ is a local unit, so $\widetilde a\in X$, and $\hpi_\gotn(\widetilde a)=a$.

Now, if $a\in\hpi_\gotn(Z)-\hpi_\gotn(X)$, then for every $i$ one has $f_i(a)\in G_\gotn\sqcup P_\gotn$, but for at least one $i$ one has $f_i(a)\notin G_\gotn$. Hence
\[ \hpi_\gotn(Z)-\hpi_\gotn(X)\subseteq\bigcup_{i=1}^k f_{i,\gotn}^{-1}(P_\gotn). \]
It is therefore enough to prove
\begin{equation} \label{e:reduction} \lim_{\gotn\to0}\frac{\sum_{i=1}^k|f_{i,\gotn}^{-1}(P_\gotn)|}{|X_\gotn|}=0. \end{equation}

Let $h:=f_1\cdots f_k$. Suppose first that either $\operatorname{char}(F)\neq2$, or that the exact constant field of $F$ has more than two elements. Let $p_F$ be as in Lemma~\ref{l:square-class-nonconcentration}, and choose constants
\[ \frac{1}{p_F}<\rho<\lambda<1. \]
For each $i=1,\ldots,k$, apply Lemma~\ref{l:square-class-nonconcentration} to $f_i$ with this value of $\rho$. After replacing the corresponding exceptional sets by their union, we obtain a finite set of prime ideals $\Sigma$ such that, for every $\gotp\notin\Sigma$, every $a\ge1$, every $i=1,\ldots,k$, and every $c\in(D/\gotp^a)^*/((D/\gotp^a)^*)^{p_F}$, one has
\[ \left|\left\{x\in D/\gotp^a:f_i(x)\in c\right\}\right|\le\rho|D/\gotp^a|. \]
In this case, set
\[ C_D:=|D^*/(D^*)^{p_F}| \]
which is finite by the $S$-unit theorem.

It remains to consider the case in which the exact constant field of $F$ is $\F_2$. Then $D^*$ is finite by hypothesis. Choose constants
\[ 0<\rho<\lambda<1. \]
After removing finitely many prime ideals, we may assume that the reduction of each $f_i$ has the same positive degree as $f_i$. Consequently, for every remaining prime $\gotp$, every $a\ge1$, and every $c\in D/\gotp^a$, one has
\[ \left|\left\{x\in D/\gotp^a:f_i(x)\equiv c\pmod{\gotp^a}\right\} \right|\le\deg(f_i)q_\gotp^{a-1}. \]
Indeed, there are at most $\deg(f_i)$ possible reductions modulo $\gotp$, and each of them has at most $q_\gotp^{a-1}$ lifts. Enlarging the exceptional set $\Sigma$, we may therefore assume that
\[ \left|\left\{x\in D/\gotp^a:f_i(x)\equiv c\pmod{\gotp^a}\right\}\right|\le\rho|D/\gotp^a| \]
for every $\gotp\notin\Sigma$, uniformly in $a$, $i$, and $c$. In this case, set
\[ C_D:=|D^*|. \]

In either case, apply Lemma~\ref{l:unit-values-product} to $h$ with the above value of $\lambda$. After enlarging $\Sigma$, for every $\gotp\notin\Sigma$ and every $a\ge1$, one has
\[ \left|\left\{x\in D/\gotp^a:h(x)\in(D/\gotp^a)^*\right\}\right|\ge\lambda|D/\gotp^a|. \]

We then set
\[ m(\gotn):=|\{\gotp\mid\gotn:\gotp\notin\Sigma\}|, \]
and it is immediate that $\lim_{\gotn\to0}m(\gotn)=\infty$.

Fix $i$. We now estimate $f_{i,\gotn}^{-1}(P_\gotn)$.

Let us recall the decomposition \eqref{e:clirr}
\[ \widehat{\calp(D)}=\da^*\sqcup\bigsqcup_{t_\gotp\in T}t_\gotp\wds, \]
from which only primes $\gotp\mid\gotn$ can contribute to $P_\gotn$. Hence
\[ P_\gotn\subseteq\bigcup_{\substack{t_\gotp\in T\\\gotp\mid\gotn}}\hpi_\gotn(t_\gotp\wds). \]

Fix such a $\gotp$. If $\gotq\mid\gotn$ and $\gotq\neq\gotp$, then $t_\gotp$ is a unit modulo $\gotq$. In the first case, as $u$ varies in $D^*$, the $p_F$-power-class vector of $t_\gotp u$ over the primes
\[ \gotq\mid\gotn,\qquad\gotq\notin\Sigma,\qquad \gotq\neq\gotp \]
takes at most $C_D$ possible values. In the second case, the corresponding residue-class vector takes at most $C_D=|D^*|$ possible values. For each of these values, the relevant estimate above gives a factor at most $\rho$ at every such prime $\gotq$. We impose no condition at $\gotp$ itself and no condition at the primes in $\Sigma$. Therefore, for $\gotn$ sufficiently small,
\[ |f_{i,\gotn}^{-1}(\hpi_\gotn(t_\gotp\wds))|\le C_D\,\rho^{m(\gotn)-1}|D/\gotn|. \]
Since there are at most $\omega(\gotn)$ principal prime ideals dividing $\gotn$, we obtain
\[  |f_{i,\gotn}^{-1}(P_\gotn)|\le C_D\,\omega(\gotn)\rho^{m(\gotn)-1}|D/\gotn|. \]

On the other hand, by the Chinese remainder theorem,
\[ |X_\gotn|=\prod_{\gotp\mid\gotn}q_\gotp^{a_\gotp-1}(q_\gotp-r_\gotp), \]
where $a_\gotp=v_\gotp(\gotn)$ and
\[ r_\gotp:=|\{x\in D/\gotp:h(x)\equiv0\pmod{\gotp}\}|. \]
For $\gotp\notin\Sigma$, Lemma~\ref{l:unit-values-product}, applied to $h$, gives $q_\gotp-r_\gotp\ge\lambda q_\gotp$. For $\gotp\in\Sigma$, we have $\delta_\gotp:=\frac{q_\gotp-r_\gotp}{q_\gotp}>0$, as $X\neq\emptyset$. Set $c_\Sigma:=\prod_{\gotp\in\Sigma}\delta_\gotp>0$. Then 
\[ |X_\gotn|=|D/\gotn|\prod_{\gotp\mid\gotn}\frac{q_\gotp-r_\gotp}{q_\gotp}\ge c_\Sigma\lambda^{m(\gotn)}|D/\gotn|. \]
Therefore
\[ \frac{|f_{i,\gotn}^{-1}(P_\gotn)|}{|X_\gotn|}\le\frac{C_D}{c_\Sigma}\omega(\gotn)\rho^{m(\gotn)-1}\lambda^{-m(\gotn)}=C'\omega(\gotn)\left(\frac{\rho}{\lambda}\right)^{m(\gotn)-1}. \]
Set $\theta:=\rho/\lambda<1$. The finitely many primes in $\Sigma$ contribute only a bounded amount to $\omega(\gotn)$, so
\[ \omega(\gotn)\theta^{m(\gotn)-1}\le\bigl(m(\gotn)+|\Sigma|\bigr)\theta^{m(\gotn)-1}\longrightarrow0. \]
Thus \eqref{e:reduction} follows. Summing over $i=1,\ldots,k$, we get
\[ \lim_{\gotn\to0}\frac{|\hpi_\gotn(Z)-\hpi_\gotn(X)|}{|\hpi_\gotn(X)|}=0. \]
Since $X\subseteq Z$, this implies
\[ \lim_{\gotn\to0}\frac{|\hpi_\gotn(Z)\triangle\hpi_\gotn(X)|}{|\hpi_\gotn(Z)|}=0. \]
Hence $X$ and $Z$ form a close pair. By Proposition~\ref{p:invimgelr}, the procounting measure $\mu_X=\mu_{\bff^{-1}((\da^*)^k)}$ exists. By Proposition~\ref{p:muST}, closeness implies that $\mu_Z$ exists and equals $\mu_X$. \end{prf}

\subsection{A profinite Bateman--Horn conjecture} We are now almost ready to state the main conjecture of this paper. Henceforth, we shall keep the assumption
\[ \text{char}(F)\neq2\;. \]

\begin{lem} \label{l:bunyakcnd} For $f\in D[x]$, the equality $f(\da)\cap\da^*=\emptyset$ is verified if and only if there is a prime ideal $\gotp$ such that $f(D)\subseteq\gotp$. \end{lem}

\begin{prf} The intersection is empty if and only if $f(\da)$ is contained in a finite union of prime ideals $\widehat\gotp_i$, because $\da^*$ is the complement of $\bigcup_\gotp\widehat\gotp$, each $\widehat\gotp$ is open and $f(\da)$ is compact. Nonetheless
$$f(\da)\subseteq\widehat\gotp_1\cup\dots\cup\widehat\gotp_n$$
is possible only if the image is entirely contained in a single prime ideal (if not, then for each $i$ there is $x_i\in\da$ such that $f(x_i)\notin\widehat\gotp_i$ and taking $x\equiv x_i\widehat\gotp_i\;\forall\,i$ one obtains a contradiction). Finally, $f(\da)\subseteq\widehat\gotp$ is equivalent to $f(D)\subseteq\gotp$, by density. \end{prf}

\begin{cnj}[Profinite Bateman--Horn] \label{cnj:pBHD} Let $f_1,\dots,f_k\in D[x]$ be non-associate separable irreducible polynomials and denote their product by $\bff$. Assume $f(\da)\cap\da^*\neq\emptyset$. Then the following limits exist and the equalities hold:
\begin{equation} \label{e:pBHD} \lim_{\gotn\to 0}\mu_{\bff|_D^{-1}(\calp(D)^k),\gotn}=\mu_{\bff^{-1}((\da^*)^k)}\end{equation}
\begin{equation} \label{e:pBHDd} \lim_{\gotn\to 0}\mu_{\bff(D)\cap\calp(D)^k,\gotn}=\mu_{\bff(\da)\cap(\da^*)^k}\,.\end{equation}
We call \eqref{e:pBHDd} the {\em direct image}  profinite Bateman--Horn conjecture. \end{cnj}

\begin{rmks} \label{r:lclcndBH} \begin{itemize} \item[]\end{itemize}
\noindent{\bf 1.}  By Lemma \ref{l:bunyakcnd}, the hypothesis $f(\da)\cap\da^*\neq\emptyset$ is simply a reformulation, in the adelic language, of the usual Bunyakovsky hypothesis on the absence of local obstructions. Indeed, in the case $D=\Z$, this amounts to say that there is no prime $p$ which divides $f(n)$ for all $n\in\N$, as in the usual formulation of the Bateman--Horn conjecture.\\
\noindent{\bf 2.} If $f(\da)\cap\da^*$  is empty then so is $\bff^{-1}((\da^*)^k)$, but $\bff^{-1}(\calp(D)^k)$ might still contain some points (even infinitely many,  as shown in Example \ref{eg:prdinfdlt}). If this happens then Conjecture \ref{cnj:pBHD} fails, because then $\mu_{\bff^{-1}((\da^*)^k)}$ is the $0$ distribution, while the limit on the left-hand side of \eqref{e:pBHD}, if it exists, is not.  \end{rmks}

Note that, contrarily to the classical case, we can have infinitely many primes which generates the same ideal. Therefore {\em a priori} there could be an infinite preimage that covers only a finite number of ideals. We check this is not the case.

\begin{lem}\label{e:inf} Let $F$ be a number field, let $D=\calo_{F,S}$ be a ring of $S$-integers, and let $f\in D[x]$ be a nonconstant polynomial with at least two distinct roots in $\overline F$. Fix $p\in \calp(D)$. Then the set
$$ \{(n,u)\in D\times D^*: f(n)=up\}$$
is finite.\end{lem}

\begin{prf} Let $S'$ be obtained from $S$ by adjoining the finitely many non-archimedean places corresponding to the prime ideals dividing $p$. Then $D\subseteq D':=\calo_{F,S'}$ and,
by construction, $p\in D'^*$. If $(n,u)\in D\times D^*$ satisfies $f(n)=up$, then, viewed in $D'$, one has $f(n)\in D'^*$. Therefore it is enough to prove that the set
$$\{\,n\in D': f(n)\in D'^*\,\}$$
is finite.

Let $\Sigma\subseteq \mathbb P^1(\overline F)$ be the set consisting of the zeros of $f$ together with the point at infinity. Since $f$ has at least two distinct roots, one has
$|\Sigma|\ge 3$. Now $n\in D'$ satisfies $f(n)\in D'^*$ if and only if $n$ is an $S'$-integral point on the affine curve $\mathbb P^1-\Sigma$ which is the projective line minus at least three points. Then Siegel's theorem \cite[Theorem~D.9.1]{HS} implies that it has only finitely many $S'$-integral points. \end{prf}

\begin{rmk} For $D=\F_q[t]$, one has $D^*=\F_q^*$, so the set $\{up:\ u\in D^*\}$ is finite. Now fix $u\in D^*$. The equation $f(n)=up$ is equivalent to $(f-up)(n)=0$. Since $f$ is nonconstant, the polynomial $f-up\in D[x]$ is non-zero, hence it has only finitely many roots in $D$. Thus the set
\[ \{(n,u)\in D\times D^*: f(n)=up\} \]
is finite in this case. We stress that the finiteness of $D^*$ is essential here. In fact, the analogous statement is false for general rings of $S$-integers in function fields. For instance, take $D=\F_q[t,t^{-1},(t-1)^{-1}]$ and choose an irreducible polynomial $\varpi\in\F_q[t]-\{t,t-1\}$. Then $\varpi$ is a prime element of $D$. Let $f(x)=\varpi x(x-1)\in D[x]$.
This polynomial is nonconstant and has two distinct roots. Indeed, via the Frobenius, $t^{q^m}-1=(t-1)^{q^m}$. Therefore $f(t^{q^m})=\varpi t^{q^m}(t^{q^m}-1)=\varpi t^{q^m}(t-1)^{q^m}$
with $t^{q^m}(t-1)^{q^m}\in D^*$. As the $t^{q^m}$'s are pairwise distinct,  this gives infinitely many solutions to $f(n)=u\varpi$ with $n\in D$ and $u\in D^*$. \end{rmk}

\subsubsection{Relation with Schinzel's hypothesis H} We now formulate a profinite version of Schinzel's hypothesis.

Let $f_1,...,f_d$ be a collection of non-associate separable irreducible polynomials over $D$  such that there is no prime $\gotp$ which divides $f_1(n)\cdots f_{d}(n)$ for all $n\in D$. Consider the map 
\[ \bff\colon D\rightarrow D^d\qquad n\mapsto (f_1(n),\dots,f_d(n))\ .\]
Then Schinzel's hypothesis states that $f_i(n)$'s are all prime for infinitely many $n$. Let $f$ denote the product of the $f_i$'s.

\begin{cnj}[Extended Schinzel's Hypothesis H]\label{conjH} If $f(\widehat{D})\cap \widehat{D}^*\neq \emptyset$ then the set $\bff(D)\cap\calp(D)^d$  is infinite. \end{cnj}

It is immediate to note that Conjecture \ref{conjH} is equivalent to the classical Schinzel's hypothesis.

\begin{prop} \label{p:pbhsh} Conjecture \ref{cnj:pBHD} implies Conjecture \ref{conjH}. \end{prop}

\begin{prf} Let $\bff\colon D\rightarrow D^d$, $a\mapsto (f_1(a),\dots,f_d(a))$, and set $f:=f_1\cdots f_d$. Assume $f(\da)\cap \da^*\neq\emptyset$. Equivalently, $X:=\bff^{-1}\big((\da^*)^d\big)\neq\emptyset$, because $f(x)$ is a unit if and only if all $f_i(x)$'s are units.

Suppose, by contradiction, that Conjecture \ref{conjH} fails. Then $\bff(D)\cap\calp(D)^d$ is finite. Since the fibers of $\bff\colon D\rightarrow D^d$ are finite, the set
$A:=\{a\in D:\bff(a)\in\calp(D)^d\}$ is finite.

We claim that the left-hand side of Conjecture \ref{cnj:pBHD} has finite support. If $A=\emptyset$, it is the zero distribution. Otherwise, for every compact open $U\subseteq\da$ and every sufficiently small $\gotn$, the set $U$ is $\gotn$-saturated and the map $A\rightarrow D/\gotn$ is injective. Hence $\widetilde\mu_{A,\gotn}(\b1_U)=\frac{|U\cap A|}{|A|}$. Thus
\[ \lim_{\gotn\to0}\widetilde\mu_{A,\gotn}=\frac1{|A|}\sum_{a\in A}\delta_a. \]
In both cases, by Lemma~\ref{l:fntspprt}, the limit distribution has finite support.

By Conjecture~\ref{cnj:pBHD}, this limit distribution equals $\mu_X$. Hence $\mu_X$ has finite support. On the other hand, $X$ is a nonempty compact open subset of $\da$, and Lemma \ref{l:mucmpap} gives $\mu_X=\frac{\b1_X}{\mu_{\da}(X)}\mu_{\da}$. Therefore $\supp(\mu_X)=X$. Since every nonempty compact open subset of $\da$ is infinite, this is a
contradiction. Hence $A$ is infinite. Since the fibers of $\bff\colon D\rightarrow D^d$ are finite, it follows that $\bff(D)\cap\calp(D)^d$ is infinite. \end{prf}

\subsection{Some examples}\label{s:da-eg} We now provide some specific examples, making our conjecture explicit in the following cases: \begin{itemize}
\item Dirichlet's theorem on arithmetic progressions, for $k=1$ and $f(x)=ax+b$, for $(a,b)=1$;
\item twin primes conjecture, for $k=2$, $f_1(x)=x$ and $f_2(x)=x+2$;
\item Landau's conjecture, for $k=1$, and $f(x)=x^2+1$. \end{itemize} 

\subsubsection{Dirichlet's theorem} \label{sss:btmnhrndg1} Let $f(x)=a+bx$, with $b\neq0$ and $(a,b)=1$. Set
\[ A:=f^{-1}(\calp)\qquad\text{and}\qquad X:=f^{-1}(\za^*). \]
The map $f\colon\za\rightarrow a+b\za$ is a homeomorphism onto its image. By Dirichlet's theorem, every admissible residue class modulo $n$ contains a
prime of the form $a+bm$. Hence
\begin{equation} \label{e:chsrinvim} \widehat A=f^{-1}(\cpr). \end{equation}
Moreover $A$ and $X$ form a close pair: the only residue classes in $\pi_n(A)$ which are not in $\pi_n(X)$ come from primes dividing $n$, and these are negligible compared with $|\pi_n(X)|$. Therefore, by Proposition \ref{p:muST},
\[ \mu_{f^{-1}(\calp)}=\mu_{f^{-1}(\za^*)}. \]
Let us now compute the local factors. Since $f^{-1}(\za^*)=\prod_p f^{-1}(\Z_p^*)$, we have
\[ f^{-1}(\Z_p^*)= \begin{cases} \Z_p, & p\mid b,\\[4pt]\Z_p-\big(-ab^{-1}+p\Z_p\big), & p\nmid b.\end{cases} \]
Thus
\[ \mu_{\Z_p}\big(\boldsymbol{1}_{f^{-1}(\Z_p^*)}\big)=\begin{cases} 1, & p\mid b,\\[4pt] \dfrac{p-1}{p}, & p\nmid b.\end{cases} \]
By Lemma \ref{l:mucmpap} and Proposition \ref{p:prddstr2}, we obtain
\begin{equation} \label{e:cfrdrc}\mu_{f^{-1}(\calp)}=\bigotimes_{p\mid b}\mu_{\Z_p}\otimes\bigotimes_{p\nmid b}\frac{p}{p-1}\,\mu_{\Z_p}\big|_{f^{-1}(\Z_p^*)}\ .\end{equation}

\subsubsection{Twin primes} Let $\bff=(f_1,f_2)$, with $f_1(x)=x$ and $f_2(x)=x+2$. The classical twin prime conjecture predicts
\[ \big|\{m\in\N\cap[2,x]: m\in\calp,\ m+2\in\calp\}\big|\approx2\prod_{p>2}\frac{p(p-2)}{(p-1)^2}\int_2^x\frac{dt}{(\log t)^2}. \]
In this case Conjecture \ref{cnj:pBH} reads
\[ \lim_{n\to0}\widetilde\mu_{\bff^{-1}(\calp^2),n}=\mu_{\bff^{-1}((\za^*)^2)}. \]
We now compute the right-hand side. We set
\[ X:=\bff^{-1}((\za^*)^2). \]
For $p=2$, the conditions $x\in\Z_2^*$ and $x+2\in\Z_2^*$ are equivalent. For $p>2$, set
\[ A_p:=\Z_p^*\cap(\Z_p^*-2)=\Z_p-\big(p\Z_p\cup(-2+p\Z_p)\big). \]
Thus
\begin{equation}\label{e:preimagetwinprimes} X=\Z_2^*\times\prod_{p>2}A_p. \end{equation}

For $p>2$, the set $A_p$ consists of all unit residue classes modulo $p$ except the class $-2$. Hence
\[ \mu_{\Z_p^*}(A_p)=\frac{p-2}{p-1}\;\text{ and }\;\mu_{\Z_p}(A_p)=\frac{p-2}{p}\ .\]
Therefore 
\begin{equation}\label{e:msrzr} \mu_{\za^*}(X)=\prod_{p>2}\left(1-\frac{1}{p-1}\right)=0. \end{equation}

Nevertheless, the procounting distribution of $X$ is obtained by normalizing each local factor. By Lemma \ref{l:mucmpap} and Proposition \ref{p:prddstr2},
\[ \mu_X=\mu_{\Z_2^*}\otimes\bigotimes_{p>2}\frac{p-1}{p-2}\,\mu_{\Z_p^*}\big|_{A_p}. \]

Finally, the usual twin-prime Euler factors are recovered from
\[ \frac{\mu_{\Z_2}(\boldsymbol{1}_{\Z_2^*})}{\mu_{\Z_2}(\boldsymbol{1}_{\Z_2^*})^2}=2\qquad\text{and}\qquad\frac{\mu_{\Z_p}(\boldsymbol{1}_{A_p})}{\mu_{\Z_p}(\boldsymbol{1}_{\Z_p^*})^2}
=\frac{p(p-2)}{(p-1)^2}\quad\text{for}\;\;p>2. \]

\begin{rmk} The same computations hold, {\em mutatis mutandis}, for the case $D=\F_q[t]$. In particular, the results are compatible with \cite[(1.4)]{sstwin}, which goes as follows. For $h\in\F_q[t]-\{0\}$, and for $q$ a power of a prime $p$ such that $q> 685090p^2$,
\[ \big|\big\{f\in\F_q[t] : q^{\deg_t(f)}=X,\;\text{with } f,\ f+h\in\calp(\F_q[t]). \big\}\big|\;\approx\; \prod_{P\in\calp(\F_q[t])} \dfrac{(1-q^{-\deg_t(P)}-q^{-\deg_t(P)}\mathbf{1}_{P\nmid h})}{(1-q^{-\deg_t(P)})^{2}}\cdot \dfrac{X}{\log_q^2 X} \]
as $X\to\infty$ through powers of $q$, where $\mathbf{1}_{P\nmid h}$ equals $1$ if $h$ is not divisible by $P$, and $0$ otherwise. We remark that the formula above reproduces exactly the same result showed in \cite[Theorem 1.1]{sstwin}. \end{rmk} 

\subsubsection{Landau's conjecture} Let $f(x)=x^2+1$. Landau's conjecture predicts that $f(n)$ is prime for infinitely many $n\in\N$. In our setting, Conjecture \ref{cnj:pBH} predicts
\[ \lim_{n\to0}\widetilde\mu_{f^{-1}(\calp),n}=\mu_{f^{-1}(\za^*)}. \]
We now compute the right-hand side. Clearly $f^{-1}(\za^*)=\prod_p f^{-1}(\Z_p^*)$. For $p=2$, one has
\[ x^2+1\in\Z_2^*\quad\text{if and only if}\quad x\in 2\Z_2. \]
For $p>2$, the condition is $x^2\not\equiv -1\bmod p$. Hence
\[ f^{-1}(\Z_p^*)=\begin{cases}\Z_p, & p\equiv 3 \bmod 4,\\[4pt]\Z_p-U_p, & p\equiv 1 \bmod 4,\end{cases} \]
where $U_p=(\alpha_p+p\Z_p)\sqcup(\beta_p+p\Z_p)$ and $\alpha_p,\beta_p\in\Z_p$ are the two roots of $x^2+1$ in $\Z_p$. Thus
\[ f^{-1}(\za^*)=2\Z_2\times\prod_{p\equiv 3 \bmod 4}\Z_p\times\prod_{p\equiv 1 \bmod 4}(\Z_p-U_p). \]
In particular,
\[ \mu_{\za}\big(f^{-1}(\za^*)\big)=\frac12\prod_{p\equiv 1 \bmod 4}\left(1-\frac2p\right)=0. \]

The procounting distribution is obtained by normalizing each local factor. By Lemma \ref{l:mucmpap} and Proposition \ref{p:prddstr2}, we get
\begin{equation}\label{landau} \mu_{f^{-1}(\za^*)}=2\,\mu_{\Z_2}\big|_{2\Z_2}\otimes\bigotimes_{p\equiv 3\bmod 4}\mu_{\Z_p}\otimes\bigotimes_{p\equiv 1\bmod 4}\frac{p}{p-2}\,\mu_{\Z_p}\big|_{\Z_p-U_p}. \end{equation}

Finally, we obtain
\[ \omega_f(p)=\begin{cases}1, & p=2,\\2, & p\equiv 1 \bmod  4,\\0, & p\equiv 3 \bmod 4.\end{cases} \]
Therefore the Bateman--Horn constant, in the notation of \eqref{e:cffBH}, is
\[ \prod_p\frac{1-\omega_f(p)/p}{1-1/p}. \]

\subsubsection{A numerical experiment up to $n=10!$} We have numerically tested Conjecture \ref{cnj:pBH} in the particular case of the twin prime conjecture for $D=\Z$. For $\bff(x)=(x,x+2)$, define
$$S:=\bff^{-1}((\za^{*})^2)$$
and consider $S_n=\{a\in \Z/n\Z : a, a+2 \in (\Z/n\Z)^{*}\}$. We say that $S_n$ is {\em covered by  primes} if for every $x\in S_n$, there is some prime $p$ such that $\pi_n(p)=x$ and $p+2$ is prime. Consider the function 
\[ F(n):=\min_{m\in\N}\big\{  S_n\;\text{is covered by primes in}\;[0,m] \big\}. \]
It is natural to wonder, if our Conjecture \ref{cnj:pBH} is true, how $F(n)$ varies for $n\to\infty$. This is ``morally'' equivalent to a comparison between Conjecture \ref{cnj:pBH} and the classical Bateman--Horn. We have verified this fact for $n=k!$, where $k\le10$. It may be interesting to note that $10^9$ is not significantly bigger than $10!$. The software we used is Python, and the list of twin primes was generated by the computer as well using a list of the first prime numbers up to one billion. This turned out to be necessary for $k=8$. 

We invite the interested reader to visit our \href{https://github.com/Francesco-M-Saettone/Profinite-distributions-and-the-Bateman-Horn-conjecture}{GitHub repository} for the Python code and for the lists of numbers it generates.

\section{A first comparison} \label{s:cmpr}

\subsection{Relative densities} \label{ss:dens} We write $2^S$ to denote the power set of a set $S$. 

Let $ Y\subseteq D^n$ and assume that $\mu_Y$ exists. A {\em relative density} $d_Y$ on $Y$ is the datum of two functions
\[ d^+_Y,\;d_Y^-\colon 2^{Y}\rightarrow [0,1] \] 
which satisfy the following properties:\begin{itemize} 
\item[(${\rm Dn}_Y$1)] $d_Y^-(Y)=1$\,;
\item[(${\rm Dn}_Y$2)] $d_Y^-(A)\le d_Y^+(A)$ for every $A\in2^Y$\,;
\item[(${\rm Dn}_Y$3)] $A\subseteq B\subseteq Y$ implies $d_Y^\circ(A)\le d_Y^\circ(B)$, with $\circ\in\{+,-\}$\,;
\item[(${\rm Dn}_Y$4)] if $Y$ is the disjoint union of $A$ and $B$, then $d_Y^+(A)+d_Y^-(B)=1$;
\item[(${\rm Dn}_Y$5)] if $A\cap B=\emptyset$, then
$$d_Y^-(A\cup B)\ge d_Y^-(A)+d_Y^-(B)\;\;\text{ and }\;\;d_Y^+(A\cup B)\le d_Y^+(A)+d_Y^+(B)\,;$$
\item[(${\rm Dn}_Y$6)] if $X$ is a coset of $\gotn D^n$, with $\gotn$ an ideal of $D$, then
$$d_Y^+(Y\cap X)=d_Y^-(Y\cap X)=\mu_Y\big(\xa\big).$$
\end{itemize}
The functions $d_Y^+,d_Y^-$ are called respectively {\em upper} and {\em lower relative density}. If $d_Y^-(A)=d_Y^+(A)$, we denote the common value by $d_Y(A)$ and call it the relative density of $A$ with respect to $Y$ (or, more precisely, with respect to $d_Y$).

\begin{rmk} The above adapts to the relative setting the definition of density on $D^n$ given in \cite[Section~4.2.1]{dl}. Being of set-theoretic nature, the first five properties are  unchanged; the last one combines (Dn6) and (Dn7) of {\em loc.~cit.}, which are arithmetic and have to be modified when restricting to a subset $Y$ having no $D$-module structure of its own. \end{rmk}

Just to show that the notion of relative density as defined above is not vacuous, we give the trivial example of the {\em relative Buck density} $\dsbk$, obtained from $\mu_Y$ by putting
$$\dsbkp(A):=\mu_Y(\widehat A)\,,\;\;\;\;\dsbkm(A):=1-\dsbkp(Y-A).$$
This mimics the classical definition of Buck density, for which we refer to \cite[Section~4.3.1]{dl} and the references wherein. Checking that $\dsbk$ satisfy the conditions (${\rm Dn}_Y$1)-(${\rm Dn}_Y$6) is an easy task, which we leave to readers.

\begin{eg}[Asymptotic relative densities] 
Let $F_v$ denote the completion of $F$ at a place $v$, and set
\[ F_\cals:=\prod_{v\in\cals}F_v. \]
The diagonal embedding makes $D$ a cocompact discrete subomodule of $F_\cals$.
We also fix a metric on $F_\cals^n$ compatible with its product topology, and denote by $B(0,r)$ the ball of radius $r\in\mathbb R_{>0}$. Then for $X\subseteq Y\subseteq D^n$ we set
\begin{equation} \label{e:asmptdns} d_Y^-(X):=\liminf_{r\to+\infty}\frac{|X\cap B(0,r)|}{|Y\cap B(0,r)|}\;, \qquad d_Y^+(X):=\limsup_{r\to+\infty}\frac{|X\cap B(0,r)|}{|Y\cap B(0,r)|}\;. \end{equation}
It is natural to use \eqref{e:asmptdns} to define a relative density on $Y$. The formal properties (${\rm Dn}_Y$1)-(${\rm Dn}_Y$5) are immediate; however, (${\rm Dn}_Y$6) is not automatic: if it holds, we say that \eqref{e:asmptdns} defines a relative asymptotic density on $Y$. \end{eg} 

\begin{eg}[Densities on $\calp$] \label{eg:dnstprm} In the scope of this paper, the most important example of asymptotic relative density is given by prime numbers. Here we take $D=\Z$, so that $F_\cals$ is just $\R$, with its usual metric, and $Y=\calp$. Then for $X=(a+b\Z)\cap\calp$, with $(a,b)=1$, the prime number theorem for arithmetic progressions gives
\[ d_\calp(X)=\lim_{r\to+\infty}\frac{|(a+b\Z)\cap\calp\cap B(0,r)|}{|\calp\cap B(0,r)|}=\frac{1}{\varphi(b)}=\frac{1}{|(\Z/b\Z)^*|}=\mu_\calp\big(\xa\big). \]
Thus (${\rm Dn}_Y$6) in this case is precisely the equidistribution of primes among the admissible residue classes modulo $b$.

We also mention another classical instance of relative densities on the primes, the {\em analytic density} $d_{\calp,{\rm an}}$ defined by
$$d_{\calp,{\rm an}}^-(X):=\liminf_{s\rightarrow1^+}\frac{\zeta_X(s)}{\zeta_\calp(s)}\;, \qquad d_{\calp,{\rm an}}^+(X):=\limsup_{s\rightarrow1^+}\frac{\zeta_X(s)}{\zeta_\calp(s)}\;,$$
where $\zeta_A(s):=\sum_{n\in A}n^{-s}$. The validity of (${\rm Dn}_Y$6) for $d_{\calp,{\rm an}}$ is well known: see e.g. \cite[VI, Theorem 2]{ser} \end{eg}

We now exploit the above axiomatization to show the analogue of \cite[Lemma 4.12]{dl}.

\begin{lem} \label{l:dsgdnsmsr} Let $Y\subseteq D^n$ be such that $\mu_Y$ exists. Then one has
$$d_Y^+(X)\le\mu_Y\big(\xa\big)$$
for every $X\subseteq Y$ and any relative density $d_Y$. \end{lem}

\begin{prf} For any non-zero ideal $\gotn$,  the set 
$$X_\gotn:=\hpi_\gotn^{-1}(\hpi_\gotn(X))=\bigsqcup_{\bar a\in\hpi_\gotn(X)}\hpi_\gotn^{-1}(\bar a)$$
is a finite disjoint union of cosets of $\gotn\da^n$, so (${\rm Dn}_Y$6) and (${\rm Dn}_Y$5) together yield
\[ d_Y(X_\gotn\cap Y)=\mu_Y\big(\xa_\gotn\cap\ya\big). \]
Moreover, Lemma \ref{l:clsr} yields $\xa=\bigcap_\gotn X_\gotn=\bigcap_\gotn(X_\gotn\cap\ya)$ (because $\xa\subseteq\ya$). It follows
\[ \mu_Y\big(\xa\big)=\inf_\gotn\mu_Y\big(X_\gotn\cap\ya\big)=\inf_\gotn d_Y(X_\gotn\cap Y)\ . \]
Since $X\subseteq X_\gotn\cap Y$ for every $\gotn$, monotonicity gives $d_Y^+(X)\le d_Y(X_\gotn\cap Y)$ for every $\gotn$. \end{prf}

As a consequence, if $S=o(T)$, then $d_T(S)=0$ for any relative density on $T$, by Lemma \ref{l:SoTmsr0}.

\begin{eg} \label{eg:zrdnst} Let $S=\{p\in \calp: p+2\in\calp\}$. Brun's celebrated theorem shows $\zeta_S(1)<\infty$, hence $d_{\calp,{\rm an}}(S)=0$.  Lemma \ref{l:dsgdnsmsr} implies that the same holds for any relative density on $\calp$, because $\mu_\calp(\sa)=0$. In order to check this last claim, note that $S=\bff^{-1}(\calp^2)$, with $\bff(x)=(x,x+2)$, yields
$$\sa\subseteq\bff^{-1}\big(\cpr^2\big)=\underbrace{\bff^{-1}\big(\calp^2\big)\cup\bff^{-1}\big(\calp\times\za^*\big)\cup\bff^{-1}\big(\za^*\times\calp\big)}_A\cup\,\bff^{-1}\big((\za^*)^2\big).$$
Recalling $\mu_\calp=\mu_{\za^*}$ (Example~\ref{eg:PZ}), one finds
$$\mu_\calp\big(\sa\big)=\mu_{\za^*}\big(\sa\big)\le\mu_{\za^*}\big(\bff^{-1}(\cpr^2)\big)\le\mu_{\za^*}(A)+\mu_{\za^*}\big(\bff^{-1}((\za^*)^2)\big)=0$$
with the last equality coming from \eqref{e:msrzr} and $\sigma$-additivity ($A$ is countable because it is a subset of $\Z$, by the definition of $\bff$).

A similar argument works also, {\em mutatis mutandis}, for the Landau's primes. \end{eg}

\begin{rmk} The inequality $d_\calp^+(X\cap\calp)\le\mu_{\za^*}(\xa\cap\za^*)$ was stated (without proof) in \cite[Section~3.3.2]{lms}, with the claim it leads to a better understanding of \cite[Section~III]{gol}; we now complete that discussion. In {\em loc.~cit.}, Golomb considers the asymptotic relative density $d_\calp$ (called the ``$\pi$-measure'' in his paper\footnote{Comparing \cite[top of page 184]{gol} with our definition in Example \ref{eg:dnstprm}, one finds $\pi(A)=d_\calp(A\cap\calp)$ for any $A\subseteq\Z$.}) and asks for relations between $d_\calp(X)=0$ and the fact that $X\subseteq\calp$ is nowhere dense with respect to a certain topology.\footnote{The so-called Golomb topology on the positive integers: its restriction to $\calp$ coincides with the one induced by the profinite topology of $\za$, as explained in \cite[Section~3.3.2]{lms}. Caveat: in \cite{gol}, the author defines his topology only on the positive integers, but in Section~III he considers the closure of an arbitrary $S\subseteq\Z$.} Actually, he finds that the two properties are independent of each other (\cite[Theorem 12]{gol}): this is not surprising, since our Lemma \ref{l:dsgdnsmsr} shows that the density is bounded by $\mu_{\za^*}(\xa)$ and a set can have both positive measure and empty interior. Also, \cite[Theorem 13]{gol} proves that if there are arbitrarily large $a_i\in\N$ such that for each $i$ all but finitely many elements of $X\subseteq\calp$ are in a fixed coset of $a_i\Z$, then $X$ has zero relative density (indeed, the hypothesis implies $\mu_{\za^*}(\xa)\le\varphi(a_i)^{-1}$ for all $i$, with $\lim_i\varphi(a_i)=+\infty$) and is nowhere dense (obvious, since $\xa$ contains no open coset of $\za$). An interesting example of such $X$ is the set of primes dividing the Fermat numbers, $\{p\in\calp:\exists\,n\in\N,\,p|(2^{2^n}+1)\}$ (\cite[Theorem 14]{gol}). \end{rmk}

The following results are the relative versions of \cite[Propositions 6.4 and 6.18]{dl}.

\begin{prop} \label{p:dl645} Let $Y\subseteq D^n$ be openly Eulerian and $d_Y$ a relative density on $Y$. Let $X\subseteq Y$ satisfy $\xa\subseteq C=\prod_\gotp C_\gotp$, where each $C_\gotp\subseteq Y(\gotp)$ is open. If $d_Y^+(X)\ge\mu_Y(C)>0$, then $\xa=C$. \end{prop}

\begin{skprf} The hypothesis that $Y$ is openly Eulerian ensures the existence of $\mu_Y$ (Corollary~\ref{c:3.4}) and the equality $\mu_Y(C)=\prod_\gotp\mu_{Y(\gotp)}(C_\gotp)$. For proving the claim, one proceeds by contradiction: $\xa\neq C$ implies $\xa\subseteq C-U$ for some open $U\subset\da^n$ such that $U\cap C\neq\emptyset$, leading to the sequence of inequalities
$$\mu_Y(\xa)\le\mu_Y(C-U)\stackrel{\dagger}<\mu_Y(C)\le d_Y^+(X)\,,$$
impossible by Lemma \ref{l:dsgdnsmsr}. The key passage $\dagger$ is proved reasoning exactly as in \cite[Proposition~6.4]{dl} (with the trivial modification of replacing $\mu_{\da^n}$ with $\mu_Y$), to which we refer for details. \end{skprf}

\begin{prop} \label{p:pnnstll} Let $d_Y$ be a relative density on $Y\subset D^n$. Suppose that $X$ is a subset of $Y$ such that
\begin{equation} \label{e:cndzps} \lim_{\gota\to 0} d_Y^+(X_\gota-X)=0\,, \end{equation}
where $\gota$ varies among non-zero ideals of $D$ and $X_\gota:=Y\cap\hpi^{-1}_\gota\big(\widehat{\pi}_\gota(X)\big)$. Then $d_Y(X)$ exists and satisfies
$$\mu_Y\big(\xa\big)=d_Y(X).$$
\end{prop}

\begin{prf} As before, one follows the same reasoning as in the corresponding statement of \cite{dl}. 

To start with, $(\text{${\rm Dn}_Y$}5)$ implies that, if the relative density $d_Y(A_i)$ exists for any $A_i$ in a finite collection of pairwise disjoint subsets of $Y$, then the same is  true for their union, i.e., $d_Y^+(\cup_iA_i)=d_Y^-(\cup_iA_i)$; moreover, this common value is the sum of the $d_Y(A_i)$'s. These facts, together with (${\rm Dn}_Y$6), show $d_Y(X_\gota)=\mu_Y(X_\gota)$ for all $\gota$. Now note that the analogue of \cite[Lemma~4.5]{dl} still holds: one has
\begin{equation} \label{e:lmm45dl} d_Y^+(W-X)=d_Y(W)-d_Y^-(X)  \end{equation}
for all $X\subseteq W\subseteq Y$ such that $d_Y(W)$ exists. The proof of \eqref{e:lmm45dl} uses only $(\text{${\rm Dn}_Y$}4)$ and $(\text{${\rm Dn}_Y$}5)$, so is exactly the same as in {\em loc.~cit.}. Applying \eqref{e:lmm45dl} to $W=X_\gota$, it follows
$$d^+_Y(X_\gota-X)=d_Y(X_\gota)-d_Y^-(X)=\mu_Y\big(X_\gota\big)-d_Y^-(X)$$
and hence
$$\lim_{\gota \to 0}d_Y^+(X_\gota-X)=\lim_{\gota\to\infty}\big(\mu_Y(X_\gota)-d_Y^-(X) \big)=\mu_Y\big(\xa\big)-d_Y^-(X)\,,$$
using $\xa=\ya\cap\bigcap_\gota X_\gota$ (by Lemma \ref{l:clsr}) and $\mu_Y(C)=\mu_Y(C\cap\ya)$ for any closed $C\subseteq\da^n$ (since $\mu_Y$ has support contained in $\ya$). Therefore \eqref{e:cndzps} is equivalent to $d^-_Y(X)=\mu_Y\big(\xa\big)$ and one concludes by Lemma \ref{l:dsgdnsmsr}. \end{prf}

\subsection{Comparisons} \label{ss:cmpr}

Let $\mu_\cals$ be the normalized Haar measure on $F_\cals$, i.e., the Lebesgue measure for the archimedean primes and the Haar measure for the non-archimedean ones (see \cite[Section~4.3.2]{dl} for more details). 

Fix a metric on $F_\cals^n$ (compatible with its natural topology) and consider the ball  $B(0,r)$ of radius $r$ for $r\in\R$. Let $X\subset D^n$, and consider
\[ \mu_X^\cals:=\sum_{x\in X}\delta_x\in\cald(F_\cals^n,\R). \]
For $X\subset D^n$ we can thus define
\begin{equation} \label{e:counting} c_X(r):=\mu^\cals_X\big(\b1_{B(0,r)}\big)=\vert B(0,r)\cap X \vert. \end{equation}

We say that $Y$ admits an {\em asymptotic density function} $f_Y$ on $F_\cals$ if
\[ c_Y(r)\approx \int_{B(0,r)}f_Y\mu_\cals \]
in $\cald(F_\cals,\R)$, for $r\to+\infty$.

\begin{rmk} Assume that $X\subseteq Y$, that $\xa$ is compact open, and that $Y$ admits an asymptotic density function $f_Y$. Under the hypotheses of Proposition \ref{p:pnnstll}, we have $d_Y(X)=\mu_Y(\b1_{\xa})$. Therefore the comparison with the asymptotic density function of $Y$ gives
\[ \mu_X^\cals\big(\b1_{B(0,r)}\big)\approx d_Y(X)\int_{B(0,r)} f_Y\,d\mu_\cals=\mu_Y(\b1_{\xa})\int_{B(0,r)} f_Y\,d\mu_\cals. \]

For example, let $D=\Z$, $S=\{\infty\}$, and $X=\calp\cap(a+b\Z)$, with $(a,b)=1$. By (${\rm Dn}_Y$6) and Lemma~\ref{l:mucmpap}, the congruence condition $x\equiv a \bmod b$
contributes the  local coefficient
\[ d_\calp(X)=\mu_\calp(\b1_{\xa})=\frac{1}{|(\Z/b\Z)^*|}=\frac{1}{\varphi(b)}. \]
Since the prime number theorem gives the asymptotic density function $f_\calp(t)=\dfrac{1}{2\log t}$, we obtain
\[ \mu_X^S(\b1_{B(0,r)})\approx\frac{1}{\varphi(b)}\frac{1}{2\log t} \cdot r\ .\]
\end{rmk}

\subsubsection{Back to Bateman--Horn} Let $\bff=(f_1,\dots,f_k)\in (D[x])^k$ be a vector of monic, separable and irreducible polynomials and denote by $\bff_\gotn$ its reduction modulo an ideal $\gotn$. For $a\in(D/\gotn)^k$, we set 
$$\calp_a(D):=\{p\in\calp(D)^k:  \pi_\gotn(p)=a\}.$$
We thus have
\begin{equation} \bff^{-1}(\calp(D)^k)=\bigsqcup_{a\in (D/\gotn)^k}\bff^{-1}(\calp_a(D))=Z\cup\bigsqcup_{b\in W}\bff^{-1}(\calp_{\bff_\gotn(b)}(D)), \end{equation}
where $W:=\bff_\gotn^{-1}\big(((D/\gotn)^*)^k\big)$ and
\[ Z:=\big\{ z\in \bff^{-1}(\calp(D)^k):\pi_\gotn(f_i(z))\notin(D/\gotn)^*\text{ for some }i\big\}. \]
If $D^*$ is finite then so is the set $Z\cap D$, because only finitely many $p\in\calp(D)$ reduce to a non-unit of $D/\gotn$ and each fiber $f_i^{-1}(p)$ is finite since $f_i$ is non-constant. As a consequence, if we assume that $D\cap\bff^{-1}(\calp(D)^k)$ is infinite, then most of it must be distributed among the sets $\bff^{-1}(\calp_{\bff_\gotn(b)}(D))$, with $b$ varying in $W$.

\begin{rmk} We point out how the finite local factors appearing in Section~\ref{s:da-eg} are reflected in the asymptotic comparison. Since
\[ \mu_\calp=\mu_{\za^*}=\bigotimes_p \frac{p}{p-1}\,\mu_{\Z_p}\big|_{\Z_p^*}, \]
we can compare this expansion with \eqref{e:cfrdrc}. For the linear polynomial $f(x)=a+bx$, with $(a,b)=1$, the local coefficients agree for $p\nmid b$, while for $p\mid b$ the factor $p/(p-1)$ is missing from $\mu_{f^{-1}(\calp)}$. Hence the finite local correction is
\[ \prod_{p\mid b}\left(1-\frac1p\right)^{-1}=\frac{b}{\varphi(b)}=C(f), \]
which is indeed the Bateman--Horn constant for $f$. Thus the coefficient predicted by the profinite local factors matches the classical density
coefficient in the linear case. \end{rmk}

We now recall the Bateman--Horn conjecture over $\F_q[t]$ for a single polynomial, which was already stated in \cite[Conjecture 1.1]{sslandau} and in \cite[Conjecture 6.2]{ccg}.

In the next two conjectures we specialize to $D=\F_q[t]$.  We write
\[ \calp_q:=\{P\in\F_q[t]: P\text{ monic and irreducible}\}. \]
We also set $|P|:=q^{\deg_t P}$.  All products $\prod_P$ below are taken over $P\in\calp_q$.

\begin{cnj} \label{cnj:BHfinite}  Let $f(x)\in\F_q[t][x]$ be a non-constant irreducible separable monic in $x$ polynomial. Assume $|\{a\in\F_q[t]/(P):f(a)=0\}|<q^{\deg(P)}$ for every $P\in\calp_q$. Then we have
\[ \big\vert\{g\in\F_q[t] : |g|=X,\; g\;\text{monic},\;f(g)\in\calp(\F_q[t]) \}\big\vert\approx \frac{C_q(f)}{\deg_x(f)}\cdot\dfrac{X}{\log_q(X)}  \]
as $X\to\infty$ through powers of $q$, and 
\[ C_q(f):=\prod_P \frac{1-|P|^{-1}\, \big|\{\alpha\in\F_q[t]/(P): f(\alpha)\equiv0 \bmod P\}\big|}{1-|P|^{-1}}. \]
\end{cnj}

The importance of being separable is discussed in \cite{ccg}, where the authors also propose a conjecture for inseparable polynomials.\\

We might as well propose a more general formulation of Conjecture \ref{cnj:BHfinite} as follows:

\begin{cnj}\label{cnj:BHmult} Let $\bff(x)=(f_1(x),\dots, f_k(x))\in(\F_q[t][x])^k$ be a $k$-tuple of non-constant irreducible separable monic in $x$ pairwise distinct polynomials, and $f:=\prod_{i=1}^kf_i$. Assume $|\{a\in\F_q[t]/(P):f(a)=0\}|<|P|$ for every $P\in\calp_q$.  Then we have
\[ \big\vert\{g\in \F_q[t] : |g|=X,\; g\;\text{monic},\;f_i(g)\in\calp(\F_q[t])\text{ }\forall i\le k \}\big\vert\approx \dfrac{C_q(\bff)}{\prod_{i=1}^{k}\deg_x(f_i)}\cdot\dfrac{X}{(\log_q(X))^k} \]
as $X\to\infty$ through powers of $q$, and
\[  C_q(\bff):=\prod_P\frac{1-|P|^{-1}\,\big|\{\alpha\in\F_q[t]/(P): f(\alpha)\equiv0 \bmod P\}\big|}{(1-|P|^{-1})^k}. \]
\end{cnj}

In the following the counting functions $c_X$ are defined as in \eqref{e:counting}, with respect to the metric induced by the Euclidean absolute value (if $D=\Z$) or by the degree in $t$ (if $D=\F_q[t]$).

\begin{prop} \label{p:4.11} Let $D$ be either $\Z$ or $\F_q[t]$ and let $\bff\in (D[x])^k$ be a tuple of irreducible and separable polynomials. Suppose that one has 
\begin{equation} \label{e:thm4.11} c_{\bff^{-1}\big(\calp(D)^k\big)\cap(b+\gotn D)}\approx\frac{1}{|\bff_{\gotn}^{-1}(((D/\gotn)^*)^k)|}\,c_{\bff^{-1}(\calp(D)^k)\cap D} \end{equation}
for every $\gotn$ and for every $b\in \bff_\gotn^{-1}(((D/\gotn)^*)^k)$. Then either Conjecture \ref{cnj:BH} or Conjecture \ref{cnj:BHmult} implies Conjecture \ref{cnj:pBHD}, both with respect to $\bff$, for $D=\Z$ or $D=\F_q[t]$ respectively.  \end{prop}

\begin{prf} Let us denote $Y:=D\cap \bff^{-1}(\calp(D)^k)$. Let $x\in \bff^{-1}((\da^*)^k)$, and let $U=\hpi_\gotn^{-1}(b)$ be a basic open neighbourhood of $x$. Since $\bff(x)\in (\da^*)^k$, we have $b\in \bff_\gotn^{-1}\big(((D/\gotn)^*)^k\big)$. By the Bateman--Horn conjecture (either \ref{cnj:BH} for $\Z$ or \ref{cnj:BHmult} for $\F_q[t]$), $c_Y$ is unbounded. Hence the hypothesis $\eqref{e:thm4.11}$ implies $Y\cap(b+\gotn D)\neq\emptyset$. Thus $U\cap Y\neq\emptyset$, and therefore $x\in\ya$. This proves
\[ \bff^{-1}((\da^*)^k)\subseteq \ya. \]
On the other hand, by continuity of $\bff$, we have $\ya\subseteq \bff^{-1}\big(\widehat{\calp(D)}^k\big)$. Hence we have a chain of inclusions
\[ \bff^{-1}((\da^*)^k)\subseteq\ya\subseteq\bff^{-1}\big(\widehat{\calp(D)}^k\big). \]
By the proof of Proposition~\ref{p:measpol}, the two outer sets form a close pair. Therefore, by Lemma~\ref{l:inclprssm}, also $\bff^{-1}((\da^*)^k)$ and
$\ya$ form a close pair. Since the procounting measure of $\bff^{-1}((\da^*)^k)$ exists by Proposition~\ref{p:invimgelr}, Proposition~\ref{p:muST} yields
\[ \mu_{\bff^{-1}((\da^*)^k)}=\mu_{\ya}=\mu_Y \]
where the second equality holds by Lemma~\ref{l:muS-chs}. Therefore $\mu_{\bff^{-1}(\calp(D)^k)}=\mu_{\bff^{-1}((\da^*)^k)}$, which is Conjecture~\ref{cnj:pBHD} for $\bff$. \end{prf}

\begin{rmk} Having a reasonable analogue of Conjectures \ref{cnj:BH} and \ref{cnj:BHmult} for rings of $S$-integers in a global field, one could immediately extend Proposition~\ref{p:4.11} accordingly.  \end{rmk} 

\begin{rmk} \label{r:den} Let $Y:=D\cap \bff^{-1}(\calp(D)^k)$. For $X\subseteq Y$, whenever the limit exists, we define its relative density by
\[ d_Y(X):=\lim_{r\to+\infty}\frac{c_{X\cap Y}(r)}{c_Y(r)}. \]
With this definition, the formal properties (${\rm Dn}_Y$1),...,(${\rm Dn}_Y$5) are purely set-theoretic and follow straightforwardly. We have to check the arithmetic axiom (${\rm Dn}_Y$6). In the present situation, this is precisely hypothesis \eqref{e:thm4.11}: for every non-zero ideal $\gotn$ and every $b\in \bff_\gotn^{-1}\big(((D/\gotn)^*)^k\big)$, one has
\[ c_{Y\cap(b+\gotn D)}\approx\frac{1}{|\bff_\gotn^{-1}(((D/\gotn)^*)^k)|}\,c_Y. \]
Equivalently, the relative density $d_Y(b+\gotn D)$ exists for every admissible class $b+\gotn D$, and satisfies
\[ d_Y(b+\gotn D)=\frac{1}{|\bff_\gotn^{-1}((D/\gotn)^*)^k|}.\]
Thus \eqref{e:thm4.11} consists of (${\rm Dn}_Y$6) for the relative density $d_Y$. \end{rmk}

\begin{cor} \label{dirichlet} Conjecture \ref{cnj:BH} implies Conjecture \ref{cnj:pBHD} for $D=\Z$, $k=1$ and $f(x)=ax+b$ with $a$ and $b$ coprime. \end{cor}

\begin{prf} As hypothesis \eqref{e:thm4.11} holds for $f(x)=ax+b$ with $(a,b)=1$ (see, for instance, \cite[Theorem 6.1]{1bh} and the references therein), by Proposition \ref{p:4.11} we conclude. \end{prf}

In \cite[Theorem~1.2]{sslandau}, Sawin and Shusterman established the quadratic case of Conjecture \ref{cnj:BHfinite}  for $q$ large enough, where $q$ is a power of some prime $p$. We thus immediately obtain the following Corollary. 

\begin{cor}\label{c:ss} Let $q$ be a power of an odd prime $p$. Assuming \eqref{e:thm4.11} is satisfied, Conjecture \ref{cnj:pBHD} holds for $D=\F_q[t]$ in the following cases:\begin{enumerate}
\item $\bff=(f)$ if $q>2^{10}3^2e^2p^4$ and $f(x)=x^2+1$ is irreducible over $\F_q$; 
\item $\bff=(x,x+h)$ for any non-zero $h\in\F_q[t]$, if $q>685090p^2$. \end{enumerate} \end{cor}

\begin{prf} The two assertions follow respectively from \cite[Theorem~1.2]{sslandau} and \cite[Theorem~1.1]{sstwin}, together with Proposition~\ref{p:4.11}. \end{prf}

\subsubsection{From profinite to classical?}\label{sss:nonequiv} We explain why the profinite distributional datum is weaker than the classical Bateman--Horn asymptotic. This is already visible for $D=\Z$ and $f(x)=x$. 

Let $H\colon\R_{\ge0}\rightarrow\R_{\ge0}$ be increasing, with $H(x)\to\infty$. We construct $S\subset\Z_{>0}$, disjoint from $\calp\cup\{\pm1\}$, such that
\[ c_S(x)\le H(x) \]
for all sufficiently large $x$, but $\mu_S=\mu_{\za^*}=\mu_\calp$. Enumerate all pairs $(n,a)$, with $n\ge1$ and  $a\in(\Z/n\Z)^*$ as $(n_j,a_j)_{j\ge1}$. We choose pairwise coprime positive non-prime integers $s_j$, inductively, such that $s_j\equiv a_j\bmod {n_j}$. Fix $s_1,\dots,s_{j-1}$. Let $P_j$ be the set of prime divisors of $s_1\cdots s_{j-1}$, and define
\[ M_j:=\prod_{\substack{p\in P_j\\p\nmid n_j}}p. \]
By the Chinese remainder theorem, we may choose $u_j,v_j>1$, arbitrarily large, such that
\[ u_j\equiv a_j\bmod {n_j} \qquad  u_j\equiv1\bmod {M_j} \qquad v_j\equiv1\bmod {n_jM_j}\ . \]
Then $s_j:=u_jv_j$ is non-prime satisfies $s_j\equiv a_j\bmod {n_j}$, and is coprime to the previous $s_i$'s.  Choosing $u_j,v_j$ sufficiently large, we may also assume that $s_j$ is strictly increasing and that $H(x)\ge j$ for  $x\ge s_j$. We set $S:=\{s_j:j\ge1\}$. Then $S\cap(\calp\cup\{\pm1\})=\emptyset$, and, if $s_j\le x<s_{j+1}$, we have
\[ c_S(x)=j\le H(x). \]
We now show that $S$ and $\za^*$ form a close pair. By construction, $S$ meets every unit class modulo every $n$. Hence $(\Z/n\Z)^*\subseteq \pi_n(S)$. If $s_j$ is not a unit modulo $n$, then some prime divisor of $n$ divides $s_j$. Since the $s_j$'s are pairwise coprime, each prime divisor of $n$ can occur for at most one $s_j$. Therefore
\[ \left|\pi_n(S)-(\Z/n\Z)^*\right|\le \omega(n). \]
Thus
\[ \frac{\left|\pi_n(S)\triangle(\Z/n\Z)^*\right|}{\left|\pi_n(S)\cup(\Z/n\Z)^*\right|} \le \frac{\omega(n)}{\varphi(n)}\ . \]
Indeed, the right-hand side tends to $0$ as $n\to0$. Hence $S$ and $\za^*$ form a close pair so by Proposition~\ref{p:muST} we have $\mu_S=\mu_{\za^*}$. Since $\mu_\calp=\mu_{\za^*}$ for $D=\Z$, we get $\mu_S=\mu_{\za^*}=\mu_\calp$.

Taking, for instance, $H(x)=\log\log x$ for $x$ sufficiently large, we obtain a set whose counting function is much smaller than the prime number theorem asymptotic, but whose procounting measure is the same as that of the primes.

\end{document}